\documentclass[reqno]{amsart}
\usepackage[T1]{fontenc}
\usepackage{cases,xcolor,hyperref}
\usepackage{enumerate}
\usepackage[margin=3.25 cm]{geometry}

\newcommand{\e}{\varepsilon}
\newcommand{\R}{\mathbb{R}}
\newcommand{\I}{\mathbb{I}_d}
\newcommand{\Rd}{\R^d}
\newcommand{\Sd}{\mathbb{S}^{d-1}}
\newcommand{\A}{\mathcal{A}}
\newcommand{\Ld}{\mathcal{L}^d}
\newcommand{\Hd}{\mathcal{H}^{d-1}}
\newcommand{\N}{\mathbb{N}}

\newcommand{\mres}{\mathbin{\vrule height 1.6ex depth 0pt width
0.13ex\vrule height 0.13ex depth 0pt width 1.3ex}}

\newtheorem{theorem}{Theorem}[section]
\newtheorem{corollary}[theorem]{Corollary}
\newtheorem{lemma}[theorem]{Lemma}
\newtheorem{proposition}[theorem]{Proposition}

\newtheorem{remark}[theorem]{Remark}

\begin{document}
\title[Singular perturbations for anisotropic higher-order materials] {Singular perturbations models in phase transitions for anisotropic higher-order materials}
\author [Giuseppe Cosma Brusca, Davide Donati, and Chiara Trifone]{Giuseppe Cosma Brusca, Davide Donati, and Chiara Trifone \\
\tiny SISSA \\ \tiny
via Bonomea 265\\ \tiny
34136 Trieste, Italy}

\thanks{Email addresses: gbrusca@sissa.it, ddonati@sissa.it, ctrifone@sissa.it}

\begin{abstract}
We discuss a model for phase transitions in which a double-well potential is singularly perturbed by possibly several terms involving different, arbitrarily high orders of derivation. We study by $\Gamma$-convergence the asymptotic behaviour as $\e\to0$ of the functionals 
\begin{equation*}
    F_\e(u):=\int_\Omega \Bigl[\frac{1}{\e}W(u)+\sum_{\ell=1}^{k}q_\ell\e^{2\ell-1}|\nabla^{(\ell)}u|_\ell^2\Bigr]\,dx, \qquad u\in H^k(\Omega),
\end{equation*} for fixed $k>1$ integer, addressing also the case in which the coefficients $q_1,...,q_{k-1}$ are negative  and $|\cdot|_\ell$ is any norm on the space of symmetric $\ell$-tensors for each $\ell\in\{1,...,k\}$. 
The negativity of the coefficients leads to the lack of a priori bounds on the functionals; such issue is overcome by proving a nonlinear interpolation inequality. With this inequality at our disposal, a compactness result is achieved by resorting to the recent paper \cite{BDS}.
A further difficulty is the presence of general tensor norms which carry anisotropies, making standard slicing arguments not suitable. We prove that the $\Gamma$-limit is finite only on sharp interfaces and that it equals an anisotropic perimeter, with a surface energy density described by a cell formula.

\medskip



\end{abstract}
\maketitle

\section{Introduction}
In the recent paper \cite{BDS}, it is studied a model for phase transitions in which a double-well potential is perturbed by a singular term involving derivatives of high order. In a one-dimensional setting, this model is described by the family of functionals defined for $\e>0$ as
\begin{equation}\label{functionalBDS1d}
 E^1_\e(u, I):= \begin{cases} 
 \displaystyle\int_I\Bigl[\frac{1}{\e}W(u) + \e^{2k-1}|u^{(k)}|^2\Bigr]\,dt  & \text{ if } u\in H^k(I), \\
     +\infty & \text{ if } u\in L^2(I)\setminus H^k(I),
\end{cases}
\end{equation}
where $I$ is a bounded, open interval, $W$ is a non-negative double-well potential, and $k$ is a fixed positive integer that can be assumed to be strictly larger than $2$, the cases $k=1$ and $k=2$ being already studied by Modica and Mortola in \cite{MM}  (see also \cite{Modica1987}) and Fonseca and Mantegazza in \cite{FM}, respectively. As customary for phase transition models, the study of the functionals \eqref{functionalBDS1d} is divided in two parts: first, compactness properties are investigated as $\e$ tends to $0$, here it is proved that a family of functions having equi-bounded energy has a sharp interface as a cluster point; then, the asymptotic analysis is performed, in this case, by $\Gamma$-convergence. Under some mild assumptions on the potential $W$, and, in particular, without requiring any growth condition on $W$ at infinity, it is proved that (see Theorem \ref{BDS1dim} below for the precise statement) given $\{u_\e\}_\e\subset H^k(I)$ satisfying $\sup_\e E^1_\e(u_\e, I)<+\infty$, there exists a function $u\in BV(I;\{-1,1\})$ such that, upon extracting a subsequence, $u_\e\to u$ in measure as $\e\to0$, having set the two wells of $W$ at $-1$ and $1$. Hence, the $\Gamma$-limit as $\e\to0$ computed with respect to the convergence in measure is finite only on $BV(I;\{-1,1\})$ and it is proved to be equal to $m_k \#( S(u)\cap I)$, where $S(u)$ denotes the set of the discontinuity points of $u$ and $m_k$ is a strictly positive constant determined by the optimal-profile problem
\begin{equation}\label{minBDS}
    m_k:=\inf\Bigl\{ \int_{-\infty}^{+\infty}\Bigl[W(v) + |v^{(k)}|^2\Bigr]\,dt : v\in H^k_{\text{loc}}(\R), \lim_{t\to\pm \infty} v(t)=\pm1 \Bigr\}.
\end{equation}
Note that such result generalizes the one by Fonseca and Mantegazza for $k=2$ to a wider class of double-well potentials.

\noindent The precompactness of a family with equi-bounded energy $\{u_\e\}_\e$ is obtained  by adapting an argument originally devised by Solci \cite{Solci} for perturbations giving free-discontinuity functionals, which relies on local interpolation inequalities on intervals that allow to estimate the $L^2$-norm of an intermediate derivative $u_\e^{(\ell)}, \ell\in\{1,...,k-1\}$, in terms of the integral of $W(u_\e)$, which is related to the $L^2$-norm of $u_\e-1$ or $u_\e+1$ assuming that $W$ is quadratic close to the bottom of the wells, and the $L^2$-norm of $u_\e^{(k)}$.  In this way, small {\em transition intervals} on which $u_\e$ `passes from a well to the other' are detected, with the further condition that, at the endpoints, all the intermediate derivatives are small. These observations lead to considering the minimum problem  
\begin{align*}
    m_k(\eta, T) & :=\inf\Bigl\{ \int_{-T}^{T}\Bigl[W(v) + |v^{(k)}|^2\Bigr]\,dt : v\in H^k((-T,T)), \\
     & \qquad  |v(-T)+1|<\eta, |v(T)-1|<\eta,  |v^{(\ell)}(\pm T)|<\eta \text{ for all } \ell\in\{1,...,k-1\} \Bigr\}
\end{align*}
that, once optimized in $\eta$ and $T$, yields \eqref{minBDS}. Since, in turn, \eqref{minBDS} is proved to be strictly positive, the equi-boundedness of the energy implies that the number of the aforementioned intervals on which a transition occurs is equi-bounded as well, and this property, in broad terms, yields compactness. Having proved that for every $\e$ each transition interval increases the energy $E_\e^1$ by the fixed amount $m_k$, which is independent of $\e$, the computation of the $\Gamma$-limit follows by standard arguments. 

The general $d$-dimensional case deals with the functionals 
\begin{equation}\label{functionalBDS}
    E_\e(u, \Omega):= \begin{cases}  \displaystyle\int_\Omega \Bigl[\frac{1}{\e}W(u) + \e^{2k-1}\|\nabla^{{(k)}}u\|_k^2\Bigr]\,dx & \text{ if } u\in H^k(\Omega) ,\\
    +\infty & \text{ if } u\in  L^2(\Omega)\setminus H^k(\Omega),
\end{cases}
\end{equation}
where $\Omega$ is a bounded, open subset of $\R^d$ with Lipschitz boundary, $W$ and $k$ are as above, and $\|\cdot\|_k$ is the operatorial norm of a symmetric $k$-tensor defined as
\begin{equation}\label{operatorialnorm}
    \|T\|_k:=|\sup\{T(\,\underbrace{\xi,...,\xi}_{k \text{ times}}\,): |\xi|=1\}|.
\end{equation}
In \cite{BDS}, the study of functionals \eqref{functionalBDS} is confined to the computation of the $\Gamma$-limit,  and it is remarked that the  compactness theorem holds true assuming suitable growth conditions at infinity on $W$. The choice of the operatorial norm makes the asymptotic analysis of such functionals a one-dimensional task since slicing arguments are well suited in this case. Indeed, \eqref{operatorialnorm} and the slicing properties of Sobolev functions imply that 
\begin{equation}\label{operatorial}
   \|\nabla^{(k)}v(x)\|_k \geq \Bigl|\frac{d^k v^{\xi,y}}{dt^k}(t)\Bigr| , \qquad \text{ for } \Ld \text{-a.e.\ } x\in\Omega,
\end{equation}
where $\xi$ is a unit vector, $y\in\{z\in\Rd: z\cdot\xi=0\}$, and we set $x=y+t\xi$ and $v^{\xi,y}(t):=v(y+t\xi)$. Employing the blow-up argument of Fonseca and M\"uller \cite{FonsecaMuller} together with \eqref{operatorial} and applying the $\Gamma$-limit result previously obtained in dimension $1$, it is proved that the $\Gamma$-liminf is estimated by
\begin{equation}\label{GammalimitBDSd}
    m_k P(\{u=1\};\Omega),
\end{equation}
where $P(\{u=1\};\Omega)$ is the perimeter of the set $\{u=1\}$ in the domain $\Omega$, representing the measure of the interface in higher dimension.
Finally, the construction of a recovery sequence, that by density can be confined to the case of a smooth interface $\partial\{u=1\}\cap \Omega=\partial E \cap\Omega$ for some $E$ smooth, open set, is obtained by composing the signed distance function from $\partial E$ with an (almost) optimal-profile for the one-dimensional problem \eqref{minBDS} suitably dilated at scale $\e$.  In this way, the energy of such sequence can be computed by slicing in direction of the normal field to $\partial E\cap \Omega$ and, since for every $\e$ each slice increases the energy by the fixed amount $m_k$, \eqref{GammalimitBDSd} is proved to be the $\Gamma$-limit.

In this paper, we study a more general family of functionals in which possibly several singular perturbations, and accordingly several order of derivatives, are involved with no restrictions on the tensor norms. We study the asymptotic behaviour, as the positive parameter $\e$ tends to $0$, of functionals defined as
\begin{equation}\label{functionals}
F_\e(u, \Omega):=  
\begin{cases}
  \displaystyle\int_\Omega \Bigl[\frac{1}{\e}W(u)+\sum_{\ell=1}^{k}q_\ell\e^{2\ell-1}|\nabla^{(\ell)}u|_\ell^2\Bigr]\,dx & \text{ if } u\in H^k(\Omega), \\
  +\infty & \text{ if } u\in L^2(\Omega)\setminus H^k(\Omega),
\end{cases}
\end{equation}
where $\Omega$ is a bounded, open subset of $\R^d$ with $C^1$ boundary, $W$ is a non-negative double-well potential, $k$ is a fixed positive integer larger than $1$, the coefficients $q_1,...,q_{k-1}$ are real constants, possibly negative, $q_k=1$, and $|\cdot|_\ell$ is a norm on the space of symmetric $\ell$-tensors for every $\ell\in\{1,...,k\}$.
More specifically, we suppose that $W:\R\to [0,+\infty)$ is a continuous function satisfying the following assumptions: 
\begin{itemize}
    \item[(H$1$)] $W(s)=0$ if and only if $s\in\{-1,1\}$;
    \item[(H$2$)] $W(s)\geq \min\{(s+1)^2, (s-1)^2 \}$ for every $s\in\R$;
    \item[(H$3$)] $W(s)\leq  W(t)+1$ for every $t\in\R$ and for every $s\in\R$ with $|s|\leq |t|$.
\end{itemize}
We remark that, with some minor adaptations in the computations, we could also assume more in general that, for some positive constants $\alpha,\beta$, the following conditions hold:
\begin{itemize}
    \item[(H$2$')] $W(s)\geq \alpha\min\{(s+1)^2, (s-1)^2 \}$ for every $s\in\R$;
    \item[(H$3$')] $W(s)\leq  \beta W(t)+\beta$ for every $t\in\R$ and for every $s\in\R$ with $|s|\leq |t|$,
\end{itemize}
having our main results still holding unchanged.

One of the difficulties in dealing with general tensor norms is the arising of possible anisotropies, which make slicing techniques inadequate. For this reason, we rely on arguments that are not based on a one-dimensional analysis. But the main issue of our study concerns the presence of negative terms in the energy. To overcome this difficulty we combine the approaches of  Chermisi, Dal Maso, Fonseca, and Leoni in \cite{ChermisiDalMaso11}, and of Cicalese, Spadaro, and Zeppieri in \cite{CSZ}, in which, independently, it is treated the case $k=2$: in particular, in the first work, the analysis is performed also in higher dimension (with the operatorial norms), while the latter gives a detailed description of the one dimensional case. Our first step is proving a global interpolation inequality  at scale $\e$ that establishes bounds on the square of the $L^2$ norm of $\nabla^{(\ell)}u$ in terms of $\int_\Omega W(u)\,dx$ and the square of the $L^2$ norm of $\nabla^{(k)}u$. This inequality also determines lower bounds for the constants $q_1,...,q_{k-1}$ for which our asymptotic analysis can be carried out.
\begin{theorem}\label{nonlinintd} Let $\Omega\subset \R^d$ be a bounded, open set with $C^1$ boundary, let $k>1$ be an integer, and assume that {\rm (H2)} is satisfied. Then, for every $\ell\in\{1,...,k-1\}$, there exists a positive constant $q^*_\ell$ independent of $\Omega$ such that for every $q<q^*_{\ell}$ there exists $\e_0=\e_0(q,\Omega)$ such that it holds
\begin{equation}\notag
    q\e^{2\ell}\int_\Omega \|\nabla^{(\ell)}u\|_\ell^2\,dx \leq \int_\Omega \Bigl[W(u)+ \e^{2k}\|\nabla^{(k)}u\|_k^2\Bigr]\,dx
\end{equation}
for every $\e\in(0,\e_0)$ and $u\in H^k(\Omega)$. Moreover, for every $\ell\in\{1,...,k-1\}$, there exists a positive constant $\widetilde{q}^*_\ell$ independent of $\Omega$ such that for every $q<\widetilde{q}^*_{\ell}$ there exists $\e_0=\e_0(q,\Omega)$ such that it holds 
\begin{equation}\notag
    q\e^{2\ell}\int_\Omega |\nabla^{(\ell)}u|_\ell^2\,dx \leq \int_\Omega \Bigl[W(u)+ \e^{2k}|\nabla^{(k)}u|_k^2\Bigr]\,dx
\end{equation}
 for every $\e\in(0,\e_0)$ and $u\in H^k(\Omega)$. In particular, $\widetilde{q}^*_\ell$ depends on $q^*_\ell, |\cdot|_\ell$, and $|\cdot|_k$.
\end{theorem}
The above interpolation inequality is proved by, first, adapting the arguments in \cite{CSZ}, which yield the result in the case $d=1$, and then by applying the slicing method as in \cite{ChermisiDalMaso11} to recover the general case. We remark that assumption (H$2$), despite not being fully satisfactory in describing the model as it prescribes the growth rate at infinity of the double-well potential, is needed to apply classical, global $L^2$-interpolation estimates. Indeed, the presence of negative terms in the energy seems to prevent the use of the local arguments applied in \cite{BDS}, as it is not evident to us how to estimate the energy when far from the bottom of the wells.

Having at disposal such interpolation estimate allows us to find a priori bounds on the functionals \eqref{functionals}, provided the coefficients $q_1,...,q_{k-1}$ are not too small, and this is the key ingredient to get the main compactness result.

\begin{theorem}\label{compactness} Let $\Omega\subset \Rd$ 
be a bounded, open set with $C^1$ boundary, let $\{F_\e\}_\e$ be defined as in \eqref{functionals} with $k>1$ an integer, let $\mathcal{N}:=\{\ell\in\{1,...,k-1\}: q_\ell \leq 0\}$, and let $\{\e_n\}_n$ be a positive sequence converging to $0$. Assume that {\rm (H1)} and {\rm (H2)} are satisfied and that one of the following conditions holds: 
\begin{itemize}
    \item[\emph{(i)}] $\mathcal{N}=\emptyset$;
    \item [\emph{(ii)}] $\mathcal{N}\neq \emptyset$ and $q_\ell>-\overline{q}_\ell$ for every $\ell\in\mathcal{N}$ for suitable positive constants $\{\overline{q}_\ell:\ell\in\mathcal{N}\}$.
\end{itemize}
If $\{u_n\}_n\subset H^k(\Omega)$ is a sequence such that
\begin{equation*}
    \sup_n F_{\e_n}(u_n, \Omega)<+\infty,
\end{equation*}
then there exist a subsequence $\{u_{n_j}\}_j$ and a function $u\in BV(\Omega;\{-1,1\})$ such that $u_{n_j}\to u$ in $L^2(\Omega)$ as $j\to+\infty$. In particular, if $\mathcal{N}\neq \emptyset$, letting $\widetilde{q}^*_1,...,\widetilde{q}^*_{k-1}$ denote the same positive constants appearing in Theorem \ref{nonlinintd} and given $\{\alpha_\ell : \ell\in \mathcal{N}\}$ such that $\alpha_\ell>0$ for all $\ell\in\mathcal{N}$ and $\sum_{\ell\in\mathcal{N}}\alpha_\ell=1$, it holds that $\overline{q}_\ell=\alpha_\ell \widetilde{q}^*_\ell$ for all $\ell\in\mathcal{N}$; hence, $\{\overline{q}_\ell:\ell\in\mathcal{N}\}$ are independent of $\Omega$.
\end{theorem}
Clearly, if the coefficients $q_1,...,q_{k-1}$ are non-negative, the compactness result in \cite{BDS} suggests there is no need to resort to Theorem \ref{nonlinintd} since, by the equivalence of the tensor norms, compactness of the functionals $\{F_\e\}_\e$ can be inferred by compactness of the functionals $\{E_\e\}_\e$ defined in \eqref{functionalBDS} through a direct comparison. For this reason, we also prove a compactness result for functionals $\{E_\e\}_\e$ (Proposition \ref{compactnessBDSanydim}). The first step is enhancing the compactness result for functionals \eqref{functionalBDS1d} stated in \cite{BDS} taking advantage of assumption (H$2$). Indeed, this assumption implies that $W$ has quadratic growth at infinity, and this condition, in turn, allows to improve precompactness in measure into precompactness in $L^2$. Then, we recover the general $d$-dimensional case by resorting to a precompactness criterion in $L^1$ by slicing that has been introduced in \cite{ABS} (see Proposition \ref{slicingcompactness} below), which in our case is sufficient to obtain precompactness in $L^2$. 

Afterwards, we compute the $\Gamma$-limit of the family $\{F_\e\}_\e$ with respect to the strong convergence in $L^2(\Omega)$. In order to state our result, we introduce some notation.
We set $\mathbb{S}^{d-1}:=\{\xi\in\R^d : |\xi|=1\}$. For fixed $\nu\in \Sd$ and $T>0$, we let the symbol $Q_T^\nu$ denote an open cube of side length $T$, centred at the origin, with two faces orthogonal to $\nu$,
\begin{equation}\label{Q^nu}
    Q^\nu_T:=\Bigl\{x\in \R^d : |x\cdot \nu|<\frac{T}{2},\,|x\cdot \nu_i|<\frac{T}{2} \text{ for every } i\in\{1,...,d-1\}\Bigr\},
\end{equation}
where $\{\nu_1,...,\nu_{d-1},\nu\}$ is an orthonormal basis of $\Rd$ that we fix arbitrarily. 

\noindent We fix a function $\overline{u}\in H_{\text{loc}}^k(\R)$ such that $\overline{u}(t)=-1$ if $t\leq -1/8$, $\overline{u}(t)=1$ if $t\geq 1/8$, $|\overline{u}|\leq 1$, and
\begin{equation*}
    0<\int_{-\infty}^{+\infty} \Bigl[W(\overline{u})+\sum_{\ell=1}^k|q_\ell||\overline{u}^{(\ell)}|^2\Bigr]\,dt<+\infty,
\end{equation*}
then, for fixed $\nu\in \Sd$ we define the function
\begin{equation*}
    \overline{u}^\nu(x):=\overline{u}(x\cdot \nu) \quad \text{ for all } x\in\R^d,
\end{equation*}
and for every $\e>0$ we further define 
\begin{equation*}
\overline{u}^\nu_\e(x):=\overline{u}^\nu\bigl(\frac{x}{\e}\bigr)=\overline{u}\bigl(\frac{x\cdot\nu}{\e}\bigr)  \quad \text{ for all } x\in\R^d.  
\end{equation*}
We introduce a class of admissible functions on unit cubes as
\begin{equation}\label{A_nu}
\A_\e^\nu:=\bigl\{v\in H^k(Q_1^\nu): v= \overline{u}^\nu_\e \text{ near } \partial Q_1^\nu \bigr\},
\end{equation}
where by `near $\partial Q_1^\nu$', we mean `in $Q_1^\nu \setminus Q_r^\nu$ for some $r\in(0,1)$'.

\noindent Finally, for every $\nu\in \Sd$ we introduce the function
 \begin{equation}\label{density}
       g(\nu):=\inf\Big\{\int_{{Q_1^\nu}} \Bigl[\frac{1}{\e}W(v)+\sum_{\ell=1}^{k}q_\ell\e^{2\ell-1} |\nabla^{(\ell)}v|_\ell^2\Bigr]\,dx  :  v \in \A^\nu_\e,\e\in(0,1)\Big\},
 \end{equation}
and we state our last main result.

 \begin{theorem}\label{main theorem}  Let $\Omega\subset \Rd$ 
be a bounded, open set with $C^1$ boundary, let $\{F_\e\}_\e$ be defined as in \eqref{functionals} with $k>1$ an integer, and let $\mathcal{N}:=\{\ell\in\{1,...,k-1\}: q_\ell \leq 0\}$. Assume that {\rm (H1)-(H3)} are satisfied and that one of the following conditions holds: 
\begin{itemize}
    \item[\emph{(i)}] $\mathcal{N}=\emptyset$;
    \item [\emph{(ii)}] $\mathcal{N}\neq \emptyset$ and $q_\ell>-\overline{q}_\ell$ for every $\ell\in\mathcal{N}$ with the same positive constants $\{\overline{q}_\ell:\ell\in\mathcal{N}\}$ independent of $\Omega$ appearing in Theorem \ref{compactness}.
\end{itemize} Then the family of functionals $\{F_\e\}_\e$ $\Gamma$-converges, with respect to the strong convergence in $L^2(\Omega)$, to the functional $F_0$ given by 
 \begin{equation}\notag
    F_0(u, \Omega):=
    \begin{cases}
\displaystyle\int_{\partial^*\!A\cap \Omega} g(\nu_A)\, d\Hd & \text{ if } u\in BV(\Omega;\{-1,1\}),\, \text{ with } u=2\chi_A-1, \\
  +\infty & \text{ if } u\in L^2(\Omega)\setminus BV(\Omega;\{-1,1\}),
  \end{cases}
    \end{equation}
 where $A$ is a set of finite perimeter in $\Omega$, $\partial^*A$ is its reduced boundary, $\nu_A$ is its inner unit normal, and $g$ is defined as in \eqref{density}.
 \end{theorem}

 In the special case $q_1=...=q_{k-1}=0$, this result improves the corresponding one in \cite{BDS} as it deals with arbitrary norms on the spaces of tensors, upon assuming stronger hypotheses on the double-well potential.

In order to compute the $\Gamma$-limit, we follow the strategy presented in \cite{ChermisiDalMaso11}. The lower bound is proved by using the blow-up argument and relying on a method due to De Giorgi that allows to vary boundary conditions on converging sequences (Lemma \ref{lemma:modifying sequences}). The construction of a recovery sequence is confined to the case of polyhedra, which are seen to be dense in energy for the limit $F_0$, being the surface energy density $g$ upper semicontinuous (Proposition \ref{densitywelldef}).

\noindent  As stated below in Remark \ref{re.inftolim}, the density \eqref{density} can be described by the following asymptotic formula
\begin{align*}
    g(\nu)& = \lim_{\e\to0}\,\inf\Big\{\int_{{Q_1^\nu}} \Bigl[\frac{1}{\e}W(v)+\sum_{\ell=1}^{k}q_\ell\e^{2\ell-1} |\nabla^{(\ell)}v|_\ell^2\Bigr]\,dx  :  v \in \A^\nu_\e\Big\} \\
    & = \lim_{T\to+\infty} \frac{1}{T^{d-1}} \,\inf\Big\{\int_{{Q_T^\nu}} \Bigl[W(v)+\sum_{\ell=1}^{k}q_\ell |\nabla^{(\ell)}v|_\ell^2\Bigr]\,dx  :  v \in \widetilde{\A}^\nu_T\Big\},
\end{align*}
where 
\begin{equation*}
\widetilde{\A}_T^\nu:=\bigl\{v\in H^k(Q_T^\nu): v= \overline{u}^\nu \text{ near } \partial Q_T^\nu \bigr\},
\end{equation*}
and the last equality follows by setting $T:=\frac{1}{\e}$ and a change of variables.

\noindent The formula appearing in \eqref{density} should be compared with the one in \cite{ChermisiDalMaso11}. In that work such density is a constant because of the absence of anisotropies, and it is given by 
    \begin{equation*}
       \inf\Big\{\int_{{Q}} \Bigl[\frac{1}{\e}W(v)-q\e \|\nabla v\|_1^2+\e^3\|\nabla^2v\|_2\Bigr]\,dx  :  v \in \A,\e\in(0,1)\Big\},
 \end{equation*}
with $Q=(-\frac{1}{2}, \frac{1}{2})^d$, and
\begin{align}\notag
\A:=\bigl\{v\in H^k_{\text{loc}}(\R^d) :\,\, & v=-1 \text{ near } x\cdot e_d=-\frac{1}{2}, v=1 \text{ near } x\cdot e_d =\frac{1}{2}, \\ \notag
& v(x+e_i)=v(x) \text{ for all } x\in \R^d, i\in\{1,...,d-1\}\bigr\},
\end{align}
where $\{e_1,...,e_d\}$ is the canonical basis of $\Rd$ and by `near $x\cdot e_d =\pm\frac{1}{2}$' it is meant `in a neighborhood of the set $\{ x\in\R^d: x\cdot e_d=\pm \frac{1}{2} \}$'.

\noindent In light of this  observation, we expect that our density $g$ could be characterized as
\begin{equation*}
       \inf\Big\{\int_{{Q_1^\nu}} \Bigl[\frac{1}{\e}W(v)+\sum_{\ell=1}^{k}q_\ell\e^{2\ell-1} |\nabla^{(\ell)}v|_\ell^2\Bigr]\,dx  :  v \in \widetilde{\A}^\nu,\e\in(0,1)\Big\},
 \end{equation*}
with
\begin{align}\notag
\widetilde{\A}^\nu:=\bigl\{v\in H^k_{\text{loc}}(\R^d) :\,\, & v=-1 \text{ near } x\cdot \nu =-\frac{1}{2}, v=1 \text{ near } x\cdot \nu =\frac{1}{2}, \\ \notag
& v(x+\nu_i)=v(x) \text{ for all } x\in \R^d, i\in\{1,...,d-1\}\bigr\},
\end{align}
where $\{\nu_1,...,\nu_{d-1}, \nu\}$ is the orthonormal basis that we fixed in the definitions \eqref{Q^nu} and \eqref{A_nu}.
To prove this  fact, we would need to impose boundary conditions resorting to Lemma \ref{lemma:modifying sequences}. In turn, to do this we would need to improve Theorem \ref{nonlinintd}, even just in the case of the operatorial norms, to the case $\Omega$ be an open cube. In \cite{ChermisiDalMaso11}, this  result is obtained by slicing, which is well suited because, roughly speaking, the operatorial norm of a $1$-tensor, the gradient, coincides with its euclidean norm as a vector of $\R^d$. Because of the higher order of derivatives that our interpolation estimates involve, such argument seems to be difficult to adapt. As an additional consequence of the lack of interpolation estimates on cubes, the proofs of Lemma \ref{lemma:modifying sequences} (and its statement) and of the lower bound (Proposition \ref{prop:liminf}) present several differences with respect to their counterparts in \cite{ChermisiDalMaso11}. In particular, we will need to apply our interpolation inequalities on a $C^{1}$ subset of the cube $Q^\nu_1$ that coincides with this cube inside the strip $\{x\in \Rd:|x\cdot \nu|\leq \frac{1}{4}\}$ which is close to the interface on which most of the energy is stored.

We remark that, in the case $k=2$, such asymptotic analysis has already been performed for more general functionals that are bounded from above and from below by positive multiples of the functionals \eqref{functionals}, allowing, for instance, the dependence of the energy densities on the variable $x$. In particular, in \cite{BaiaCher}, Ba\'ia, Barroso, Chermisi, and Matias studied the case $q_1=1$; while, in the recent paper \cite{Donnarumma}, Donnarumma has addressed the problem of deterministic and stochastic homogenization of such functionals. As a consequence of their results, the validity of Theorem \ref{main theorem} for generic norms follows. For a treating of phase transitions models encoding homogenization phenomena we also refer to \cite{AnsBraChia1} and \cite{AnsBraChia2}.  

In the final section, we further discuss the cell problem in \eqref{density} in the one-dimensional case, and we also address the case in which all the norms $|\cdot|_\ell$ appearing in \eqref{functionals} are `compatible with slicing'; that is, satisfying an estimate like \eqref{operatorial}, and all the coefficients $q_\ell,\,\ell\in\{1,...,k-1\}$ are non-negative. In both cases, the function $g$ in \eqref{density} is constant, and its value is determined by the one-dimensional optimal-profile problem 
\begin{equation*}
    m:=\inf \Bigl\{ \int_{-\infty}^{+\infty}\Bigl[W(u)+\sum_{\ell=1}^kq_\ell|u^{(\ell)}|^2\Bigr]\,dt: u\in H^k_{\text{loc}}(\R), \lim_{t\to\pm\infty} u(t)=\pm1 \Bigr\}.
\end{equation*}
By a cut-off argument, we are also able to prove, without requiring (H$3$), that the above infimum actually is a minimum, determining the existence of optime-profiles for a wide class of phase-transitions problems on the real line.
 
We conclude this introductory section by noting that in recent years  there has also been a growing interest in non-local counterparts of the models for phase transitions described above. Starting from the foundational work of Savin and Valdinoci \cite{Savin}, several works on the subject have been published. We  mention here the works \cite{DalFonLeoSecondOrder,PalatucciVincini,Picerni} and the recent paper by Solci \cite{SolciNonLocal}, in which a non-local counterpart of the results of \cite{BDS} is proved. It would be interesting to extend our analysis to that context.

\smallskip

 Near the completion of our work, we learnt that, independently, Brazke, G\"otzmann, and Kn\"upfer performed the same asymptotic analysis in the one-dimensional case without requiring (H$3$) and with the further assumption that $q_1=...=q_{k-2}=0$; that is, taking into account a possibly negative singular perturbation described by the only derivative of order $k-1$ (see \cite{Brazke}) using an argument different from ours for the liminf inequality.

\section{Notation and preliminaries}
In this section we collect the main notation and some preliminary results using \cite{AmbrFuscoPallara} as main reference. 

\smallskip

\textbf{Notation.} In what follows, we let $\Omega$ be a bounded, open subset of $\R^d$ with Lipschitz or $C^1$ boundary, for $d\geq1$. We let $\lfloor t\rfloor$ and $\lceil t\rceil$ denote the lower and upper integer part of a real number $t$, respectively. We let $\{e_1,...,e_d\}$ denote the canonical basis of $\R^d$ and $\I$ the identity matrix in $\R^{d\times d}$. Given any $x\in \R^d$, we let $|x|$ denote its euclidean norm. We let $\mathcal{L}^d$ denote the Lebesgue measure on $\R^d$ and $\mathcal{H}^{d-1}$ denote the $(d-1)$-dimensional Hausdorff measure on $\R^d$. Given a Borel measure $\mu$ on $\Rd$ and a Borel subset $B$ of $\Rd$, we let $\mu\mres B$ denote the restriction of $\mu$ to $B$, that is, the measure defined by $\mu\mres B (A):=\mu(A\cap B)$. For every positive integer $\ell$, we let $\|\cdot \|_\ell$ denote the operatorial norm of a symmetric $\ell$-tensor on $\Rd$ defined by \eqref{operatorialnorm}, while a generic norm will be denoted by $|\cdot|_\ell$. Given $\nu\in \Sd$, we let the symbol $\nu^\perp$ denote the hyperplane $\{x\in\Rd:x\cdot\nu=0\}$, and given $T>0$ we let $Q_T^\nu(z)$ denote the open cube of side length $T$, centred at $z$, and with two faces orthogonal to $\nu$,
\begin{equation*}\label{oriented cube}
    Q^\nu_T(z):=\Bigl\{x\in \R^d : |(z-x)\cdot \nu|<\frac{T}{2},\,|(z-x)\cdot \nu_i|<\frac{T}{2} \text{ for every } i\in\{1,...,d-1\}\Bigr\},
\end{equation*}
where $\{\nu_1,...,\nu_{d-1}, \nu\}$ is a fixed orthonormal basis of $\R^d$. Given $r>0$, let $B_r(z)$ denote the $d$-dimensional open ball of centre $z$ and radius $r$. We adopt the convention that, both for cubes and balls, when omitted, the centre is in the origin in accordance with \eqref{Q^nu}. 

\smallskip

\textbf{Functions of bounded variation.} Given $u\in L^1(\Omega)$, the variation of $u$ in $\Omega$ is defined as
\[
V(u; \Omega):=\sup\Bigl\{\int_\Omega u\text{ div}  \varphi\,dx : \varphi \in C^1_c(\Omega; \R^d),\, |\varphi|\leq 1 \Bigr\}.
\]
We say that $u$ has bounded variation, and we write $u\in BV(\Omega)$, if $V(u;\Omega)<+\infty$. For such a function $u$, the singular set $S(u)$ is defined as the complement in $\Omega$ of the set of Lebesgue points; i.e., 
\begin{equation*}
S(u):=\Omega\setminus\Bigl\{x\in\Omega : \exists z\in \R^N \text{ such that } \lim_{r\to0^+}\frac{1}{\Ld(B_r(x))}\int_{B_r(x)}|u(y)-z|\,dy =0 \Bigr\}.
\end{equation*}
The distributional derivative of $u$ in $\Omega$ is the sum of three mutually singular vector measures, $Du= \mu^a+\mu^j+\mu^c$, where $\mu^a$ is absolutely continuous with respect to $\Ld$, $\mu^j$, the jump part, is a 
$(d-1)$-dimensional measure concentrated on $S(u)$ with density given by $(u^+-u^-)\nu_u$, where $u^-, u^+$ are the bilateral traces of $u$ on $S(u)$ defined in accordance with the unit vector field $\nu_u$, while $\mu^c$ is the so-called Cantor part.  

\smallskip

\textbf{Sets of finite perimeter.} Given an $\Ld$-measurable set $A \subseteq \Omega$, the perimeter of $A$ in $\Omega$ is defined as
\[
P(A; \Omega):=\sup\Bigl\{\int_A \text{div}  \varphi\,dx : \varphi \in C^1_c(\Omega; \R^d),\, |\varphi|\leq 1 \Bigr\}.
\]
We say that $A$ has finite perimeter in $\Omega$ if $P(A;\Omega)<+\infty$. Since $\Omega$ is assumed to be bounded, the set $A$ has finite perimeter in $\Omega$ if and only if its characterstic function $\chi_A$ is a function of bounded variation in $\Omega$; i.e., $\chi_A\in BV(\Omega)$. Therefore, for such a set $A$, the distributional derivative of $\chi_A$ on $\Omega$ is a vector measure $\mu_A$. Denoting by $|\mu_A|$ the total variation of this measure, the reduced boundary of $A$ is defined as
\begin{equation*}\label{reduced}
    \partial^*A:=\Bigl\{x\in \R^d : \text{ there exists } \nu_A(x):=\lim_{r\to0}\frac{\mu_A(B_r(x))}{|\mu_A|(B_r(x))}, \text{ with } |\nu_A(x)|=1\Bigr\},
\end{equation*}
and the unit vector $\nu_A(x)$ stands for the inner unit normal (in the sense of measure theory) to $A$ at the point $x$.

We summarize the main properties we need about sets of finite perimeter in the following statements.
\begin{theorem}\label{thm:structurethm} If $A$ is a set of finite perimeter in $\Omega\subset \R^d$, the distributional derivative of $\chi_A$ in $\Omega$ is
\begin{equation}\notag
    \mu_A  = \nu_A \mathcal{H}^{d-1} \mres (\partial^*A \cap\Omega), 
\end{equation}
    and $P(A;\Omega)=|\mu_A|(\R^d)=\mathcal{H}^{d-1}(\partial^* A\cap \Omega)$. Moreover, setting $H_{\nu}:=\{y\in \R^d : y\cdot\nu > 0\}$, for every point $x\in\partial^* A\cap \Omega$, it holds that
\begin{align}\notag
    & \chi_{\frac{A-x}{r}} \to  \chi_{H_{\nu_A(x)}} \quad \text{ in } L^1_{\text{loc}}(\R^d) & \quad \text{ as }r\to0, \\\notag
    & \frac{\mathcal{H}^{d-1}(\partial^* A\cap Q^{\nu_A(x)}_r(x))}{r^{d-1}}\to 1  & \quad \text{ as }r\to0.
\end{align}
\end{theorem}

\noindent We remark that a function $u\in BV(\Omega;\{-1,1\})$ is of the form $u=2\chi_A-1$ for a certain $A$ set of finite perimeter in $\Omega$. In particular, it holds that $S(u)=\partial^*A\cap \Omega$ up to $\Hd$-negligible sets and the distributional derivative of $u$ in $\Omega$ is the measure $2\nu_A\Hd\mres (\partial^*A\cap \Omega)$. 

We shall also recall a well-known result of approximation of sets of finite perimeter via polyhedral sets, see \cite[Lemma 3.6]{AcerbiBouchitte} and \cite[Lemma 3.1]{Baldo}. 

\begin{theorem}\label{smoothapproximation}
   Let $A$ be a set of finite perimeter in $\Omega\subset \R^d,\, d\geq 2$. There exists a sequence $\{A_n\}_{n\in\N}$ of polyhedral sets satisfying
    \begin{enumerate} [(i)]
    \item \label{approx1}
       $\chi_{A_n} \to \chi_A  \text{ in  } L^1(\Omega) \quad  \text{ as } n \to \infty,$
        \item 
        \label{approx2}
      $  P(A_n;\Omega)\to P(A;\Omega) \quad
      \text{ as } n\to+\infty, $
      \item $\Hd(\partial^*A_n\cap\partial\Omega)=0$ for every $n\in\N$.
    \end{enumerate}
Moreover, for every non-negative continous function $g$ on $\Sd$, there holds 
\begin{align}\label{contlimit}
    \lim_{n\to + \infty} \int_{\partial^*\! A_n \cap \Omega} g(\nu_{A_n}(x)) d\mathcal{H}^{d-1}(x)=\int_{\partial ^* \!A \cap \Omega} g(\nu_{A}(x)) d\mathcal{H}^{d-1}(x).
\end{align}
\end{theorem}

\smallskip

\textbf{Slicing.} Given an open set $\Omega\subseteq\Rd$, $\xi\in\Sd$, and $y\in \xi^\perp$, we define the slice $\Omega^{\xi,y}$ as
\begin{equation*}
    \Omega^{\xi,y}:=\{t\in\R: y+t\xi\in\Omega\},
\end{equation*}
and we define $\Omega^\xi$ the orthogonal projection of $\Omega$ on $\xi^\perp$ as 
\begin{equation*}
    \Omega^{\xi}:=\{y\in\xi^\perp : \Omega^{\xi,y}\neq \emptyset\}.
\end{equation*}
For every function $u:\Omega\to \R$ and for every $y\in\Omega^\xi$, we define the function $u^{\xi,y}:\Omega^{\xi,y}\to\R$ by
\begin{equation*}
    u^{\xi,y}(t):=u(y+t\xi).
\end{equation*}
 Using Theorem $11.45$ in \cite{Leoni2nd} one can see that given a function $u\in H^k(\Omega)$, for every $\xi\in \Sd$ and for $\Hd$-a.e.\ $y\in\xi^\perp$, it holds that
\begin{equation*}
    \nabla^{(\ell)}u(y+t\xi)(\,\underbrace{\xi,...,\xi}_{\ell \text{ times}}\,) = \frac{d^\ell u^{\xi,y}}{dt^\ell}(t) \quad \text{for } \mathcal{L}^1 \text{-a.e.\ }t\in \Omega^{\xi,y}, \text{and for all } \ell\in\{1,...,k-1\}.
\end{equation*}

\noindent We also mention \cite[Remark 3.104]{AmbrFuscoPallara} that yields a characterization of functions of bounded variation with values in $\{-1,1\}$ in terms of their slices.

\begin{theorem}\label{bvslicing} A function $u\in L^1(\Omega;\{-1,1\})$ belongs to $BV(\Omega;\{-1,1\})$ if and only if for every $j\in\{1,...,d\}$ the one-dimensional slices of the function $u$ given by $u^{e_j, y}$ have bounded variation in the open set $\Omega^{e_j,y}$ for $\Hd$-a.e.\ $y\in e_j^{\perp}$, and 
\begin{equation*}
    \int_{\Omega^{e_j}} \# (S(u^{e_j,y})\cap\Omega^{e_j,y})\,d\Hd(y)<+\infty.
\end{equation*}
\end{theorem}

\smallskip

$\mathbf{\Gamma}$\textbf{-convergence.} We briefly recall the definition of $\Gamma$-convergence in $L^2(\Omega)$. For a general introduction to the topic we refer the reader to \cite{BraidesGamma,DalMaso}. Given functionals $G_\e: L^2(\Omega)\to[-\infty, +\infty]$ for $\e>0$ and $G_0: L^2(\Omega)\to[-\infty, +\infty]$, we say that $G$ is the $\Gamma$-limit of the family $\{G_\e\}_\e$, and we write $G_0=\Gamma$-$\lim_{\e\to0} G_\e$ if, for every vanishing sequence $\{\e_n\}_n$ the following conditions hold:
\begin{itemize}
    \item[(i)] For every $u\in L^2(\Omega)$ and $\{u_n\}_n$ such that $u_n\to u$ in $L^2(\Omega)$ as $n\to+\infty$, it holds that
    \begin{equation*}
        \liminf_{n\to+\infty} G_{\e_n}(u_n) \geq G_0(u);
    \end{equation*}
    \item[(ii)] For every $u\in L^2(\Omega)$ there exists $\{u_n\}_n$ such that $u_n\to u$ in $L^2(\Omega)$ as $n\to+\infty$ and
    \begin{equation*}
        \limsup_{n\to+\infty} G_{\e_n}(u_n) \leq G_0(u).
    \end{equation*}
\end{itemize}
It is also convenient to recall the following definitions:
\begin{eqnarray}\notag
    \Gamma\text{-}\liminf_{\e\to0} G_\e (u) := \inf\{\liminf_{n\to+\infty} G_{\e_n}(u_n): \{\e_n\}_n \text{ vanishing},\, u_n\to u \text{ in } L^2(\Omega) \}; \\\label{def Gammalimsup}
        \Gamma\text{-}\limsup_{\e\to0} G_\e (u) := \inf\{\limsup_{n\to+\infty} G_{\e_n}(u_n): \{\e_n\}_n \text{ vanishing},\, u_n\to u \text{ in } L^2(\Omega) \}.
\end{eqnarray}
Both functionals above are lower semicontinuous  with respect to the strong convergence in $L^2(\Omega)$, they coincide if and only if $\Gamma$-$\lim_{\e\to0} G_\e$ exists, and in this case the $\Gamma$-limit coincides with their common value.

\smallskip 

We conclude this section proving a characterization of the density introduced in \eqref{density}. In view of the asymptotic analysis, the following proposition makes irrelevant the choice of the orthonormal basis $\{\nu_1,...,\nu_{d-1},\nu\}$ fixed in the definitions \eqref{Q^nu} and \eqref{A_nu}.  In the following we let $SO(d)$ denote the group of $d\times d$ orthogonal matrices with determinant equal to $1$.

\begin{proposition}\label{densitywelldef} For every $\nu\in\Sd$ and for every $R\in SO(d)$ such that $R\nu=\nu$ it holds 
 \begin{align}\label{densitybis} 
       g(\nu) & = \inf\Big\{\int_{{R(Q_1^\nu)}} \Bigl[\frac{1}{\e}W(v)+\sum_{\ell=1}^{k}q_\ell\e^{2\ell-1} |\nabla^{(\ell)}v|_\ell^2\Bigr]\,dx  : v \circ R \in \A_\e^\nu,\, \e\in(0,1)\Big\} \\  \notag
       & = \inf\Big\{\int_{{R'(Q_1^\nu)}} \Bigl[\frac{1}{\e}W(v)+\sum_{\ell=1}^{k}q_\ell\e^{2\ell-1} |\nabla^{(\ell)}v|_\ell^2\Bigr]\,dx  :   v \circ R' \in \A_\e^\nu,\, \e\in(0,1), \\ \label{densitytris} 
       & \qquad \qquad \qquad \qquad \qquad \qquad \qquad \qquad \qquad \qquad \quad  R'\in SO(d) \text{ with } R'\nu=\nu\Big\}.
 \end{align}
 Moreover, $g$ is upper semicontinuous.
\end{proposition}
\begin{proof}
We first prove \eqref{densitytris}. Fix $\nu\in\Sd$ and let $\e\in(0,1)$, $R\in SO(d)$ satisfying $R\nu=\nu$, and $v\circ R\in \A^\nu_\e$. We claim that for every $\eta\in(0,\e)$, there exists a function $w_\eta\in \A_\eta^\nu$ such that 
\begin{align}\label{claimdensity}
    \limsup _{\eta\to0}\int_{Q_1^\nu} \Bigl[\frac{1}{\eta}W(w_\eta)+\sum_{\ell=1}^{k}q_\ell\eta^{2\ell-1} |\nabla^{(\ell)}w_\eta|_\ell^2\Bigr]\,dx \leq \int_{R(Q_1^\nu)} \Bigl[\frac{1}{\e}W(v)+\sum_{\ell=1}^{k}q_\ell\e^{2\ell-1} |\nabla^{(\ell)}v|_\ell^2\Bigr]\,dx.    
\end{align}

For the sake of exposition, we set $T:=1/\e$ and $\widetilde{v}(x):=v(\frac{1}{T}x)$ for all $x\in R(Q_T^\nu)$. For $S>T$, we introduce 
    \[
    \mathcal{Z}(T,S):=\{z=T(z_1\nu_1+...+z_{d-1}\nu_{d-1}): z+R(Q_T^\nu) \subset Q_S^\nu,(z_1,...,z_{d-1})\in\mathbb{Z}^{d-1} \},
    \]
    where $\{\nu_1,...,\nu_{d-1},\nu\}$ is the orthonormal basis that we arbitrarily fixed in the definitions \eqref{Q^nu} and \eqref{A_nu},
and 
    \begin{equation*}
        E(T,S):=\bigcup_{z\in \mathcal{Z}(T,S)} \bigl(z +R(Q_T^\nu) \bigr).
    \end{equation*} 
Then, we define the function
\begin{equation*}
    \widetilde{w}_S(x):=
    \begin{cases}
  \widetilde{v}(x-z) & \text{ if } x\in z+R(Q_T^\nu) \text{ for some } z\in \mathcal{Z}(T,S), \\
  \overline{u}^\nu(x) & \text{ otherwise in } Q_S^\nu,
  \end{cases}
    \end{equation*}
and we note that $\widetilde{w}\in H^k(Q_S^\nu)$ with $\widetilde{w}_S=\overline{u}^\nu$ close to $\partial Q_S^\nu$. Recalling the notation \eqref{functionals}, we have
 \begin{align*}
         \frac{F_1(\widetilde{w}_S, Q_S^\nu)}{S^{d-1}} & = \frac{F_1(\widetilde{w}_S, E(T,S)) +  F_1(\widetilde{w}_S, Q_S^\nu\setminus E(T,S)) }{S^{d-1}} \\
        & = \Bigl(\sum_{z\in \mathcal{Z}(T,S)} \frac{F_1(\widetilde{w}_S, z+R(Q_T^\nu))}{S^{d-1}}\Bigr) + \frac{F_1(\widetilde{w}_S, Q_S^\nu\setminus E(T,S))}{S^{d-1}}  \\
        & = \#\mathcal{Z}(T,S)\frac{F_1(\widetilde{v}, R(Q_T^\nu))}{S^{d-1}} + \frac{\Hd(\pi^\nu( Q_S^\nu\setminus E(T,S))}{S^{d-1}}c_0,
        \end{align*}
    having set $c_0:=\int_{-\infty}^{+\infty} \bigl[W(\overline{u})+\sum_{\ell=1}^kq_\ell|\overline{u}^{(\ell)}|^2\bigr]\,dt$ and letting $\pi^\nu$ denote the orthogonal projection on the hyperplane $\nu^\perp$. 
\noindent  We observe that 
    \begin{equation*}
        \lim_{S\to+\infty}\frac{\#\mathcal{Z}(T,S)}{(S/T)^{d-1}}=1,
    \end{equation*}
    and
    \begin{equation*}
    \Hd(\pi^\nu(Q_S^\nu\setminus E(T,S))= S^{d-1}-\#\mathcal{Z}(T,S)T^{d-1},
    \end{equation*}
therefore, 
\begin{equation*}
    \limsup_{S\to+\infty}\frac{F_1(\widetilde{w}_S, Q_S^\nu)}{S^{d-1}} \leq  \frac{F_1(\widetilde{v}, R(Q_T^\nu))}{T^{d-1}}; 
\end{equation*}
that is
\begin{equation}\label{propdensity1}
    \limsup_{S\to+\infty}\frac{1}{S^{d-1}}\int_{Q_S^\nu} \Bigl[W(\widetilde{w}_S)+\sum_{\ell=1}^kq_\ell|\nabla^{(\ell)}\widetilde{w}_S|_\ell^2\Bigr]\,dx \leq \frac{1}{T^{d-1}}\int_{R(Q_T^\nu)} \Bigl[W(\widetilde{v})+\sum_{\ell=1}^kq_\ell|\nabla^{(\ell)}\widetilde{v}|_\ell^2\Bigr]\,dx.
\end{equation}
Finally, we 
set
\begin{equation*}
    \eta:=\frac{1}{S}, \qquad   w_\eta(x):=\widetilde{w}_{\frac{1}{\eta}}\Bigl(\frac{1}{\eta}x\Bigr) \text{ for all } x\in Q_1^\nu,
\end{equation*}
so that $w_\eta\in \mathcal{A}_\eta^\nu$. By applying the change of variables $y:=\eta x$ on the left-hand side and $y:=\e x$ on the right-hand side, \eqref{propdensity1} can be written as
\begin{equation*}
    \limsup_{\eta\to0}\int_{Q_1^\nu} \Bigl[\frac{1}{\eta}W(w_\eta)+\sum_{\ell=1}^kq_\ell\eta^{2\ell-1}|\nabla^{(\ell)}w_\eta|_\ell^2\Bigr]\,dy \leq \int_{R(Q_1^\nu)} \Bigl[\frac{1}{\e}W(v)+\sum_{\ell=1}^kq_\ell\e^{2\ell-1}|\nabla^{(\ell)}v|_\ell^2\Bigr]\,dy,
\end{equation*}
which is \eqref{claimdensity}. By the arbitrariness of $\e, R,$ and $v$, inequality \eqref{claimdensity} and a diagonal argument yield \eqref{densitytris}.

As for the proof of \eqref{densitybis}, consider $R_1, R_2\in SO(d)$ satisfying $R_1\nu=R_2\nu=\nu$ and $v\circ R_1\in \A_\e^\nu$. With a similar argument to the one employed above, we may prove that for every $\eta\in(0,\e)$, there exists a function $w_\eta$ such that $w_\eta\circ R_2\in \A_\eta^\nu$ and  
\begin{align}\notag
    \limsup _{\eta\to0}\int_{R_2(Q_1^\nu)} \Bigl[\frac{1}{\eta}W(w_\eta)+\sum_{\ell=1}^{k}q_\ell\eta^{2\ell-1} |\nabla^{(\ell)}w_\eta|_\ell^2\Bigr]\,dx \leq \int_{R_1(Q_1^\nu)} \Bigl[\frac{1}{\e}W(v)+\sum_{\ell=1}^{k}q_\ell\e^{2\ell-1} |\nabla^{(\ell)}v|_\ell^2\Bigr]\,dx,    
\end{align}
which, by the arbitrariness of the rotations $R_1, R_2$ and of the function $v$, proves \eqref{densitybis}.

 Finally, we prove that $g$ is upper semicontinuous. Let $\nu \in \Sd$ and $\{\nu_n\}_n \subset \Sd$ such that $\nu_n \to \nu.$ We aim to show that
 \begin{equation*}
 \limsup_{n\to +\infty} g(\nu_n) \le g(\nu).\end{equation*}
 Fix $\eta >0$ and let $\e \in (0,1)$ and $v \in \mathcal{A}_\e^\nu$ such that
 \begin{equation}\label{upper semicontinuity 1}
 g(\nu) + \eta \ge \int_{Q_1^\nu} \Bigl[\frac{1}{\e}W(v)+\sum_{\ell=1}^{k}q_\ell\e^{2\ell-1} |\nabla^{(\ell)}v|_\ell^2\Bigr]\,dx.\end{equation}
 For each $n \in \N$ we can find $R_n = \left(R_n^{i,j}\right)_{i,j=1}^d\in SO(d) $ such that
 $R_n \nu= \nu_n$ and $R_n \to \I.$ Consider the cube $R_n(Q_1^\nu)$ and, for each $n$, a rotation $S_n \in SO(d)$ such that 
 \begin{equation*}
 S_n \nu_n= \nu_n \text{ and } R_n (Q_1^\nu)=S_n (Q_1^ {\nu_n}).
\end{equation*}
Define $v_n(x):=v(R_n^Tx)$ for every $x \in R_n(Q_1^\nu)=S_n(Q_1^{\nu_n})$ and note that $v_n\circ S_n \in \mathcal{A}_\e^{\nu_n}$. By \eqref{densitytris} and a change of variables
 \begin{align}\notag
 g(\nu_n) & \le \int _{S_n(Q_1^{\nu_n})}\Bigl[\frac{1}{\e}W(v_n(x))+\sum_{\ell=1}^{k}q_\ell\e^{2\ell-1} |\nabla^{(\ell)}v_n(x)|_\ell^2\Bigr] \, dx  \\ \notag
 & = \int _{R_n(Q_1^\nu)}\Bigl[\frac{1}{\e}W(v_n(x))+\sum_{\ell=1}^{k}q_\ell\e^{2\ell-1} |\nabla^{(\ell)}v_n( x)|_\ell^2\Bigr] \, dx \\
 & = \int _{Q_1^\nu}\Bigl[\frac{1}{\e}W(v(x))+\sum_{\ell=1}^{k}q_\ell\e^{2\ell-1} |\nabla^{(\ell)}v_{n}( R_n x)|_\ell^2\Bigr] \, dx. \label{upper semicontinuity w}
 \end{align}
 
We show that, for each $\ell \in \{1,...,k\}$ and  $\Ld$-a.e.\ $x \in Q_{1}^\nu$, we have $\nabla^{(\ell)}v_n(R_n x) \to \nabla^{(\ell)}v(x) $ in the space of $\ell$-tensors as $n\to+\infty$. 
By the equivalence of the norms, to prove this it is sufficient to show that for $\Ld$-a.e.\ $x\in Q^\nu_1$  it holds
\begin{equation}\label{convergence of tensors}
\lim_{n\to +\infty} \nabla ^{(\ell)}v_n(R_nx)(e_{i_1},...,e_{i_\ell})=\nabla ^{(\ell)}v(x)(e_{i_1},...,e_{i_\ell})\,\, \text{ for every } (i_1,...,i_\ell) \in \{1,...,d\}^\ell.
\end{equation}
Let us fix $(i_1,...,i_\ell)\in\{1,...,d\}^{\ell}$. For $\Ld$-a.e.\ $x\in Q^\nu_1$ we have
\begin{align*}
    \lim_{n\to +\infty}\nabla ^{(\ell)}v_n(R_nx)(e_{i_1},...,e_{i_\ell}) & = \lim_{n\to +\infty}\frac{\partial^\ell v_n}{\partial x_{i_1}...\partial x_{i_\ell}} (R_nx) \\
    & = \lim_{n\to +\infty}\sum_{j_1=1}^d... \sum_{j_\ell=1}^d R_n^{i_1, j_1}... \ R_n^{i_\ell, j_\ell} \frac{\partial^\ell v}{\partial x_{j_1}...\partial x_{j_\ell}} (x) \\
    & = \sum_{j_1=1}^d... \sum_{j_\ell=1}^d \delta_ {i_1j_1}...\delta_{i_\ell j_\ell} \frac{\partial^\ell v}{\partial x_{j_1}...\partial x_{j_\ell}} (x) \\
    &  =\frac{\partial^\ell v}{\partial x_{i_1}...\partial x_{i_\ell}} (x) 
     = \nabla ^{(\ell)}v(x)(e_{i_1},...,e_{i_\ell}),
\end{align*}
and this equality proves \eqref{convergence of tensors}. 

 To conclude, we first observe that, since $R_n\to\I$, we may assume that $|R_n^{i,j}|\leq 2$ for all $i,j\in\{1,...,d\}$ and $n\in\N$ so that 
\begin{align}\notag
    |\nabla^{(\ell)}v_n(R_n x)|_\ell & \le C\max \Bigl\{ \big|\dfrac{\partial^\ell v}{\partial x_{i_1}...\partial x_{i_\ell}} (R_nx)\big|: \ (i_1,....,i_\ell) \in \{1,...,d\}^\ell \Bigr\} \\\notag
    & = C \max \Bigl\{ \bigl|\sum_{j_1=1}^d... \sum_{j_\ell=1}^d R_n^{i_1, j_1}... \ R_n^{i_\ell, j_\ell} \dfrac{\partial^\ell v}{\partial x_{j_1}...\partial x_{j_\ell}} (x)\bigr|: \ (i_1,....,i_\ell) \in \{1,...,d\}^\ell \Bigr\} \\\ 
    & \le \label{bound tensor norms} C \max \Bigl\{ \big|\dfrac{\partial^\ell v}{\partial x_{i_1}...\partial x_{i_\ell}} (x)\big|: \ (i_1,....,i_\ell) \in \{1,...,d\}^\ell \Bigr\}\leq  C|\nabla^{(\ell)}v(x)|_\ell,
\end{align}
where $C$ is a positive constant (that may very from line to line) which only depends on $d$ and $\ell$ as it arises by the equivalence of the tensor norms. Finally, combining \eqref{convergence of tensors} with \eqref{bound tensor norms}, we apply the Dominated Convergence Theorem in \eqref{upper semicontinuity w}, so that by \eqref{upper semicontinuity 1} we obtain
\begin{align*}
 \limsup_{n \to +\infty }g(\nu_n) &  \le 
 \lim_{n\to +\infty }\int _{Q_{1}^{\nu}}\Bigl[\frac{1}{\e}W(v(x))+\sum_{\ell=1}^{k}q_\ell\e^{2\ell-1} |\nabla^{(\ell)}v_n(R_n x)|_\ell^2\Bigr] \, dx\\
 & =\int _{Q_{1}^{\nu}}\Bigl[\frac{1}{\e}W(v(x))+\sum_{\ell=1}^{k}q_\ell\e^{2\ell-1} |\nabla^{(\ell)}v(x)|_\ell^2\Bigr] \, dx 
 \leq g(\nu)+\eta.
 \end{align*}
 The arbitrariness of $\eta$ concludes the proof.
\end{proof}

\begin{remark}\label{re.inftolim}
 {\rm   With the previous proof, we also showed that 
   \begin{equation*}
       g(\nu)=\lim_{\e\to0}\,\inf\Big\{\int_{{Q_1^\nu}} \Bigl[\frac{1}{\e}W(v)+\sum_{\ell=1}^{k}q_\ell\e^{2\ell-1} |\nabla^{(\ell)}v|_\ell^2\Bigr]\,dx  :  v \in \A^\nu_\e\Big\}.
   \end{equation*}
   This observation will be useful in the last section in order to study the density $g$ in the one-dimensional case.
   }
\end{remark}
        
\section{Interpolation}

This section is devoted to the proof of Theorem \ref{nonlinintd}. Furthermore, we obtain a useful result (Corollary \ref{interpolationwithnorms}) that we will use in the proof of Theorem \ref{compactness}. Throughout this section, we keep track of the dependence of the constants on the involved parameters with the only exceptions of $d$, the dimension of the ambient Euclidean space, and $k$, the highest order of derivation. 

We start by mentioning a classical interpolation estimate on intervals (see Theorem $7.41$ in \cite{Leoni2nd}).  For simplicity of notation, for every interval $I$ we let $|I|$ denote its length.

\begin{theorem}\label{Leoni1d} Let $I\subset \R$ be a bounded, open interval and let $k$
be a positive integer. Then there exists a positive constant $c$ independent of $I$ such that
\begin{equation}\notag
\    \int_I |u^{(k-1)}|^2\,dt \leq c|I|^{-2(k-1)}\int_I u^2\,dt+ c \int_I |u^{(k)}|^2\,dt
\end{equation}
for every $u\in H^k(I)$.
\end{theorem}

We adapt the argument of Lemma $3.1$ in \cite{CSZ} to bound the square of the $L^2$-norm of the derivative of order $k-1$ by means of the integral of the double-well potential  and the square of the $L^2$-norm of the derivative of order $k$. Then, we deduce the same kind of estimate for any derivate of intermediate order $\ell\in\{1,...,k-1\}$. 
\begin{lemma}\label{nonlinint1d}
Let $I\subset \R$ be a bounded, open interval, let $k>1$
be an integer, and assume that {\rm (H2)} is satisfied. Then there exists a positive constant $q'_{k-1}$ independent of $I$ such that
\begin{equation*}
    q\int_I |u^{(k-1)}(t)|^2\,dt \leq |I|^{-2(k-1)}\int_I W(u)\,dt+ |I|^2\int_I |u^{(k)}|^2\,dt
\end{equation*}
for every $q<q'_{k-1}$ and $u\in H^k(I)$.
\end{lemma}
\begin{proof}
    It is sufficient to prove the statement for $I=(0,1)$, the general case being obtained by translating and rescaling.
    We set 
    \begin{equation*}
        z:= \int_0^1 u^{(k-1)}\,dt,
    \end{equation*}
    and by the Fundamental Theorem of Calculus we infer that    \begin{equation}\label{nonlinint1d1}
        |u^{(k-1)}(t)-z|\leq \int_0^1|u^{(k)}|\,dt \qquad \text{for all } t\in (0,1).
    \end{equation}
    As a consequence
    \begin{equation*}
        \int_0^1 |u^{(k-1)}|^2\,dt \leq 2 \int_0^1 |u^{(k)}|^2\,dt + 2z^2;
    \end{equation*}
    therefore, the proof is complete if we prove that 
    \begin{equation}\label{nonlinint1d2}
        z^2 \leq C\int_0^1 \Bigl[W(u)+|u^{(k)}|^2\Bigr]\,dt
    \end{equation} 
    for some $C>0$. If $z^2\leq 4\int_0^1 |u^{(k)}|^2\,dt$, then \eqref{nonlinint1d2} is immediately satisfied. Otherwise, by Jensen's inequality, we have that $|z| > 2\int_0^1 |u^{(k)}|\,dt$, which together with \eqref{nonlinint1d1}, yields 
    \begin{equation}\label{nonlinint1d3}
        |u^{(k-1)}(t)|\geq \frac{|z|}{2}>0 \qquad \text{for all } t\in(0,1).
    \end{equation}
    Consider the intervals $(\frac{j}{k}, \frac{j+1}{k}),\, j\in\{0,...,k-1\}$.  We observe that there exists $j^*\in\{0,...,k-1\}$ such that $u\neq 0$ on $(\frac{j*}{k}, \frac{j*+1}{k})$. To see this, assume by contradiction that $u$ vanishes on each of these intervals; then, $u$ has at least $k$ zeros on $I$. By Lagrange's Theorem, this implies that $u'$ vanishes at $k-1$ points of $I$, and then, iterating this argument, we obtain that $u^{(k-1)}$ vanishes at a point of $I$, contradicting \eqref{nonlinint1d3}. We suppose $u>0$ in $(\frac{j^*}{k},\frac{j^*+1}{k})$, the case $u<0$ being analogous.
    Applying Theorem \ref{Leoni1d} to the function $u-1$ on the interval $(\frac{j^*}{k},\frac{j^*+1}{k})$ and resorting to (H$2$), we obtain
    \begin{align*}
         \int_{\frac{j^*}{k}}^{\frac{j^*+1}{k}} |u^{(k-1)}|^2\,dt & \leq ck^{2(k-1)}\int_{\frac{j^*}{k}}^{\frac{j^*+1}{k}} (u-1)^2\,dt+ c \int_{\frac{j^*}{k}}^{\frac{j^*+1}{k}} |u^{(k)}|^2\,dt \\
         & \leq  ck^{2(k-1)}\int_{\frac{j^*}{k}}^{\frac{j^*+1}{k}} W(u)\,dt+ c \int_{\frac{j^*}{k}}^{\frac{j^*+1}{k}} |u^{(k)}|^2\,dt,
    \end{align*}
    therefore, taking into account \eqref{nonlinint1d3},
    \begin{align*}
         z^2 & \leq 4k\int_{\frac{j^*}{k}}^{\frac{j^*+1}{k}} |u^{(k-1)}|^2\,dt \\
         & \leq 4ck^{2k-1}\int_{\frac{j^*}{k}}^{\frac{j^*+1}{k}} W(u)\,dt+ 4ck \int_{\frac{j^*}{k}}^{\frac{j^*+1}{k}} |u^{(k)}|^2\,dt,
    \end{align*}
    which, by the positivity of $W$, proves \eqref{nonlinint1d2} and concludes the proof.
\end{proof}

\begin{corollary}\label{nonlinint1dell}
Let $I\subset \R$ be a bounded, open interval, let $k>1$
be an integer, and assume that {\rm (H2)} is satisfied. Then, for every $\ell\in\{1,...,k-1\}$, there exists a positive constant $q'_{\ell}$ independent of $I$ such that
\begin{equation*}
    q\int_I |u^{(\ell)}|^2\,dt \leq |I|^{-2\ell}\int_I W(u)\,dt+ |I|^{2(k-\ell)}\int_I |u^{(k)}|^2\,dt
\end{equation*}
for every $q<q'_{\ell}$ and $u\in H^k(I)$.
\end{corollary}
\begin{proof}
    The thesis follows by applying iteratively Lemma \ref{nonlinint1d}.
\end{proof}
Now we adapt the interpolation inequality at scale $\e$, in view of its application to functionals \eqref{functionals}.

\begin{proposition} \label{nonlininteps}
Let $I\subset \R$ be a bounded, open interval, let $k>1$
be an integer, and assume that {\rm (H2)} is satisfied. Then, for every $\ell\in\{1,...,k-1\}$, there exists a positive constant $q''_{\ell}$ independent of $I$ such that, for every $\e\in(0,|I|/2)$, it holds
\begin{equation*}
    q\e^{2\ell}\int_I |u^{(\ell)}|^2\,dt \leq \int_I \Bigl[W(u)+ \e^{2k}|u^{(k)}|^2\Bigr]\,dt
\end{equation*}
for every $q<q''_{\ell}$ and $u\in H^k(I)$.
\end{proposition}
\begin{proof}
We set $v(t):=u(\e t)$ for $t\in\frac{1}{\e}I$ so that
\begin{equation*}
    \e^{2\ell-1} \int_I |u^{(\ell)}|^2\,dt = \int_\frac{I}{\e} |v^{(\ell)}|^2\,dt. 
\end{equation*}
Then, we set $n_\e:=\lfloor\frac{|I|}{\e}\rfloor$ and we subdivide the rescaled interval $\frac{1}{\e}I$ in $n_\e$ pairwise disjoint, open intervals $I_\e^j,\, j\in\{1,...,n_\e\}$ each of length $\frac{|I|}{\e n_\e}$.  Then we apply Corollary \ref{nonlinint1dell} on each subinterval $I_\e^j$ with $q<q'_\ell$ and we obtain 
\begin{align*}
    \e^{2\ell-1} \int_I |u^{(\ell)}|^2\,dt & = \sum_{j=1}^{n_\e}  \int_{I_\e^j} |v^{(\ell)}|^2\,dt \\
    &\leq \sum_{j=1}^{n_\e} \frac{1}{q}\Bigl\{|I_\e^j|^{-2\ell} \int_{I_\e^j} W(v)\,dt + |I_\e^j|^{2(k-\ell)}\int_{I_\e^j} |v^{(k)}|^2\,dt\Bigr\} \\
    & \leq \frac{1}{q}\Bigl\{2^{2\ell}\int_{\frac{1}{\e}I} W(v)\,dt + \Bigl(\frac{3}{2}\Bigr)^{2(k-\ell)} \int_{\frac{1}{\e}I} |v^{(k)}|^2\,dt \Bigr\}\\
    & \leq \frac{\max\{2^{2\ell}, (3/2)^{2(k-\ell)}\}}{q} \int_I \Bigl[\frac{W(u)}{\e}+\e^{2k-1}|u^{(k)}|^2\Bigr]\,dt,
\end{align*}
where the second inequality follows by the fact $|I_\e^j|\in (\frac{1}{2}, \frac{3}{2})$ if $\e\in(0,|I|/2)$.

\noindent We multiply the above inequality by $\e$ to obtain that
\begin{equation*}
    \widetilde{q}\e^{2\ell}\int_I |u^{(\ell)}|^2\,dt \leq \int_I \Bigl[W(u)+\e^{2k}|u^{(k)}|^2\Bigr]\,dt
\end{equation*}
for every $\widetilde{q}< q''_\ell$ and $\e\in(0,|I|/2)$, where we set
\begin{equation*}
    q''_\ell:=\frac{q'_\ell}{\max\{2^{2\ell}, (3/2)^{2(k-\ell)}\}};
\end{equation*}
therefore, the proof is complete.
\end{proof}

The following immediate corollary proves Theorem \ref{nonlinintd} in dimension $1$.
\begin{corollary}\label{nonlininepsgen}
Let $\Omega\subset \R$ be the union of finitely many bounded, open, disjoint intervals, let $k>1$
be an integer, and assume that {\rm (H2)} is satisfied. Then, denoting by $|I_o|$ the length of the shortest connected component of $\Omega$, for every $\ell\in\{1,...,k-1\}$ and for every $\e\in(0,|I_o|/2)$ it holds
\begin{equation*}
    q\e^{2\ell}\int_\Omega |u^{(\ell)}|^2\,dt \leq \int_\Omega \Bigl[W(u)+ \e^{2k}|u^{(k)}|^2\Bigr]\,dt
\end{equation*}
for every $q<q''_{\ell}$ and $u\in H^k(\Omega)$.    
\end{corollary}
\begin{proof}
    The proof follows by an immediate application of Proposition \ref{nonlininteps}.
\end{proof}

Finally, we prove the main result of this section by using the argument presented in the proof of Theorem $1.2$ in \cite{ChermisiDalMaso11}. Below, we will make use of notation about slicing that has been introduced in Section $2$.

\begin{proof} [Proof of Theorem \ref{nonlinintd}]
It suffices to prove the part of the statement concerning the operatorial norms, since the other case follows by the equivalence of the norms. We follow the proof of Theorem $1.2$ in \cite{ChermisiDalMaso11}. For $r>0,\,\xi\in\Sd$, and $y\in \Omega^\xi$, we define $\Omega^{\xi,y}_r$ as the finite union of the connected components of the slice $\Omega^{\xi,y}$ having length larger than $r$. Then we set $\Omega(\xi,r):=\{y+t\xi\in \R^d : y\in\Omega^\xi,\, t\in \Omega^{\xi,y}_r\}$ and we observe that $(\Omega(\xi,r))^{\xi,y}=\Omega^{\xi,y}_r$.  

Applying Corollary \ref{nonlininepsgen} we obtain
    \begin{equation*}
        \widetilde{q}\e^{2\ell}\int_{\Omega^{\xi,y}_r} |(u^{\xi,y})^{(\ell)}|^2\,dt \leq \int_{\Omega^{\xi,y}_r} \Bigl[W(u^{\xi,y})+ \e^{2k}|(u^{\xi,y})^{(k)}|^2\Bigr]\,dt
    \end{equation*}
    for all $\widetilde{q}<q''_\ell$ and $\e< r/2$. We integrate the above inequality on $\xi^\perp$ to get
    \begin{equation*}
        \widetilde{q}\e^{2\ell}\int_{\Omega(\xi,r)} |(\nabla^{(\ell)}u)(\xi,...,\xi)|^2\,dx \leq \int_\Omega \Bigl[W(u)+ \e^{2k}\|\nabla^{(k)}u\|_k^2\Bigr]\,dx,
    \end{equation*}
    and setting $A(\xi,r):= \Omega \setminus\Omega(\xi,r)$ we obtain
\begin{equation*}
        \widetilde{q}\e^{2\ell}\int_{\Omega} |(\nabla^{(\ell)}u)(\xi,...,\xi)|^2\,dx  -\widetilde{q}\e^{2\ell}\int_{A(\xi,r)} \|\nabla^{(\ell)}u\|_\ell^2\,dx\leq \int_\Omega \Bigl[W(u)+ \e^{2k}\|\nabla^{(k)}u\|_k^2\Bigr]\,dx.
    \end{equation*}
    Averaging this inequality on $\Sd$ yields
    \begin{align*}
        \e^{2\ell}\frac{\widetilde{q}}{\sigma_{d-1}}\int_{\Sd}\int_{\Omega} |(\nabla^{(\ell)}u)(\xi,...,\xi)|^2\,dx\,d\Hd(\xi) & - \e^{2\ell}\frac{\widetilde{q}}{\sigma_{d-1}}\int_{\Sd}\int_{A(\xi,r)} \|\nabla^{(\ell)}u\|^2\,dx\,d\Hd(\xi) \\
        & \leq \int_\Omega \Bigl[W(u)+ \e^{2k}\|\nabla^{(k)}u\|_k^2\Bigr]\,dx,
    \end{align*}
    where we set $\sigma_{d-1}:=\Hd(\Sd)$.

    By a simple compactness argument, there exists a positive constant $\tau(\ell)$ such that
    \begin{equation*}
        \frac{1}{\sigma_{d-1}}\int_{\Sd} |T(\xi,...,\xi)|^2\,d\Hd(\xi) \geq \tau(\ell)\|T\|^2_\ell
    \end{equation*}
    for every symmetric $\ell$-tensor $T$. Therefore, by Fubini's Theorem, 
    \begin{equation}\label{nonlininepsgen1}
        \frac{1}{\sigma_{d-1}}\int_{\Sd}\int_{\Omega} |(\nabla^{(\ell)}u)(\xi,...,\xi)|^2\,dx\,d\Hd(\xi) \geq \tau(\ell) \int_\Omega \|\nabla^{(\ell)} u\|_\ell^2\,dx.
    \end{equation}

\noindent Again by Fubini's Theorem, we have 
\begin{equation*} 
        \frac{1}{\sigma_{d-1}}\int_{\Sd}\int_{A(\xi,r)} \|\nabla^{(\ell)}u\|_\ell^2\,dx\,d\Hd(\xi) = \frac{1}{\sigma_{d-1}} \int_{\Omega} \|\nabla^{(\ell)}u\|_\ell^2\Hd(D(x,r))\,dx,
    \end{equation*}
where $D(x,r):=\{\xi\in \Sd: x\in A(\xi,r)$\}; hence, to conclude the proof, we have to select $r$ small enough so that $\Hd(D(x,r))$ is small uniformly in $x$. Given $x\in \partial \Omega$, we let $C_R(x)$ denote the cylinder having height $2R$ and radius $R$ with axis parallel to the normal $\nu(x)$ to $\partial \Omega$ at $x$.

\noindent Let $q<\tau(\ell)q''_\ell$, and fix $\widetilde{q}$ such that $\frac{q}{\tau(\ell)}<\widetilde{q}< q''_\ell$. Fix $\eta=\eta(q)>0$ such that 
    \begin{equation}\label{choiceq}
        q< \widetilde{q}(\tau(\ell)-\eta);
    \end{equation}
since $\Omega$ has $C^1$ boundary, there exists $R$ such that, for every $x\in \partial \Omega$, it holds $C_R(x)\cap \partial\Omega$ is the graph of a $C^1$ function with respect to the base of the cylinder, and moreover,
\begin{equation*}
        \Hd(\{\xi\in\Sd:\xi\in T_{\partial\Omega}(y),\, y\in C_R(x)\cap \partial\Omega\}) \leq \eta\sigma_{d-1}.
    \end{equation*}
Employing the same argument presented in the proof of Theorem $1.2$ in \cite{ChermisiDalMaso11}, it holds that,
\begin{equation*}
    D(x,r)\subseteq \{\xi\in\Sd:\xi\in T_{\partial\Omega}(y),\, y\in C_R(x)\cap \partial\Omega \}
\end{equation*}
up to choosing $r<R/2$; as a consequence, $\Hd(D(x,r))\leq \eta$ for all $x\in\partial\Omega$ if $r<R/2$, and then
\begin{equation}\label{nonlininepsgen2} 
        \frac{1}{\sigma_{d-1}}\int_{\Sd}\int_{A(\xi,r)} \|\nabla^{(\ell)}u\|_\ell^2\,dx\,d\Hd(\xi) \leq  \eta \int_{\Omega} \|\nabla^{(\ell)}u\|_\ell^2\,dx,
    \end{equation}
for all $r<R/2$.

\noindent Gathering \eqref{nonlininepsgen1} and \eqref{nonlininepsgen2}, we obtain that
    \begin{equation}\label{nonlininepsgen3}
        \widetilde{q}(\tau(\ell)-\eta)\e^{2\ell}\int_{\Omega} \|\nabla^{(\ell)}u\|_\ell^2\,dx \leq \int_\Omega \Bigl[W(u)+ \e^{2k}\|\nabla^{(k)}u\|_k^2\Bigr]\,dx
    \end{equation}
    for all $\widetilde{q}<q''_\ell$ and $\e< R/2$, where $R=R(\eta,\Omega)$.
    
\noindent Recalling \eqref{choiceq}, by \eqref{nonlininepsgen3} we get
    \begin{equation*}
        q\e^{2\ell}\int_{\Omega} \|\nabla^{(\ell)}u\|_\ell^2\,dx \leq \int_\Omega \Bigl[W(u)+ \e^{2k}\|\nabla^{(k)}u\|_k^2\Bigr]\,dx.
    \end{equation*}
    The above inequality is valid for all $q<\tau(\ell)q''_\ell$ and $\e< R(\eta, \Omega)/2$.
    Then, the thesis holds with $q^*_\ell:=\tau(\ell)q''_\ell$ and $\e_0:=R(\eta, \Omega)/2$, which in turn, since $\eta$ depends on $q$, only depends on $q$ and $\Omega$.
 \end{proof}

\begin{remark}\label{re:uniform interpolation}
    {\rm We are not able to prove the interpolation inequality on an open cube. The same slicing arguments employed in \cite{ChermisiDalMaso11} cannot be applied in our case because of the high order of the derivatives that we want to estimate. This complicates our arguments in view of the proof of Theorem \ref{main theorem} since, in what follows, it will be important to apply Theorem \ref{nonlinintd} on a class of $C^1$ sets $\{\Omega_{\rho}:\rho\in(\rho_1,\rho_2)\}$ for some $\rho_1,\rho_2>0$ for which the threshold $\e_0=\e_0(q, \Omega_\rho)$ can be chosen independently of $\rho\in (\rho_1,\rho_2)$. Inspecting the proof of Theorem \ref{nonlinintd}, this amounts to prove that for each properly fixed $q$, given $\eta>0$ such that
    \begin{equation*}
        q< \widetilde{q}(\tau(\ell)-\eta),
    \end{equation*}
    there exists $R>0$ such that, for every $\rho\in(\rho_1,\rho_2)$ and for every $x\in \partial \Omega_\rho$, it holds that
    \begin{equation}\label{graphcondition}
        C_R(x)\cap \partial\Omega_\rho\,\,\text{ is the graph of a } C^{1} \text{ function}    
        \end{equation}
    and
    \begin{equation}\label{uniforminterpolation}
        \Hd(\{\xi\in\Sd:\xi\in T_{\partial\Omega_\rho}(y),\, y\in C_R(x)\cap \partial\Omega_\rho \}) \leq \eta\Hd(\Sd).
    \end{equation}
We claim that the above conditions are satisfied if we consider $\Omega$ a bounded, open set with $C^{2}$ boundary, and  
    \begin{equation*}
        \Omega_\rho:=\{x\in \Omega : \text{dist}(x,\partial\Omega)<\rho\}
    \end{equation*}
    for $\rho\in(\rho^*,\rho^*+R^*)$, where $\rho^*>0$ is fixed and $R^*=R^*(\rho^*)>0$. In particular,  we fix $\rho^*$ small enough so that $\{x\in\Omega : \text{dist}(x,\partial \Omega)=\rho\}$ is a $(d-1)$-manifold of class $C^1$ and there exists a $C^1$ projection from $\Omega_{\rho}$ onto $\partial \Omega$ for every $\rho\leq 2\rho^*$.  
    
    Let $\eta '=\eta '(\eta,d)$ be such that 
    \begin{equation}\label{etaprimo}
        \Hd\Bigl(\Bigl\{ \xi\in \Sd : |\xi_d| \leq \eta ' \Bigr\}\Bigr) \leq \eta \Hd(\Sd).
    \end{equation}
Since $\partial \Omega$ is of class $C^1$, for every $x\in \partial \Omega$ there exist $R>0$ and a $C^1$ function $\varphi$ defined on the base of $C_R(x)$ such that 
\begin{equation*}
        C_R(x)\cap\partial \Omega = C_R(x)\cap \text{graph }\varphi .
    \end{equation*}
By a rotation argument, we can further assume that $\nabla \varphi(x)=0$, and upon taking a smaller $R$, that 
\begin{equation*}
    \|\nabla \varphi\|_{L^\infty(C_R(x))}\leq \eta'.
\end{equation*}
Moreover, since $\partial\Omega$ is compact, we can select $R$ as above independent of $x\in\partial \Omega$.   

Since $\partial\Omega_{\rho^*}$ is of class $C^1$ and compact, using the same argument, we find a possibly smaller $R$ with the property that for all $x\in\partial\Omega_{\rho^*}\cap \Omega$ there exists a $C^1$ function $\varphi_{\rho^*}$ defined on the base of $C_R(x)$ such that 
    \begin{equation*}
        C_R(x)\cap\partial \Omega_{\rho^*} = C_R(x)\cap \text{graph }\varphi_{\rho^*}  \quad \text{ with } \quad  \|\nabla \varphi_{\rho^*}\|_{L^\infty(C_R(x))}\leq \eta'.
    \end{equation*}
Finally, we assume without loss of generality that $R<\rho^*$.

Now, let $x\in\partial \Omega_{\rho^*}\cap \Omega$ and $0<r<R/4$ and note that
    \begin{equation*}
         \{y\in C_{R-r}(x) : \text{dist}(y, \text{ graph }\varphi_{\rho^*})=r\} = C_{R-r}(x)\cap \partial \Omega_{\rho^*+r}.
    \end{equation*}
Then, letting $x_r\in\partial\Omega_{\rho^*+r}$ denote the unique point such that $|x-x_r|=r$, we have that $C_{R-2r}(x_r)\subset C_{R-r}(x)$, which implies
\begin{equation*}
         \{y\in C_{R-2r}(x_r) : \text{dist}(y, \text{ graph }\varphi_{\rho^*})=r\} = C_{R-2r}(x_r)\cap \partial \Omega_{\rho^*+r}.
    \end{equation*}
We also note that there exists a $C^1$ function $\varphi_{\rho^*+r}$ such that 
    \begin{equation*}
        \{y\in C_{R-2r}(x_r) : \text{dist}(y, \text{ graph }\varphi_{\rho^*})=r\} =C_{R-2r}(x_r) \cap \text{ graph }\varphi_{\rho^*+r},
    \end{equation*}
    and
    \begin{equation*}
    \|\nabla \varphi_{\rho^*+r}\|_{L^\infty(C_{R-2r}(x_r))} \leq \|\nabla \varphi_{\rho^*}\|_{L^\infty(C_{R-2r}(x_r))} \leq \|\nabla \varphi_{\rho^*}\|_{L^\infty(C_{R}(x))}. 
    \end{equation*}
Recalling that there exists a $C^1$ projection from $\Omega_{\rho}$ onto $\partial \Omega$  for every $\rho\in(\rho^*, \rho^*+R/4)$, the above observations prove that for every $\rho\in(\rho^*, \rho^*+R/4)$ and for every $x\in\partial\Omega_\rho$, it holds that
    \begin{equation*}
        C_{R-2r}(x)\cap\partial\Omega_\rho = C_{R-2r}(x)\cap \text{graph }\varphi_\rho,
    \end{equation*}
    where $\varphi_\rho$ is a $C^1$ function such that 
    \begin{equation*}
    \|\nabla \varphi_\rho\|_{L^\infty(C_{R-2r}(x))} \leq \|\nabla \varphi_{\rho^*}\|_{L^\infty(C_{R-2r}(x))} \leq\eta'. 
    \end{equation*}
As $r$ ranges in $(0, R/4)$, in particular we obtain
 \begin{equation}\label{graphs}
        C_{R/2}(x)\cap\partial\Omega_\rho = C_{R/2}(x)\cap \text{graph }\varphi_\rho
    \end{equation}
    and 
    \begin{equation}\label{lipbound3}
    \|\nabla \varphi_\rho\|_{L^\infty(C_{R/2}(x))} \leq \|\nabla \varphi_{\rho^*}\|_{L^\infty(C_{R/2}(x))} \leq \eta'. 
    \end{equation}
\noindent At this point, we conclude that our claim is proved with $R^*=R/4$. Indeed, \eqref{graphcondition}
follows by \eqref{graphs}, and we are left with proving that \eqref{lipbound3} implies \eqref{uniforminterpolation}. To see this, note that for fixed $\rho\in(\rho^*, \rho^*+R/4)$ and $x\in \partial \Omega_\rho$, \eqref{lipbound3} implies that 
\begin{align*}
    & \hspace{-1 cm}\{\xi\in\Sd:\xi\in T_{\partial\Omega_\rho}(y),\, y\in C_{R/2}(x)\cap \partial\Omega_\rho \} \\
    & \subseteq \Bigl\{\xi\in\Sd: |\xi_d|=|\nabla\varphi_\rho(y_1,...,y_{d-1})\cdot(\xi_1,...,\xi_{d-1})|,\,  y\in C_{R/2}(x)\cap \partial\Omega_{\rho} \Bigr\} \\
    & \subseteq \Bigl\{\xi\in \Sd : |\xi_d|\leq \eta'\Bigr\};
\end{align*}
therefore, by \eqref{etaprimo} we conclude that
\begin{equation*}
    \Hd(\{\xi\in\Sd:\xi\in T_{\partial\Omega_\rho}(y),\, y\in C_{R/2}(x)\cap \partial\Omega_\rho \}) \leq \eta\Hd(\Sd).
\end{equation*}
Being the above estimates uniform in $\rho$ and $x$, \eqref{uniforminterpolation} is satisfied. 
    }
\end{remark}

The following corollary will come into play in the forthcoming sections. 
 \begin{corollary}\label{interpolationwithnorms}  Let $\Omega\subset \R^d$ be a bounded, open set with $C^1$ boundary, let $k>1$ be an integer, and let $\mathcal{N}:=\{\ell\in\{1,...,k-1\}: q_\ell \leq 0\}$. Assume that $\mathcal{N}\neq \emptyset$ and that {\rm (H2)} is satisfied. There exist positive constants $\{\overline{q}_\ell : \ell\in\mathcal{N}\}$ independent of $\Omega$ such that if $q_\ell>-\overline{q}_\ell$ for every $\ell\in\mathcal{N}$ there exist positive constants $\e_0=\e_0(\{q_\ell : \ell\in\mathcal{N}\},\Omega)$ and $\delta=\delta(\{q_\ell : \ell\in\mathcal{N}\})$ such that it holds
\begin{equation*}
     \int_\Omega \Bigl[\frac{1}{\e}W(u)+\sum_{\ell=1}^kq_\ell\e^{2\ell-1}|\nabla^{(\ell)}u|_\ell^2\Bigr]\,dx \geq \delta \int_\Omega \Bigl[\frac{1}{\e}W(u)+\sum_{\ell=1}^k\e^{2\ell-1}|\nabla^{(\ell)}u|_\ell^2\Bigr]\,dx
\end{equation*}
for every $\e\in(0,\e_0)$ and $u\in H^k(\Omega)$. In particular, letting $\widetilde{q}^*_1,...,\widetilde{q}^*_{k-1}$ denote the same positive constants appearing in Theorem \ref{nonlinintd} and given $\{\alpha_\ell : \ell\in \mathcal{N}\}$ such that $\alpha_\ell>0$ for all $\ell\in\mathcal{N}$ and $\sum_{\ell\in\mathcal{N}}\alpha_\ell=1$, it holds that $\overline{q}_\ell=\alpha_\ell \widetilde{q}^*_\ell$ for all $\ell\in\mathcal{N}$.
\end{corollary}
 \begin{proof}
 By Theorem \ref{nonlinintd}, for every $\ell\in\{1,...,k-1\}$ there exists a positive constant $\widetilde{q}^*_\ell$ independent of $\Omega$ such that for every $q > - \widetilde{q}^*_\ell$ there exists $\e_0=\e_0(q,\Omega)$ such that it holds
\begin{equation}\label{interpolaitonfunctional}
    \int_\Omega \Bigl[\frac{1}{\e}W(u) + \e^{2k-1}|\nabla^{(k)}u|_k^2\Bigr]\,dx+q\int_\Omega \e^{2\ell-1}|\nabla^{(\ell)}u|_\ell^2\,dx \geq 0
\end{equation}
for every $\e\in(0,\e_0)$ and $u\in H^k(\Omega)$.

Let $\mathcal{N}:=\{\ell\in\{1,...,k-1\}: q_\ell \leq 0\}$ and consider $\{\alpha_\ell : \ell\in \mathcal{N}\}$ such that $\alpha_\ell>0$ for all $\ell\in\mathcal{N}$ and $\sum_{\ell\in\mathcal{N}}\alpha_\ell=1$. Suppose that $q_\ell>-\alpha_\ell\widetilde{q}^*_\ell$ for every $\ell\in\mathcal{N}$ and let $\delta$ be a positive real number smaller than $1$ such that
\begin{equation}\label{conditiondelta}
    \frac{1}{\alpha_\ell}\frac{q_\ell-\delta}{1-\delta} > - \widetilde{q}^*_\ell
\end{equation}
for all $\ell\in\mathcal{N}$. Note that
\begin{align*}
    & \alpha_\ell\frac{W(u)}{\e}+q_\ell\e^{2\ell-1}|\nabla^{(\ell)}u|_\ell + \alpha_   \ell\e^{2k-1} |\nabla^{(k)}u|_k \\
    & = (1-\delta)\Bigl(\alpha_\ell\frac{W(u)}{\e}+\frac{q_\ell-\delta}{1-\delta}\e^{2\ell-1}|\nabla^{(\ell)}u|_\ell + \alpha_\ell\e^{2k-1} |\nabla^{(k)}u|_k \Bigr) \\
    & \,\,\, \quad + \delta\Bigl(\alpha_\ell\frac{W(u)}{\e}+\e^{2\ell-1}|\nabla^{(\ell)}u|_\ell + \alpha_\ell\e^{2k-1}|\nabla^{(k)}u|_k \Bigr).
\end{align*}
We integrate the above equality and, taking into account \eqref{conditiondelta}, we apply \eqref{interpolaitonfunctional} with $q=\frac{1}{\alpha_\ell}\frac{q_\ell-\delta}{1-\delta} $ to obtain
\begin{align} \notag
    & \int_{\Omega}\Bigl[\alpha_\ell\frac{W(u)}{\e}+q_\ell\e^{2\ell-1}|\nabla^{(\ell)}u|_\ell + \alpha_\ell\e^{2k-1} |\nabla^{(k)}u|_k\Bigr]\,dx \\ \label{interpolationwithnorms1}
    & \geq \delta \int_{\Omega}\Bigl[\alpha_\ell\frac{W(u)}{\e}+\e^{2\ell-1}|\nabla^{(\ell)}u|_\ell + \alpha_\ell\e^{2k-1} |\nabla^{(k)}u|_k\Bigr]\,dx
\end{align}
for every $\ell\in\mathcal{N},\, \e<\e_0(q_\ell,\Omega)$.

Clearly, for $\ell\notin \mathcal{N}$ the coefficient $q_\ell$ is strictly positive; hence it holds 
\begin{equation}\label{interpolationwithnorms2}
    \int_\Omega q_\ell\e^{2\ell-1} |\nabla^{(\ell)} u|^2_\ell\, dx \geq \delta \int_\Omega q_\ell\e^{2\ell-1} |\nabla^{(\ell)} u|^2_\ell\, dx
\end{equation}
since $\delta<1$. Therefore, summing over $\ell\in\{1,...,k-1\}$ and gathering \eqref{interpolationwithnorms1} and \eqref{interpolationwithnorms2}, upon considering a smaller $\e_0$ depending on $\{q_\ell:\ell\in\mathcal{N}\}$ and $\Omega$, we get
\begin{equation*}
     \int_\Omega \Bigl[\frac{W(u)}{\e}+\sum_{\ell=1}^kq_\ell\e^{2\ell-1}|\nabla^{(\ell)}u|_\ell^2\Bigr]\,dx  \geq \delta \int_\Omega \Bigl[\frac{W(u)}{\e}+\sum_{\ell=1}^k\e^{2\ell-1}|\nabla^{(\ell)}u|_\ell^2\Bigr]\,dx,
\end{equation*}
for every $\e<\e_0$, completing the proof.
\end{proof}

\begin{remark}\label{remarkcorollary}
  {\rm The above statement clearly holds (for every $\e$) if $\mathcal{N}= \emptyset$ with $\delta:=\min\{q_\ell:\ell\in\{1,...,k\}\}$, which is strictly positive.}
\end{remark} 
    
 \section{Compactness}

In this section we prove Theorem \ref{compactness}. First, we summarize in a single statement the main results contained in \cite{BDS}. 

\begin{theorem}\label{BDS1dim} Let $\{E^1_\e\}_\e$ be defined as in \eqref{functionalBDS1d} on a bounded, open interval $I\subset \R$, let $\{\e_n\}_n$ be a positive sequence converging to $0$, and assume that $W:\R\to[0,+\infty)$ is a continuous function such that the following conditions are satisfied:
\begin{itemize}
    \item[{\rm (H1)}] $W(s)=0$ if and only if $s\in\{-1,1\}$;
    \item[{\rm (${\rm \widetilde{H}}$2)}] $W(s)\geq \alpha_W\min\{(s+1)^2, (s-1)^2,\beta_W \}$ for every $s\in\R$ and for some $\alpha_W,\, \beta_W$ positive.
\end{itemize} 
If $\{u_n\}_n\subset H^k(I)$ is a sequence such that $\sup_n E^1_{\e_n}(u_n, I)<+\infty$, then there exist a subsequence $\{u_{n_j}\}_j$ and a function $u\in BV(I ;\{-1,1\})$ such that $u_{n_j}\to u$ in measure as $j\to+\infty$. Moreover, the family of functionals $\{E^1_\e\}_\e$ $\Gamma$-converges with respect to the convergence in measure to the functional $E_0^1$ given by
     \begin{equation}\label{BDS1dimlimit}
         E_0^1(u,I):= \begin{cases} 
 m_k \#(S(u) \cap I)  & \text{ if } u\in BV(I;\{-1,1\}), \\
     +\infty & \text{ if } u\in L^2(I)\setminus BV(I;\{-1,1\}),
\end{cases}
     \end{equation}
     where $m_k$ is given by \eqref{minBDS}.
\end{theorem}

By strengthening the assumptions on $W$, we improve the result concerning compactness stated above. 

\begin{corollary}\label{BDS1dimcorollary} 
    Let $\{E^1_\e\}_\e$ be defined as in \eqref{functionalBDS1d} on a bounded, open set $I\subset \R$, let $\{\e_n\}_n$ be a positive sequence converging to $0$, and assume that {\rm (H1)} and {\rm(H2)} are satisfied.
If $\{u_n\}_n\subset H^k(I)$ is a sequence such that $\sup_n E^1_{\e_n}(u_n, I)<+\infty$, then there exist a subsequence $\{u_{n_j}\}_j$ and a function $u\in BV(I ;\{-1,1\})$ such that $u_{n_j}\to u$ in $L^2(I)$ as $j\to+\infty$. Moreover, the family of functionals $\{E^1_\e\}_\e$ $\Gamma$-converges with respect to the convergence in $L^2(I)$ to the functional $E_0^1$ defined in \eqref{BDS1dimlimit}.
\end{corollary}

\begin{proof}
    It suffices to prove the statement in the case that $I$ be an interval. Let $\{u_n\}_n\subset H^k(I)$ be a sequence such that $\sup_n E^1_{\e_n}(u_n, I)<+\infty$. Since (H$2$) implies $\rm (\widetilde{H}2)$, we apply Theorem \ref{BDS1dim} to deduce that, up to subsequences, $u_n\to u$ in measure for some $u\in BV(I;\{-1,1\})$. Let $A\subset I$ be measurable, by (H$2$) we have
    \begin{equation*}
\int_{A} |u_n|^2\,dt = \int_{A\cap \{|u_n|\leq 1\}} |u_n|^2\,dt + \int_{A\cap \{|u_n| > 1\}} |u_n|^2\,dt \leq \mathcal{L}^1(A)+2\int_A \bigl[W(u_n)+1\bigr]\,dt,
\end{equation*}
and since $\sup_n E^1_{\e_n}(u_n, I)<+\infty$, it holds that $\int_I W(u_n)\,dt\to0$ as $n\to+\infty$. Hence, the sequence $\{|u_n|^2\}_n$ is equi-integrable, which, together with its convergence in measure, implies $u_n\to u$ in $L^2(I)$. 

As for the $\Gamma$-limit, the lower bound follows again by Theorem \ref{BDS1dim}, while a recovery sequence for the $L^2(I)$-convergence is obtained exactly as in Proposition $5.3$ in \cite{BDS}.
\end{proof}

At this point, we prove an analogous compactness result for functionals \eqref{functionalBDS}, replacing $I$ with $\Omega\subset \R^d,\,d>1$. This is possible by applying a compactness criterion by slicing that appears in \cite{ABS}. We introduce some notation specific for its statement: here, $\Omega$ denotes a $2$-dimensional rectangle which is the product of the bounded, open intervals $I,J$, and a point $x\in\Omega$ is denoted by $(y,z)$, with $y\in I, z\in J$. Given a function $u$ defined on $\Omega$, for every $y\in I$ we let $u^y$ denote the function $u^y(z):=u(y,z)$ for every $z\in J$, and for every $z\in J$ we let $u^z$ denote the function $u^z(y):=u(y,z)$ for every $y\in I$. We say that two sequences $\{u_n\}_n$ and $\{v_n\}_n$ are $\delta$-close if $\|u_n-v_n\|_{L^1(\Omega)}\leq \delta$.
 
\begin{proposition}\label{slicingcompactness}
Assume that a sequence $\{u_n\}_n\subset L^1(\Omega)$ is equi-integrable, and that for every $\delta>0$ there exist sequences $\{v_n\}_n$ and $\{w_n\}_n$ $\delta$-close to $\{u_n\}_n$, and such that $\{v^y_n\}_n$ is precompact in $L^1(J)$ for a.e.\ $y \in I$, and $\{w^z_n\}_n$ is precompact in $L^1(I)$ for a.e.\ $z \in J$. Then $\{u_n\}_n$ is precompact in $L^1(\Omega)$.
\end{proposition}

We employ this result exactly as in \cite{FM}. For the convenience of the reader, we sketch the proof, and for simplicity, we consider only the case $d=2$ when needed.

\begin{proposition}\label{compactnessBDSanydim} Let $\{E_\e\}_\e$ be defined as in \eqref{functionalBDS}, let $\{\e_n\}_n$ be a positive sequence converging to $0$, and assume that {\rm (H1)} and {\rm (H2)} are satisfied. If $\{u_n\}_n\subset H^k(\Omega)$ is a sequence such that $\sup_n E_{\e_n}(u_n, \Omega)<+\infty$, then there exist a subsequence $\{u_{n_j}\}_j$ and a function $u\in BV(\Omega ;\{-1,1\})$ such that, $u_{n_j}\to u$ in $L^2(\Omega)$ as $j\to+\infty$.        
\end{proposition}

\begin{proof}
For simplicity, in the first part of the proof we suppose $d=2$ and that $\Omega$ be the rectangle $I\times J$, otherwise, it suffices to cover a general $\Omega$ by essentially disjoint rectangles in order to reduce the proof to our simpler case. 
Using the notation previously introduced, we recall that $\{u^y_n\}_n\subset H^k(J)$ for a.e.\ $y\in I$, and, analogously, that $\{u^z_n\}_n\subset H^k(I)$ for a.e.\ $z\in J$. In particular, it holds 
\begin{equation*}
\frac{d^ku_n^y}{dz^k}(z) = \frac{\partial^k u_n}{\partial z^k}(x), \qquad \frac{d^ku_n^z}{dy^k}(y)=\frac{\partial^k u_n}{\partial y^k}(x)
\end{equation*}
for a.e.\ $x=(y,z)\in \Omega$ and for all $n$. As a consequence, by the definition of the operatorial norm, we have $\|\nabla^{(k)}u_n\|_k\geq \max\Bigl\{\Bigl|\displaystyle\frac{d^k u_n^y}{dz^k}\Bigr|,\Bigl|\displaystyle\frac{d^k u_n^z}{dy^k}\Bigr|\Bigr\}$, and then 
\begin{equation}\label{abs1}
E_\e(u_n, \Omega)\geq \int_I E^1_\e(u_n^y, J)\,dy, \qquad E_\e(u_n, \Omega)\geq \int_JE^1_\e(u_n^z, I)\,dz.
\end{equation}
Arguing as in the proof of Corollary \ref{BDS1dimcorollary}, we find that that $\{u_n\}_n$ and $\{|u_n|^2\}_n$ are equi-intergable sequences; therefore, we fix $\delta>0$ and we find $\delta'\in(0,\delta)$ such that
\begin{equation}\label{abs2}
\mathcal{L}^2(A)\leq \delta' \mathcal{L}^1(J)
 \implies \sup_{n\in\N} \int_A (|u_n|+1)\,dx\leq \delta.
\end{equation}
Letting $S$ denote $\sup_n E_{\e_n}(u_n, \Omega)$, for every $n\in \N$ we set
\[
v_n^y(z):=\begin{cases}
    u_n^y(z) & \text{ if } E^1_\e(u_n^y, J) \leq S/\delta', \\
1 & \text{ otherwise}.
\end{cases}
\]
We also put $Y_n:=\{y\in I :u^y_n- v_n^y \text{ is not identically }0\}$ and we observe that 
\begin{equation*}
    \mathcal{L}^1(Y_n) \leq \mathcal{L}^1(\{ y\in I:E^1_\e(u_n^y, J) > S/\delta'\}) \leq \frac{\delta'}{S}\int_I E^1_\e(u_n^y, J)\,dy \leq \delta',
\end{equation*}
where the last inequality follows by \eqref{abs1}. This inequality, combined with \eqref{abs2}, implies that
\[
\|u_n-v_n\|_{L^1(\Omega)}\leq \int_{Y_n\times J} (|u_n|+1)\,dx\leq \delta;
\]
that is, $\{u_n\}_n$ and $\{v_n\}_n$ are $\delta$-close. Moreover, by construction and by (H$1$), we have that for every $y\in I$ the sequence $\{v_n^y\}_n$ is such that $\sup_n E^1_{\e_n}(v_n^y, J)\leq S/\delta'$. Hence, by Corollary \ref{BDS1dimcorollary}, this sequence is precompact in $L^2(J)$, and thus in $L^1(J)$.

A similar argument allows to find a sequence $\{w_n \}_n$ that is $\delta$-close to $\{u_n\}_n$ and such that, for a.e.\ $z\in J$, the sequence $\{w_n^z\}_n$ is precompact in $L^1(I)$. Therefore, by applying Proposition \ref{slicingcompactness} we deduce that $\{u_n\}_n$ is precompact in $L^1(\Omega)$. Finally, since $\{|u_n|^2\}_n$ is equi-integrable, we deduce the precompactness in $L^2(\Omega)$.

To conclude, we prove, without further assumptions on $d$ and $\Omega$, that if the subsequence $\{u_{n_j}\}_j$ converges to $u$ in $L^2(\Omega)$ as $j\to+\infty$, then $u\in BV(\Omega;\{-1,1\})$. To see this, using the notation introduced in Section $2$, we prove that for every $m\in\{1,...,d\}$, the one-dimensional slice $u^{e_m,y}$ has bounded variation on the open set $\Omega^{e_m,y}$ for $\Hd$-a.e.\ $y\in \Omega^{e_m}$, and 
\begin{equation*}
    \int_{\Omega^{e_m}} \# (S(u^{e_m,y})\cap\Omega^{e_m,y})\,d\Hd(y)<+\infty.
\end{equation*}
Arguing as previously done for \eqref{abs1}, we have that
\begin{equation*}
    E_\e(u_{n_j}, \Omega)\geq \int_{\Omega^{e_m}} E^1_\e(u_{n_j}^{e_m,y}, \Omega^{e_m,y})\,d\Hd(y).
\end{equation*}
Moreover, upon extracting a further subsequence that we do not relabel, it holds that $u^{e_m,y}_{n_j}\to u^{e_m,y}$ in $L^2(\Omega^{e_m,y})$ for $\Hd$-a.e.\ $y\in \Omega^{e_m}$. Hence, by Fatou's lemma and Corollary \ref{BDS1dimcorollary} applied with $I=\Omega^{e_m,y}$ and $u=u^{e_m,y}$, we get
\begin{align*}
 +\infty > S \geq \liminf_{j\to+\infty}  E_\e(u_{n_j}, \Omega) & \geq \int_{\Omega^{e_m}} \liminf_{j\to+\infty} E^1_\e(u_{n_j}^{e_m,y}, \Omega^{e_m,y})\,d\Hd(y) \\
  & \geq \int_{\Omega^{e_m}} E_0^1(u^{e_m,y},\Omega^{e_m,y})\,d\Hd(y).
\end{align*}
As a consequence, $E_0^1(u^{e_m,y},\Omega^{e_m,y})$ is finite for $\Hd$-a.e.\ $y\in \Omega^{e_m}$, which implies that
\begin{align*}
 +\infty > S & \geq \int_{\Omega^{e_m}} E_0^1(u^{e_m,y},\Omega^{e_m,y})\,d\Hd(y) \\
  & \geq m_k\int_{\Omega^{e_m}} \#(S(u^{e_m,y})\cap\Omega^{e_m,y})\,d\Hd(y).
\end{align*} 
Then, since $m_k>0$, for such $y$ it holds that $u^{e_m,y}\in BV(\Omega^{e_m,y};\{-1,1\})$, hence, the conclusion follows applying Theorem \ref{bvslicing}.
\end{proof}
Finally, the proof our main result is an immediate consequence.

\begin{proof}[Proof of Theorem \ref{compactness}]
    Let $S$ denote $\sup_n F_{\e_n}(u_n, \Omega)$.
    If $\mathcal{N}=\emptyset$, by the equivalence of the norms there exists $c>0$ such that $F_{\e_n}(u_n,\Omega) \geq  cE_{\e_n}(u_n,\Omega)$ for every $n$ so that
    \begin{equation*}
        S/c \geq \sup_n E_{\e_n}(u_n, \Omega)
    \end{equation*}
    and the conclusion follows by Proposition \ref{compactnessBDSanydim}.

    \noindent If $\mathcal{N}\neq\emptyset$, again by the equivalence of the norms and by Corollary \ref{interpolationwithnorms}, upon restricting to large $n$, it holds
    \begin{equation*}
        S/\delta' \geq \sup_n E_{\e_n}(u_n, \Omega)
    \end{equation*}
    for some $\delta'>0$; therefore, the conclusion follows by Proposition \ref{compactnessBDSanydim}.
\end{proof}

\section{Lower bound}\label{section lower bound}

In this section we prove the liminf inequality of Theorem \ref{main theorem}. In particular, we prove the following result.
\begin{proposition}\label{prop:liminf}Let $\Omega\subset \Rd$ 
be a bounded, open set with $C^1$ boundary and let $\mathcal{N}:=\{\ell\in\{1,...,k-1\}: q_\ell \leq 0\}$. Assume that {\rm (H1)-(H3)} are satisfied and that one of the following  conditions holds: 
\begin{itemize}
    \item[\emph{(i)}] $\mathcal{N}=\emptyset$;
    \item [\emph{(ii)}] $\mathcal{N}\neq \emptyset$ and $q_\ell>-\overline{q}_\ell$ for every $\ell\in\mathcal{N}$ with the same positive constants $\{\overline{q}_\ell:\ell\in\mathcal{N}\}$ independent of $\Omega$ appearing in Theorem \ref{compactness}.
\end{itemize}
    Then for every $\{\e_n\}_n$  positive sequence converging to $0$ and $\{u_n\}_n\subset H^k(\Omega)$ such that $u_n\to u$ in $L^2(\Omega)$ it holds
    \begin{equation*}
F_0(u)\leq\liminf_{n\to+\infty}F_{\e_n}(u_n).
    \end{equation*}
    \end{proposition}

In the following, $C$ denotes a generic positive constant that may change from line to line, but which is independent of the relevant parameters.
We recall that $\overline{u}$ is a function in $H^k_{\rm loc}(\R)$ such that $\overline{u}(t)=-1$ if $t\leq -1/8$  and $\overline{u}(t)=1$ if $t\geq 1/8$, $|\overline{u}|\leq 1$,  and 
\begin{equation}\label{boundedness of ubar}
    0<\int_{-\infty}^{+\infty} \Bigl[W(\overline{u})+\sum_{\ell=1}^k|q_\ell||\overline{u}^{(\ell)}|^2\Bigr]\,dt<+\infty.
\end{equation}
Moreover, for every $\e>0$ we have set $\overline{u}^\nu_\e(x):=\overline{u}(\tfrac{x\cdot\nu}{\e})$ for all $x\in \Rd$. We observe that, thanks to \eqref{boundedness of ubar}, the definition of $\overline{u}^\nu_{\e}$, and Fubini's Theorem, for every open set $V\subseteq \Omega$ there exists a constant $C>0$ such that for every $\nu\in\Sd$ and $\e>0$ we have that
\begin{align}\notag
 F_{\e}(\overline{u}_{\e},V)&\leq C\int_{\pi^\nu(V)}\int_{-\infty}^{+\infty}\big[W(\overline{u})+\sum_{\ell=1}^k|q_\ell||\overline{u}^{(\ell)}|^2\big]\,dt\,d\Hd(y)\\ \label{bounds on uepsbar}
 &\leq C\int_{-\infty}^{+\infty}\big[W(\overline{u})+\sum_{\ell=1}^k|q_\ell||\overline{u}^{(\ell)}|^2\big]\,dt \leq C,
\end{align}
where $\pi^\nu$ denotes the orthogonal projection onto $\nu^\perp$ and where we have also used that $(\overline{u}_\e^\nu)^{\nu,y}(t)=\overline{u}(\tfrac{t}{\e})$ for every $\e>0$, $y\in \nu^\perp$, and $t\in\R$.

We further set \begin{equation}\label{def uzeronu}
    u^\nu_0(x):=\text{sgn}(x\cdot\nu),
\end{equation} where $\text{sgn}$ denotes the sign function and remark that, for every $K\subset \Rd$ compact, $\{\overline{u}_{\e}^\nu\}_\e$ converges to $u^\nu_0$ in $L^2(K)$ as $\e\to0$.

We start by proving a lemma which allows to modify boundary conditions on a converging sequence at the cost of an asymptotically negligible amount of energy.

For every $x\in\R^d,\nu\in\Sd$, we introduce $\widetilde{Q}^{\nu}(x)$ an open subset of $Q^\nu(x)$ having $C^2$ boundary and such that
    \begin{equation*}
             \widetilde{Q}^{\nu}(x)\cap \{y\in\Rd:|y\cdot\nu|\leq \tfrac{1}{4}\}=Q^\nu(x)\cap \{y\in\Rd:|y\cdot\nu|\leq \tfrac{1}{4}\}.
         \end{equation*}
We suppose that all the above sets are congruent, namely, that they are obtained starting by $\widetilde{Q}^{e_d}(0)$ by a translation and a rotation.  Also for the {\em rounded} cube $\widetilde{Q}^\nu(x)$ we use the convention that when $x=0$ we simply write $\widetilde{Q}^\nu$.

\begin{lemma}\label{lemma:modifying sequences}
Let $\{\e_n\}_n$ be a positive sequence converging to $0$, $\nu\in\Sd$, and let $\{v_n\}_n\subset H^k(\widetilde{Q}^\nu)$ be a sequence converging to $u^\nu_0$ in $L^2(\widetilde{Q}^\nu)$. Assume that {\rm (H1)-(H3)} are satisfied and that one of the two conditions (i) and (ii) of  Proposition \ref{prop:liminf} holds. Then there exists a sequence $\{w_n\}_n\subset H^k(\widetilde{Q}^\nu)$ such that 
\begin{enumerate}
    \item[(i)] $w_n\to u^\nu_0$ as $n\to +\infty$  $L^2(\widetilde{Q}^\nu)$,
    \item[(ii)] the support of $w_n-\overline{u}_n$ is compactly contained in $\widetilde{Q}^\nu$ for every $n\in\N$,
    \item[(iii)] it holds
    \begin{equation}\notag
        \liminf_{n\to+\infty}F_{\e_n}(w_n,\widetilde{Q}^\nu)\leq\liminf_{n\to+\infty}F_{\e_n}(v_n,\widetilde{Q}^\nu).
    \end{equation}
\end{enumerate}
\end{lemma}
\begin{proof}
For the sake of exposition, we only consider the case $\nu=e_d$, the others being analogous. For simplicity of notation, throughout the rest of the proof we set $\overline{u}_{n}:=\overline{u}^{e_d}_{\e_n}$, and we set $\widetilde{Q}:=\widetilde{Q}^{e_d}(0)$.
    It is not restrictive to assume that
    \begin{equation}\notag
    \|v_n-\overline{u}_n\|_{L^2(\widetilde{Q})}>0\text{ for every $n\in\N$}
    \end{equation}
as, if $v_n=\overline{u}_n $ for infinitely many $n\in\N$, it enough to set $w_n:=\overline{u}_n$ to conclude. An analogous argument shows that we may also assume that
\begin{eqnarray}
    &\displaystyle\label{eq:Lemma-bounded energy}\liminf_{n\to+\infty} F_{\e_n}(v_n,\widetilde{Q})=\lim_{n\to+\infty}F_{\e_n}(v_n,\widetilde{Q}) \leq C<+\infty.
\end{eqnarray}

\noindent By Corollary \ref{interpolationwithnorms}, \eqref{bounds on uepsbar}, and \eqref{eq:Lemma-bounded energy}   there exists a  constant $\delta>0,$ independent of $n$ and $\widetilde{Q}$, such that 
\begin{gather}\label{new boundedness lemma 1}
    \int_{\widetilde{Q}}\Bigl[\frac{W(v_n)}{\e_n}+\sum_{\ell=1}^k\e_n^{2\ell-1}|\nabla^{(\ell)}v_n|_\ell^2\Bigr]\,dx\leq   
    \frac1\delta F_{\e_n}(v_n,\widetilde{Q})\leq C,\\
    \label{new boundedness lemma 2}
         \int_{\widetilde{Q}}\Bigl[\frac{W(\overline{u}_n)}{\e_n}+\sum_{\ell=1}^k\e_n^{2\ell-1}|\nabla^{(\ell)}\overline{u}_{n}|_\ell^2\Bigr]\,dx\leq  \frac1\delta F_{\e_n}(\overline{u}_n,\widetilde{Q})\leq C,
\end{gather}
for $n\in\N$ large enough depending on $\{q_\ell : q_\ell\leq 0\}$ and $\widetilde{Q}$. Additionally, as both $v_n$ and $\overline{u}_n$ converge to $u^{e_d}_0$ in $L^2(\widetilde{Q})$ as $n\to+\infty$, we have that
\begin{equation}\label{L2 bound liminf}
    \int_{\widetilde{Q}}|v_n-\overline{u}_n|^2\,dx\leq C
\end{equation}
for every $n\in\N$.

Since the set $\widetilde{Q}$ is $C^2$, there exists $\rho_0>0$ such that for every $0<\rho<\rho_0$ there is a unique $C^1$ projection from the neighbourhood $\{x\in \widetilde{Q}: \text{dist}(x,\partial \widetilde{Q})<\rho \}$  onto $\partial \widetilde{Q}$. In particular, for every $0<\rho<\rho_0$ the sets $\{x\in \widetilde{Q}:\text{dist}(x,\partial \widetilde{Q})=\rho\}$ are $(d-1)$-dimensional manifolds in $\Rd$ of class $C^1$. 

Let $0<\rho<\rho_0/2$ and $0<\sigma<\rho_0/2$, and let $\lceil\e_n^{-1} \rceil$ be the smallest integer larger than or equal to $\e_n^{-1}$, so that 
\begin{equation}\label{epsilon contained}
    \tfrac{\e_n}{2}\leq\lceil\e_n^{-1} \rceil^{-1}\leq \e_n
\end{equation}
for $\e_n<1$. In light of the observations above, the set of class $C^1$ given by 
\begin{equation*}
    \{x\in \widetilde{Q} : \text{dist}(x,\partial \widetilde{Q})<\rho+\sigma\}
\end{equation*}
can be subdivided in $\lceil\e_n^{-1}\rceil$ open subsets with $C^1$ boundary defined as
\begin{equation} 
    L^i_{n,\sigma,\rho}:=\Bigl\{x\in\widetilde{Q}: (i-1)\frac{\sigma}{\lceil\e_n^{-1}\rceil}+\rho  <\text{dist} (x,\partial \widetilde{Q})< i\frac{\sigma}{\lceil\e_n^{-1}\rceil}+\rho\Bigr\} \quad \text{ for }i\in\{1,...,\lceil\e_n^{-1}\rceil\}.
\end{equation}
We observe that  
\begin{equation}\label{volume layers}
\Ld(L^{i}_{n,\sigma,\rho})\leq C\frac{\sigma}{\lceil\e_{n}^{-1}\rceil}.    
\end{equation}
From \eqref{new boundedness lemma 1}-\eqref{L2 bound liminf}  we deduce that 
\begin{align*}\sum_{i=1}^{\lceil\e_n^{-1
}\rceil}\Big(\int_{L^{i}_{n,\sigma,\rho}}&\Bigl[\frac{W(v_n)}{\e_n}
+\sum_{\ell=1}^k\e_n^{2\ell-1}|\nabla^{(\ell)}v_n|_\ell^2\Bigr]\,dx \\&+\int_{L^{i}_{n,\sigma,\rho}}\Bigl[\frac{W(\overline{u}_n)}{\e_n}+\sum_{\ell=1}^k\e_n^{2\ell-1}|\nabla^{(\ell)}\overline{u}_{n}|_\ell^2\Bigr]\,dx+\int_{L^{i}_{n,\sigma,\rho}}\frac{|v_n-\overline{u}_{n}|^2}{\|v_n-\overline{u}_n\|_{L^2(\widetilde{Q})}^2}\,dx\Big)
\leq C,\end{align*}
for every $n$ large enough. Hence, by \eqref{epsilon contained} we find $i^*=i^*(n,\sigma,\rho)\in \{1,...,\lceil\e_n^{-1}\rceil\}$ such that
\begin{align}\nonumber 
 \Big(\int_{L^{i^*}_{n,\sigma,\rho}}&\Bigl[\frac{W(v_n)}{\e_n}
+\sum_{\ell=1}^k\e_n^{2\ell-1}|\nabla^{(\ell)}v_n|_\ell^2\Bigr]\,dx \\&+\int_{L^{i^*}_{n,\sigma,\rho}}\Bigl[\frac{W(\overline{u}_n)}{\e_n}+\sum_{\ell=1}^k\e_n^{2\ell-1}|\nabla^{(\ell)}\overline{u}_{n}|_\ell^2\Bigr]\,dx+\int_{L^{i^*}_{n,\sigma,\rho}}\frac{|v_n-\overline{u}_{n}|^2}{\|v_n-\overline{u}_n\|_{L^2(\widetilde{Q})}^2}\,dx\Big) \leq C\e_n.\nonumber
\end{align}
In particular, for every $\ell\in\{1,...,k\}$ we have 
\begin{gather}\label{bounded derivatives layer}
\e_n^{2\ell-1}\||\nabla^{(\ell)}v_n|_\ell\|^2_{L^2(L^{i^*}_{n,\sigma,\rho})}\leq C\e_n,\\ \label{bounded derivatives layers 2}
\e_n^{2\ell-1}\||\nabla^{(\ell)}\overline{u}_n|_\ell\|^2_{L^2(L^{i^*}_{n,\sigma,\rho})}\leq C\e_n,\\
\label{bound double well witH N}
    \frac{1}{\e_n}\int_{L^{i^*}_{n,\sigma,\rho}} W(v_n)\,dx+\frac{1}{\e_n}\int_{L^{i^*}_{n,\sigma,\rho}} W(\overline{u}_n)\,dx\leq C\e_n,\\ \label{bound on L2strip}
    \int_{L^{i^*}_{n,\sigma,\rho}}|v_n-\overline{u}_{n}|^2\,dx\leq C\e_n
    \|v_n-\overline{u}_{n}\|^2_{L^2(\widetilde{Q})}.
\end{gather} 

We now introduce the sets 
\begin{gather*}
   A_{n,\sigma,\rho}:=\big\{x\in \widetilde{Q}: \text{dist}(x,\partial \widetilde{Q})\leq \rho +(i^*-1)\tfrac{\sigma}{\lceil\e_n^{-1}\rceil} \big\},\\
   B_{n,\sigma,\rho}:=\big\{x\in \widetilde{Q} :\text{dist}\big(x,\partial \widetilde{Q})>\rho +i^*\tfrac{\sigma}{\lceil\e_n^{-1}\rceil}\},   
\end{gather*}
and cut-off functions $\phi_{n,\sigma,\rho}\in C^\infty_c(\widetilde{Q}; [0,1])$ such that 
\begin{equation}\label{eq:def cut offs lemma}
    \hspace{2.5 cm}\begin{cases}
    \phi_{n,\sigma,\rho}(x)=0\quad \text{if }x\in A_{n,\sigma,\rho},\\    
    \phi_{n,\sigma,\rho}(x)=1\quad \text{if }x\in B_{n,\sigma,\rho},\\
\||\nabla^{(\ell)}\phi_{n,\sigma,\rho}|_\ell\|_{L^\infty(\widetilde{Q})}\leq C(\sigma\e_n)^{-\ell}\quad &\text{ for every }\ell\in\{1,...,k\}.
    \end{cases}
\end{equation}
For $x\in \widetilde{Q}$ we define
\begin{equation}\label{def wn}
w_n(x)=w^{\sigma,\rho}_n(x):=\phi_{n,\sigma,\rho}(x)v_n(x)+(1-\phi_{n,\sigma,\rho}(x))\overline{u}_{n}(x),
\end{equation}
and we observe that $w_n\in H^k(\widetilde{Q})$ and that $w_n-\overline{u}_n$ is compactly supported in $\widetilde{Q}$. Also note that, since both $\{v_n\}_n$ and $\{\overline{u}_n\}_n$ converge to $u^{e_d}_0$ in $L^2(\widetilde{Q})$, for every choice of $\sigma$ and $\rho$ the sequence $\{w_n\}_n$ converges in $L^2(\widetilde{Q})$ to $u^{e_d}_0$ as well.  

By \eqref{eq:def cut offs lemma} and \eqref{def wn}, we have 
\begin{align}
    \nonumber
    F_{\e_n}(w_n,\widetilde{Q})&=F_{\e_n}(v_n,B_{n,\sigma,\rho})+F_{\e_n}(w_n,L^{i^*}_{n,\sigma,\rho})+F_{\e_n}(\overline{u}_n,A_{n,\sigma,\rho})\\\nonumber 
  &\leq F_{\e_n}(v_n,B_{n,\sigma,\rho})+\int_{L^{i^*}_{n,\sigma,\rho}}\Bigl[\frac{W(w_n)}{\e_n}+\sum_{\ell=1}^k|q_\ell|\e_n^{2\ell-1}|\nabla^{(\ell)}w_n|_\ell^2\Bigr]\, dx
  \\\label{eq:splitting functional}
  & \,\,\quad +\int_{A_{n,\sigma,\rho}}\Bigl[\frac{W(\overline{u}_n)}{\e_n}+\sum_{\ell=1}^k|q_\ell|\e_n^{2\ell-1}|\nabla^{(\ell)}\overline{u}_n|_\ell^2\Bigr]\, dx=:I^1_n+I^2_n+I^3_n.
\end{align}

We now study the asymptotic behaviour of these three quantities. 

\noindent We claim that
\begin{equation}\label{I1}
    \limsup_{n\to+\infty} I^1_n = \limsup_{n\to+\infty} F_{\e_n}(v_n,B_{n,\sigma,\rho}) \leq \lim_{n\to+\infty} F_{\e_n}(v_n,\widetilde{Q})  
\end{equation}
for every fixed $\sigma, \rho$ small enough. This  inequality follows from an application of Remark \ref{re:uniform interpolation} with $\Omega=\widetilde{Q}$, $\rho^*=\rho$, and $\sigma<R^*(\rho)$. The choice of these parameters is allowed by our construction since  we will conclude our proof by letting $n\to+\infty,\, \sigma\to 0,$ and $ \rho\to0$, in this order.  

\noindent To prove \eqref{I1}, we set \begin{equation*}
    \widetilde{Q}_\tau:=\{x\in\widetilde{Q}:\text{dist}(x,\partial\widetilde{Q})<\tau\},
\end{equation*} 
for $\tau\in(\rho,\rho+\sigma)$. Then, by Remark \ref{re:uniform interpolation}, for every $\ell\in\{1,...,k-1\}$  we find $\e_0=\e_0(q_\ell,\widetilde{Q},\rho,\sigma)>0$ such that
\begin{equation*}
\int_{ \widetilde{Q}_{\tau} }\Bigl[\frac{1}{\e_n}W(v_n)+q_{\ell}\e_n^{2\ell-1}|\nabla^{(\ell)}v_n|_\ell^2\Bigr]\, dx\geq 0
\end{equation*}
holds for every  $n$ sufficiently large so that $\e_n\leq \e_0$ and for every $\tau\in(\rho,\rho+\sigma)$. Arguing as in the proof of Corollary \ref{interpolationwithnorms}, we deduce that there exists a possibly smaller $\e_0=\e_0(\{q_\ell:q_\ell\leq 0\},\widetilde{Q},\rho,\sigma)$ such that
\begin{equation}\label{positivity of energy}
    \int_{ \widetilde{Q}_{\tau} }\Bigl[\frac{1}{\e_n}W(v_n)+q_{\ell}\sum_{\ell=1}^k\e_n^{2\ell-1}|\nabla^{(\ell)}v_n|_\ell^2\Bigr]\, dx\geq 0
\end{equation}
for every $n$ sufficiently large so that $\e_n\leq \e_0$ and for every $\tau\in (\rho,\rho+\sigma)$. We now apply \eqref{positivity of energy} with $\tau=\rho+i^{*}\tfrac{\sigma}{\lceil\e_n^{-1}\rceil}$, so that $\widetilde{Q}_\tau=\widetilde{Q}\setminus B_{n,\sigma,\rho}$, and obtain that 
\begin{equation*}
    F_{\e_n}(v_n,\widetilde{Q})= F_{\e_n}(v_n,\widetilde{Q}\setminus B_{n,\sigma,\rho})+F_{\e_n}(v_n, B_{n,\sigma,\rho})\geq F_{\e_n}(v_n, B_{n,\sigma,\rho}),
\end{equation*}
which implies \eqref{I1}. 

 We now study the term $I^3_n$. From Fubini's Theorem, the definition of $\overline{u}_n$, a change of variable, and \eqref{boundedness of ubar} it follows that 
\begin{equation*}
    I^3_n\leq C \mathcal{H}^{d-1}(A_{n,\sigma,\rho}
    \cap\nu^\perp)\Big(\int_{-\infty}^{+\infty}[W(\overline{u}(t))+\sum_{\ell=1}^k|q_\ell||\overline{u}^{(\ell)}|^2]\, dt\Big)\leq C(\rho+\sigma),
\end{equation*}
so that 
\begin{equation}\label{I3vanishes}
    \limsup_{n\to+\infty}I^3_n\leq C(\rho+\sigma).
\end{equation}

We now show that $I^2_n$ is asymptotically negligible. We begin by studying the term involving the double well potential. We remark that the continuity of $W$ and (H3) imply there exists $C>0$ such that 
\begin{equation}\label{H3+ continuity}
    W(s)\leq W(t)+C\end{equation}
for every $s=(1-\theta)t+\theta t_0$ with $\theta\in [0,1)$ and $|t_0|\leq 1$. Indeed, if $|t|\leq 1$, then $|s|\leq 1$, so that $W(s)\leq \sup_{\tau\in[-1,1]}W(\tau)$ , while if instead $|t|>1$ then $s\leq (1-\theta)|t|+\theta |t_0|\leq |t|$, so that \eqref{H3+ continuity} follows from (H3).
Thus, recalling that $|\overline{u}_n|\leq 1$, from  \eqref{H3+ continuity}, \eqref{bound double well witH N}, \eqref{volume layers}, and \eqref{epsilon contained}  we deduce that 
\begin{align}\nonumber 
\frac{1}{\e_n}\int_{L^{i^*}_{n,\sigma,\rho}}W(w_n)\,dx&\leq  \frac{1}{\e_n}\Big(\int_{L^{i^*}_{n,\sigma,\rho}}W(v_n)\, dx+C\Ld( L^{i^*}_{n,\sigma,\rho})\Big)\\\label{bound double well}
 &\leq C\e_n+\frac{C\sigma}{\e_n\lceil\e_n^{-1}\rceil}\leq C(\e_n+\sigma).
 \end{align}

We now investigate the remaining terms in $I^2_n$. We note that for every  $\ell\in\{1,...,k\}$ and multi-index $\alpha$ of order $|\alpha|=\ell$ we have 
\begin{align}\nonumber
\partial^\alpha w_n&=\sum_{\beta\leq \alpha}\binom{\alpha}{\beta}\Big(\partial^\beta\phi_{n,\sigma,\rho}\partial^{\alpha-\beta}v_n+\partial^\beta(1-\phi_{n,\sigma,\rho})\partial^{\alpha-\beta}\overline{u}_n\Big)\\\nonumber
&=\phi_{n,\sigma,\rho}\partial^\alpha v_n + (1-\phi_{n,\sigma,\rho}) \partial^\alpha\overline{u}_n+ \sum_{0<\beta\leq\alpha}\binom{\alpha}{\beta}\big(\partial^\beta\phi_{n,\sigma,\rho}\,\partial^{\alpha-\beta}(v_n-\overline{u}_n)\big)
\end{align}
so that, by \eqref{eq:def cut offs lemma}, we deduce that for every $\ell\in \{1,...,k\}$ the following inequality holds
\begin{align*}
    |\nabla^{(\ell)}w_n|^2_\ell\leq C&\Bigl(|\nabla^{(\ell)}v_n|_\ell^2
    + |\nabla^{(\ell)}\overline{u}_n|_\ell^2 \\+&(\sigma\e_n)^{-2\ell}|v_n-\overline{u}_n|^2+\sum_{j=1}^{\ell-1}(\sigma\e_n)^{-2(\ell-j)}|\nabla^{(j)}(v_n-\overline{u}_n)|^2_j\Bigr).
\end{align*}
 Combining this inequality with \eqref{bounded derivatives layer}, \eqref{bounded derivatives layers 2}, \eqref{bound on L2strip}, and \eqref{epsilon contained} we obtain
\begin{align*}
\e_n^{2\ell-1}\| |\nabla^{(\ell)}w_n|_\ell\|_{L^2(L^{i^*}_{n,\sigma,\rho})}^2&\leq C\e_n+C\sigma^{-2\ell}\e_n^{-1}\int_{L^{i^*}_{n,\sigma,\rho}}|v_n-\overline{u}_n|^2\,dx\\
  &\,\,\,\,\,\,\,\,+C\sum_{j=1}^{\ell-1}\frac{\e_n^{2j-1}}{\sigma^{2(\ell-j)}}\||\nabla^{(j)}(v_n-\overline{u}_n)|_j\|^2_{L^{2}(L^{i^*}_{n,\sigma,\rho})}\\
  &\leq C\e_n+C{\sigma^{-2\ell}}\|v_n-\overline{u}_n\|^2_{L^2(\widetilde{Q})}+C\big(\sum_{j=1}^{\ell-1}\frac{1}{\sigma^{2(\ell-j)}}\big)\e_n.
\end{align*}
Finally, this last inequality, together with \eqref{bound double well}, yields
 \begin{equation}\label{estimate on I2}
     I^2_n\leq C\e_n+C\sigma+C\big(\sum_{\ell=1}^k{\sigma^{-2\ell}}\big)\|v_n-\overline{u}_n\|^2_{L^2(\widetilde{Q})}+C\big(\sum_{\ell=1}^{k}\sum_{j=1}^{\ell-1}\frac{1}{\sigma^{2(\ell-j)}}\big)\e_n.
 \end{equation}
Since $\{v_n-\overline{u}_n\}_n$ converges to $0$ in $L^2(\widetilde{Q})$ as $n\to+\infty$, by \eqref{eq:splitting functional}, \eqref{I1}, \eqref{I3vanishes}, and \eqref{estimate on I2}  we infer that
$$\limsup_{\rho\to 0}\limsup_{\sigma\to 0}\limsup_{n\to+\infty }F_{\e_n}(w_n,\widetilde{Q})\leq \lim_{n\to+\infty}F_{\e_n}(v_n,\widetilde{Q}).$$
Recalling that for every fixed $\sigma$ and $\rho$ the sequence $\{w_n\}_n=\{w_n^{n,\sigma,\rho}\}_n$ converges to $u_0^\nu$ in $L^2(\widetilde{Q})$ and is such that the support of $w_n-\overline{u}_n$ is compactly contained in $\widetilde{Q}$ for all $n\in\N$, we let $\sigma$ and $\rho$ vary in two positive vanishing sequences and use a diagonalization argument to conclude the proof.
\end{proof}

We are now ready to prove the liminf inequality. The proof is based on the blow-up argument of Fonseca and M\"uller \cite{FonsecaMuller}.
    \begin{proof} [Proof of Proposition \ref{prop:liminf}]
      Let us fix a positive sequence $\{\e_n\}_n$ converging to $0$ as $n\to+\infty$.  Without loss of generality, we may assume that 
        \begin{equation}\label{eq:energy bounded liminf}
            \liminf_{n\to+\infty}F_{\e_n}(u_n,\Omega)=\lim_{n\to+\infty} F_{\e_n}(u_n,\Omega)\leq C<+\infty.
        \end{equation}
       In light of Theorem \ref{compactness}, this implies that $u\in BV(\Omega;\{-1,1\})$ so that  $u=2\chi_E-1$ for $E$ a set of finite perimeter in $\Omega$ and $S(u)\cap \Omega=\partial^*E\cap\Omega $, where we recall that $\partial^*E$ denotes the reduced boundary of $E$.

For every $n\in \N$, we consider the measure $\mu_n$ defined by
    $$\mu_n(B):=F_{\e_n}(u_n,B)\quad \text{ for every }B \text{ Borel subset of } \Omega.$$
    Note that by Corollary \ref{interpolationwithnorms} and \eqref{eq:energy bounded liminf}, upon assuming $n$ large enough, the sequence $\{f_n:=\frac{W(u_n)}{\e_n}+\sum_{\ell=1}^kq_{\ell}\e^{2\ell-1}|\nabla^{(\ell)}u_n|^2\}_n$ is equibounded in $L^1(\Omega)$.
    Therefore, letting $|M|$ denote the total variation of a measure $M$, we have $\sup_n|\mu_n|(\Omega)\leq C$, so that, up to passing to a not relabelled subsequence, we may assume that $\{\mu_n\}_n$ and $\{|\mu_n|\}_n$ converge weakly* to two finite Radon measures $\mu$ and $\lambda$, respectively. We now show that $\mu\geq 0$.
    To this aim, we use the Besicovitch Derivation Theorem (see \cite[Theorem 1.153]{FonsecaLeoniBook}) to write for $|\mu|$-a.e.\ $x\in \Omega$ 
    \begin{equation*}
        \frac{d\mu}{d|\mu|}(x)=\lim_{r\to 0^+}\frac{\mu(B_r(x))}{|\mu|(B_r(x))},
    \end{equation*}
    where $d\mu/d|\mu|$ is the Radon-Nikod\'ym derivative of $\mu$ with respect to its total variation $|\mu|$. Let us fix a point $x$ such that the previous equality holds. Since the set $\{r\in (0,+\infty): \lambda(\partial B_r(x))>0\}$ is at most countable, we may choose a sequence of positive real numbers $\{r_m\}_m$, vanishing as $m\to +\infty$, such that 
    \begin{equation*}
         \frac{d\mu}{d|\mu|}(x)=\lim_{m\to +\infty}\lim_{n\to +\infty}\frac{1}{|\mu|(B_{r_m}(x))}  \int_{B_{r_m}(x)}f_n(y)\,dy.
    \end{equation*}
    For every $m\in\N$ fixed, by Corollary \ref{interpolationwithnorms} the integral in the right-hand side of the previous equality is non-negative for $n\in\N$ large enough, hence, we conclude that $d\mu/d|\mu|$ is non-negative $|\mu|$-a.e.\, and then $\mu\geq 0$.

    By the Radon-Nikod\'ym Theorem we write $\mu=f\Hd\mres(\partial^* E \cap \Omega)+\sigma$, for $f$ a non-negative function and for $\sigma$ a non-negative measure which is singular with respect to $\Hd\mres(\partial^*E\cap \Omega)$. 
    
    We now show that the function $f$ can be estimated from below in terms of the function $g$ defined by \eqref{density}.  To see this, for every $x\in\R^n$, $\nu\in\Sd$ we  consider $\widetilde{Q}^{\nu}(x)$ an open subset of $Q^\nu(x)$ having $C^2$ boundary and such that 
    \begin{equation}
    \label{Q rectangle}
             \widetilde{Q}^{\nu}(x)\cap \{y\in\Rd:|y\cdot\nu|\leq \tfrac{1}{4}\}=Q^\nu(x)\cap \{y\in\Rd:|y\cdot\nu|\leq \tfrac{1}{4}\}.
         \end{equation}
    For every $x\in\partial^*E\cap \Omega$ we consider $\nu=\nu_E(x)$ the inner unit normal to $\partial^*E\cap \Omega$ at $x$. By \eqref{Q rectangle}, using Theorem \ref{thm:structurethm} one can show that for $\Hd$-a.e.\ $x\in\partial^*E\cap \Omega$ we have 
         \begin{equation}\label{new ratio liminf}
             \lim_{r\to 0}\frac{\Hd(\widetilde{Q}_r^\nu(x)\cap \partial^*E) }{r^{d-1}}=1,
         \end{equation}
         where we used the notation $\widetilde{Q}_r^\nu(x)=\{y\in \R^n: y/r\in\widetilde{Q}^\nu(x)\}$.
    By the Besicovitch Derivation Theorem, for $\Hd$-a.e.\ $x\in \partial^*E\cap \Omega$ we have 
        \begin{equation}\label{eq:besicovitch 1}
            f(x)=\lim_{r\to 0^+}\frac{\mu(\widetilde{Q}_r^\nu(x))}{{\mathcal{H}^{d-1}(\widetilde{Q}_r^\nu(x)\cap  \partial^*E})},           
        \end{equation}
    Additionally, given $x\in \partial^*E\cap\Omega$ such that \eqref{eq:besicovitch 1} holds, since $\{\mu_n\}_n$ weakly*-converges to $\mu$, it follows from \eqref{new ratio liminf} and a diagonal argument that there exists a vanishing sequence $\{r_n\}_n$, with $\e_n\ll r_n$, such that   
        \begin{equation}\label{eq:besicovitch 2} f(x) = \lim_{n\to+\infty}\frac{\mu_n(\widetilde{Q}_{r_n}^\nu(x))}{r_n^{d-1}}.
        \end{equation}
        
\noindent Upon translating the point $x$, it is not restrictive to assume that $x=0\in\partial^*E\cap \Omega$; moreover, for the sake of notation and simplicity of exposition, we only discuss the case and $\nu_E(x)=e_d$, the other cases being analogous. 

We write $\widetilde{Q}$ in place of $\widetilde{Q}^{e_d}(0)$ and $\widetilde{Q}_n$ in place of $\widetilde{Q}_{r_n}^{e_d}(0)$. Then, we set $v_n(y):=u_n(r_ny)$ for every $y\in\widetilde{Q}$ and, by  Theorem \ref{thm:structurethm},  we have that $\{v_n\}_n$ converges in $L^2(\widetilde{Q})$ to the function $u^{e_d}_0$ defined by \eqref{def uzeronu}. We further set $\tau_n:=\e_n/r_n$ and use the change of variables $y=x/r_n$ to obtain
\begin{eqnarray}\label{eq:liminf change ofc}
           \frac{\mu_n(\widetilde{Q}_n)}{r_n^{d-1}}=\frac{F_{\e_n}(u_n,\widetilde{Q}_n)}{r_n^{d-1}}=\int_{\widetilde{Q}}\Bigl[\frac{1}{\tau_n}W(v_n)+\sum_{\ell=1}^{k}q_{\ell}\tau_n^{2\ell-1}|\nabla^{(\ell)}v_n|_\ell^2\Bigr]\, dy.        \end{eqnarray}
            Since $\{v_n\}_n$ satisfies all the hypotheses of Lemma  \ref{lemma:modifying sequences}, we can apply such Lemma to obtain  a sequence $\{w_n\}_n\subset H^k(\widetilde{Q})$ which satisfies {\it (i)}-{\it (iii)} of said Lemma.
            
\noindent Noting that $\tau_n\to 0$ as $n\to +\infty$, by \eqref{eq:besicovitch 2}, \eqref{eq:liminf change ofc} and {\it (iii)} of Lemma \ref{lemma:modifying sequences}, it follows 
\begin{equation}\label{liming gbarra} 
f(0)\geq \liminf_{n\to+\infty}\int_{\widetilde{Q}}\Bigl[\frac{1}{\tau_n}W(w_n)+\sum_{\ell=1}^kq_{\ell}\tau_n^{2\ell-1}|\nabla^{(\ell)}w_n|_\ell^2\Bigr]\,dy.
\end{equation}
Now we observe that, by {\it(ii)} of Lemma \ref{lemma:modifying sequences}, putting
\begin{equation*}
    w_n:= \overline{u}^{e_d}_{\tau_n}
\quad \text{ on }\,\, Q_1^{e_d}(0)\setminus \widetilde{Q}
\end{equation*}
for all $n\in\N$, we have $\{w_n\}_n\subset \mathcal{A}^{e_d}_{\tau_n}$. Moreover, 
since  $|\overline{u}^{e_d}(y)|=1$ if $y\in \{z:|z\cdot e_d|>\frac{1}{8}\}$, it follows that $|\overline{u}_{\tau_n}^{e_d}(y)|=1$ if $y\in \{z:|z\cdot e_d|> \frac{\tau_n}{8}\}$.

\noindent Therefore, by \eqref{Q rectangle}, we have
\begin{equation*}
f(0)\geq \liminf_{n\to+\infty}\int_{Q_1^{e_d}(0)}\Bigl[\frac{1}{\tau_n}W(w_n)+\sum_{i=1}^kq_{\ell}\tau_n^{2\ell-1}|\nabla^{(\ell)}w_n|_\ell^2\Bigr]\,dy,
\end{equation*}
and by the definition of $g$ in \eqref{density}, we conclude
\begin{equation*}
     f(0)\geq g(e_d).
\end{equation*}
Since the above argument holds for $\Hd$-a.e.\ $x\in \partial^*E \cap \Omega$ and every $\nu=\nu_E(x)$, we obtain 
 \begin{equation}\label{ultimaliminf}
     f(x)\geq g(\nu_E(x)).
 \end{equation}
Finally, by the inequality $\liminf_{n}\mu_n(\Omega)\geq \mu(\Omega)$ we infer that
\begin{equation*}
    \lim_{n\to+\infty} F_{\e_n}(u_n, \Omega)= \lim_{n\to +\infty}\mu_{n}(\Omega)\geq \mu(\Omega)=\int_{\partial^*\!E\cap \Omega}f(x)\, d\Hd(x)+\sigma(\Omega) 
\end{equation*}
and, recalling that $\sigma(\Omega)\geq 0$ and exploiting \eqref{ultimaliminf}, this last inequality leads to
$$
\lim_{n\to+\infty} F_{\e_n}(u_n, \Omega)\geq\int_{\partial^*\!E\cap \Omega}f(x)\, d\Hd(x)\geq\int_{\partial^*\!E\cap\Omega}g(\nu_E(x)) d\Hd(x),
$$
concluding the proof.
\end{proof}

\begin{remark} \label{rmk: liminf}
{\rm 
We note that, thanks to \eqref{liming gbarra}, as a by-product of the above proof we have also showed that
\begin{equation}\notag
    \Gamma\text{-}\liminf_{\e\to0} F_\e(u, \Omega)\geq
    \begin{cases}
\displaystyle\int_{\partial^*\!A\cap \Omega} \overline{g}(\nu_A)\, d\Hd & \text{ if } u\in BV(\Omega;\{-1,1\}),\, \text{ with } u=2\chi_A-1, \\
  +\infty & \text{ if } u\in L^2(\Omega)\setminus BV(\Omega;\{-1,1\}),
  \end{cases}
    \end{equation}
where
 \begin{equation*}
      \overline{g}(\nu):=\inf\Big\{\int_{{\widetilde{Q}^\nu}} \Bigl[\frac{1}{\e}W(v)+\sum_{\ell=1}^{k}q_\ell\e^{2\ell-1} |\nabla^{(\ell)}v|_\ell^2\Bigr]\,dx  :  v \in \overline{\A}^\nu_\e,\,\e\in(0,1)\Big\}
 \end{equation*}
 for every $\nu\in\Sd$ and we have set
\begin{equation*}
\overline{\A}_\e^\nu:=\bigl\{v\in H^k(\widetilde{Q}^\nu): v= \overline{u}^\nu_\e \text{ near } \partial \widetilde{Q}^\nu \bigr\}.
\end{equation*}
Such observation will be useful in order to prove that the $\Gamma$-limit is strictly positive. }
\end{remark}

\section{Upper bound}

In this section we conclude the proof of Theorem \ref{main theorem} by proving the limsup inequality. As in Section \ref{section lower bound}, $C$ denotes a generic positive constant that may change from line to line, but which is independent of the relevant parameters. We now state the main result of the section.
\begin{proposition}\label{gammalimsup}
Let $\Omega\subset \Rd$ 
be a bounded, open set with $C^1$ boundary. Let $\{\e_n\}_n$ be a positive sequence converging to $0$ and let $u\in BV(\Omega;\{- 1,+1\})$.  Assume that {\rm (H1)-(H3)} are satisfied.   Then there exists a sequence $\{u_n\}_n\subset H^k(\Omega) $ converging to $u$ in $L^2(\Omega)$ as $n\to+\infty$ and such that
    \begin{equation}\notag
        \limsup_{n\to+\infty}F_{\e_n}(u_n)\leq F_0(u).
    \end{equation}
\end{proposition}
\begin{proof}
We subdivide the proof in three parts. 

{\it Step 1.}
First, we consider the case where $\Omega$ has Lipschitz boundary (and not necessarily $C^1$) and $u=2\chi_A - 1$, where 
\begin{align*}
    A:=\{x \in \Omega: (x-x_0)\cdot \nu >0\}
\end{align*}
for some $x_0 \in \Omega$ and $\nu \in \Sd$ such that
\begin{equation}\label{transv}
    \nu_\Omega(x) \perp \nu \ \text{ for every } x \in \partial \Omega \text{ with } |(x-x_0)\cdot \nu| \text{ small enough},
\end{equation}
recalling that $\nu_\Omega(x)$ is the inner unit normal to $\Omega$ at $x$. For the sake of exposition, we assume $x_0=0$ and $\nu=e_d$, the other cases being analogous, so that for every $x\in\Omega$ we have $u(x)=u_0^{e_d}(x)$, where by definition $u_0^{e_d}(x)=\text{sgn}(x\cdot e_d)$, and 
\begin{equation*}
    F_0(u)=g(e_d)\mathcal{H}^{d-1}\left (\Omega \cap \{x_d=0\}\right).
\end{equation*}

For every $\eta>0$, there exist $0<\e_0<1$ and $v\in \A_{\e_0}^{e_d}$ such that $v=\overline{u}^{e_d}_{\e_0}$ in a neighborhood of $\partial Q^{e_d}_1$, and
    \begin{equation}\label{minv}
        \int_{Q_1^{e_d}}\Bigl[\frac{1}{\e_0}W(v)+ \sum_{\ell=1}^k q_\ell \e_0^{2\ell-1} |\nabla^{(\ell)} v|_{\ell}^2\Bigr]dx \le g(e_d) + \eta.
    \end{equation} 
By Proposition \ref{densitywelldef}, upon rotating $Q^{e_d}_1$ leaving $e_d$ fixed, we may assume that $Q^{e_d}_1=Q:=(-\frac{1}{2},\frac{1}{2})^d$. We then extend  $v$ to a function in $H^k_{\rm loc}(\Rd)$ which is periodic in the first $d-1$ variables; that is, such that 
\begin{equation}\label{v is periodic step1}
    v(x)=v(x+e_i),
\end{equation}
for all $i\in\{1,...,d-1\}$. We also set $\Omega_n:=\{x\in\Omega:|x_d|\leq \tfrac{\e_n}{2\e_0}\}$ for every $n\in\N$.
    
We claim that a recovery sequence for $u$ is given by the sequence of functions $\{u_n\}\subset H^k(\Omega)$ defined by
\begin{equation}\label{def recovery step 1}
    u_n(x):=
    \begin{cases}
     v\big(\frac{\e_0x }{\e_n}\big) & \text{if }  x \in \Omega_n ,\\
    u(x) & \text{if }   x \in \Omega\setminus \Omega_n.
    \end{cases}   
    \end{equation}
To prove the claim, we first show that
\begin{equation}\label{L2 convergence step 1}
    \|u_n -u\| _{L^2(\Omega)}\to 0 \text{ as } n \to +\infty.
\end{equation}
To this aim, we set $\Omega' := \big\{x' \in \R^{d-1} : (x',0)\in \Omega\big\}$ and observe that by (\ref{transv}), it holds 
\begin{equation*}\Omega \cap\Omega_n= \Omega'\times \Bigl[-\frac{\e_n}{2\e_0}, \frac{\e_n}{2\e_0}\Bigr]\end{equation*}
for $n$ sufficiently large. 

\noindent Using the change of variables $(x',x_d)\to (x', \frac{\e_n}{\e_0} x_d)$, by Fubini's Theorem and \eqref{def recovery step 1} we have 
\begin{align*}
  \int_{\Omega_n}|u_n |^2\, dx&= \int_{\Omega'}\int_{-\frac{\e_n}{2\e_0}}^\frac{\e_n}{2\e_0} \big|v\Bigl(\frac{\e_0 x'}{\e_n},\frac{\e_0 x_d}{\e_n} \Bigr)\big|^2 dx_d dx'  \\
   & = \frac{\e_n}{\e_0}  \int_{-\frac{1}{2}}^{\frac{1}{2}} \int_{\Omega'} \big|v\Bigl(\frac{\e_0 x'}{\e_n}, x_d\Bigr) \big|^2 dx' dx_d.
\end{align*}
We observe that, since $v\in H^k_{\rm loc}(\Rd)$, the function  $f\colon \R^{d-1} \to \R$ defined  by  \begin{equation*}
    f(y)=\int_{-\frac{1}{2}}^\frac{1}{2} \left|v(y, x_d)\right|^2 dx_d\quad \text{ for every } y\in\R^{d-1}
\end{equation*} is integrable on $\Omega'$ and that, by definition of $v$,
it is periodic. We may then apply the Riemann-Lebesgue Lemma to the function $f$ to obtain
\begin{align}\notag
    \lim_{n\to +\infty } \int_{\Omega'} \int_{-\frac{1}{2}}^{\frac{1}{2}} \big|v\Big(\frac{\e_0x'}{\e_n}, x_d\Big) \big|^2 dx_d \,dx' 
    &=\int_{\Omega'} \int_{(-\frac{1}{2},\frac{1}{2})^{d-1}} \int_{-\frac{1}{2}}^\frac{1}{2} \left| v(y', x_d)\right|^2 dx_d \,dy'\, dx'\\\label{primo Riemann}
       &=\mathcal{L}^{d-1}(\Omega ')\| v\|_{L^2(Q)}^2,
\end{align}
which implies 
\begin{equation}\label{bound recory with target}
   \int_{\Omega_n} |u_n |^2\,dx \leq \frac{\e_n}{\e_0}\|v\|^2_{L^2(Q)}.
\end{equation}
From this inequality, \eqref{def recovery step 1}, and the fact that $u$ takes values only in $\{\pm 1\}$, we deduce that
\begin{align*}
    \|u_n -u\| _{L^2(\Omega)} ^2 
    \le  2\int_{\Omega_n}|u_n |^2\,dx  + 2\mathcal{L}^{d}(\Omega_n)\leq C\frac{\e_n}{\e_0}+2\mathcal{L}^{d}( \Omega_n),
\end{align*}
which tends to $0$ as $n\to +\infty$. This inequality proves \eqref{L2 convergence step 1}.

We now show that $\{u_n\}_n$ is a recovery sequence for $u$. To this aim, we first note that for every $x\in \Omega\setminus \Omega_n$ we have
\begin{equation*}
    u_n(x)=u(x) \in \{-1, 1\} \,\,\text{ and } \,\,\nabla ^{(\ell)}u_n(x)=0 \ \text{ for every }\ell\in\{1,...,k\},
\end{equation*}
while, if $x\in\Omega_n$, we have
\begin{equation*}
  \nabla ^{(\ell)} u_n(x)= \Big(\frac{\e_0}{\e_n}\Big)^{\ell} \nabla^{(\ell)} v\Big(\frac{\e_0 x}{\e_n}\Big) \quad \text{for every }\ell\in\{1,...,k\}.
\end{equation*}
Then, we  use (H1), \eqref{def recovery step 1}, Fubini's Theorem, and a change of variable to obtain
\begin{align}\nonumber
    F_{\e_n}(u_n)&= \int_{\Omega_n } \Bigl[\frac{1}{\e_n}W(u_n) + \sum_{\ell=1}^k q_\ell \e_n^{2\ell-1} |\nabla^{(\ell)} u_n|_\ell^2 \Bigr]dx\\
    &=\nonumber \int_{\Omega'} \int_{-\frac{\e_n}{2\e_0}}^{\frac{\e_n}{2\e_0}} \Bigl[ \frac{1}{\e_n} W\Big(v\Big(\frac{\e_0 x'}{\e_n},\frac{\e_0 x_d}{\e_n}\Big)\Big) + \sum_{\ell=1}^k q_\ell \e_n^{2\ell} \big|\nabla^{(\ell)} v\Big(\frac{\e_0 x}{\e_n},\frac{\e_0 x_d}{\e_n}\Big)\big|_\ell^2 \Bigr] dx_d \, dx'\\\label{conclusion step 1} 
    &=\int_{\Omega '} \int_{-\frac{1}{2}}^{\frac{1}{2}}\Bigl[ \frac{1}{\e_0}W \Big(v\Big(\frac{\e_0x'}{\e_n}, x_d \Big) \Bigl) + \sum_{\ell=1}^k q_\ell\e_0^{2\ell-1} \big |\nabla^{(\ell)} v\Big(\frac{\e_0 x'}{\e_n}, x_d\Big)\big |_\ell^2 \Bigr] dx_d dx'
\end{align}
for every $n\in\N$.
Since $v$ satisfies \eqref{minv}, the function $h\colon\R^{d-1}\to \Rd$ defined for $y\in\R^{d-1}$ by
\begin{equation*}
h(y):= \int_{-\frac{1}{2}}^\frac{1}{2} \Bigl[\frac{1}{\e_0}W(v(y', x_d))+ \sum_{\ell=1}^k q_\ell \e_0^{2\ell-1} |\nabla^{(\ell)} v (y', x_d)|_\ell^2 \Bigr] dx_d 
\end{equation*}
is integrable and, by \eqref{v is periodic step1}, is also periodic. Therefore, we may apply once again the Riemann-Lebesgue Lemma to obtain
\begin{align*}
    \lim_{n \to +\infty} \int_{\Omega '}  &\int_{-\frac{1}{2}}^\frac{1}{2}  \Bigl[\frac{1}{\e_0}W\Big(v\Big(\frac{\e_0x'}{\e_n}, x_d\Big)\Big)+ \sum_{\ell=1}^k q_\ell \e_0^{2\ell-1} \big |\nabla^{(\ell)} v\Big(\frac{\e_0x'}{\e_n}, x_d\Big)\big|_\ell^2 \Bigr] dx_d\, dx' \\
   &= \mathcal{L}^{d-1}(\Omega')\int_{Q}\Bigl[\frac{1}{\e_0}W(v)+ \sum_{\ell=1}^k q_\ell \e_0^{2\ell-1} |\nabla^{(\ell)} v|_{\ell}^2\Bigr]dx.
\end{align*}
Combining this equality with \eqref{conclusion step 1} and (\ref{minv}), we conclude that
\begin{align*}
    \lim_{n\to +\infty }F_{\e_n}(u_n)=
   \mathcal{L}^{d-1}(\Omega')\int_{Q}\Bigl[\frac{1}{\e_0}W(v)+ \sum_{\ell=1}^k q_\ell \e_0^{2\ell-1} |\nabla^{(\ell)} v|_{\ell}^2\Bigr]dx
   \le F_0(u) + \mathcal{L}^{d-1}(\Omega ')\eta.
\end{align*}
Since $\eta$ was arbitrary, a diagonal argument allows to prove the thesis in this special case.

 {\it Step 2.}
We now assume that $u=2\chi_A- 1$, with $A=P\cap \Omega$ and $P$ polyhedral. This means that $\partial P=\bigcup_{i=1}^M H_i \cup F$, where the sets $(H_i)_{i=1}^M$ are pairwise disjoint and relatively open convex polyhedra of dimension $d-1$, while $F$ is the union of a finite number of convex polyhedra of dimension $d-2$. In particular, there exists $\nu_1,..., \nu_M \in \Sd$ and $x_1,..., x_M \in \R^d$ such that
$$
H_i \subseteq \{x \in \R^d: (x-x_i) \cdot \nu_i =0   \},
$$
and each $\nu_i$ is the inner unit normal to $\partial P$ on $H_i$. For simplicity,
we may also assume that $\partial \Omega \cap \partial P$ is the union of a finite number of $C^1$ manifolds of dimension $d-2$.

Let us fix $0<\sigma<1$ and  let $H_1',..., H_M'$ relatively open subsets of $H_1,...,H_M$ with $(d-2)$-dimensional boundary of class $C^1$, such that
\begin{equation}\notag
    \big\{ x \in H_i \cap \Omega: \text{dist}(x, \partial \Omega \cup F) \ge \tfrac{\sigma}{2}\big\} \subseteq H_i' \subseteq \overline{H_i'}\subseteq H_i.
\end{equation}
Fix $0<\eta< \frac{\sigma}{2}$ and define for each $i\in\{1,...,M\}$
\begin{equation}\notag
    \Omega_i:=\{x + t\nu_i: x \in H_i', \,t \in (-\eta, \eta) \}.
\end{equation}
We can choose $\eta$ so small  that the sets $\Omega_1,..., \Omega_M$ are pairwise disjoint. 
Since each $\Omega_i,$ satisfies (\ref{transv}), by the previous step, we obtain sequences $\{u_n^i\}_n\subset H^k(\Omega_i)$ such that
\begin{equation}\label{convergence L2 on Omegai}
    u_n^i \to u \ \text{in} \ L^2(\Omega_i),
\end{equation}
\begin{equation}\notag
    u_n^i(x)= u(x) \,\,\text{ for every } x\in\big\{x \in \overline \Omega _i:  \text{dist}(x, H_i) \ge \tfrac{\e_n}{2\e^i_0}\big\},
\end{equation}
and 
\begin{equation}\label{recoveri aux step 2}
    \limsup_{n\to\infty} F_{\e_n} (u_n^i, \Omega_i) \le \left(g(\nu_i)+\eta\right) \mathcal{H}^{d-1}(\Omega_i \cap H_i),        
\end{equation}
where for every $x\in\Omega_i$ it holds \begin{equation*}
    u^i_n(x):=
    \begin{cases}
     v^i\big(\frac{\e^i_0x }{\e_n}\big) & \text{if }  |(x-x_i)\cdot \nu_i|\leq\frac{\e_n}{2\e_0^i}, \\
    u(x) & \text{if }   |(x-x_i)\cdot \nu_i|>\frac{\e_n}{2\e_0^i},
    \end{cases}   
    \end{equation*}
and, for each $i\in\{1,...,M\}$, $v^i$ and $\e_0^i$ are such that 
\begin{equation}\notag 
        \int_{Q_1^{\nu_i}}\Bigl[\frac{1}{\e^i_0}W(v)+ \sum_{\ell=1}^k q_\ell (\e_0^i)^{2\ell-1} |\nabla^{(\ell)} v|_{\ell}^2\Bigr]dx \le g(\nu_i) + \eta.
    \end{equation} 
    
In order to present the candidate recovery sequence, we let $(\Psi_\e)_{\e>0}$ be a radial smooth convolution kernel, with $\text{supp}\Psi_\e\subset B_\e(0)$ and $\int\Psi_\e\,dx=1$ for every $\e>0$.  
We extend  $u$ to the whole $\R^d$ by setting $u:=2\chi_{P}-1$ and introduce  smooth functions $\widetilde{u}_n$ defined for every $n\in\N$ by 
\begin{equation}\notag
    \widetilde{u}_n:= u * \Psi_{\e_n},
\end{equation}
where $*$ denotes the operation of convolution.
Observe that these functions satisfy the 
following conditions
\begin{gather}\label{prop2bar}
     \widetilde{u}_n(x)= u(x) \,\text{ for every } x\in \{x\in \Omega:\text{dist}(x, \partial P) \ge {\e_n}\},\\\label{prop2bar 2}
    \|\widetilde{u}_n\|_{L^\infty(\Rd)}\leq 1\,\,\,\text{ and }\,\,\, \||\nabla ^{(\ell)}\widetilde{u}_n|_\ell\|_{L^\infty(\Rd)}\leq \frac{C}{\e^{\ell}} \text{ for every }\ell\in\{1,...,k\}. 
\end{gather}    
For $\tau\in(0,1)$ we set 
\begin{equation}\notag
U_\tau:=\{x \in \Omega: \text{dist}(x, F \cup (\partial \Omega \cap \partial P))\le \tau \}
\end{equation}
and choose cut-off functions $\phi_{\sigma}\in C^\infty_c(\Rd; [0,1])$ such that 
\begin{align}\label{properties cut off step2}
    \hspace{2.5 cm}\begin{cases}
    \phi_{\sigma}(x)=1\quad &\text{if }x\in U_\sigma,\\    
    \phi_{\sigma}(x)=0\quad &\text{if }x\in \R^d \setminus U_{2\sigma},\\
\||\nabla^{(\ell)}\phi_{\sigma}|_\ell\|_{L^\infty(\R^d)}\leq C\sigma^{-\ell}\quad &\text{ for every }\ell\in\{1,...,k\}.
    \end{cases}
\end{align}
We claim that the sequence of functions $\{u_{n}\}_n\subset H^k(\Omega)$ defined by 
\begin{align}\notag
    \hspace{2.5 cm}
    u_n:=\begin{cases}   
    (1-\phi_\sigma )u_n^i+\phi_\sigma\widetilde{u}_n\quad &\text{in } \overline{\Omega}_i \text{ for } \ i\in\{1,...M\} ,\\
    \widetilde {u}_n \quad &\text{in } \Omega':=\Omega \setminus \cup_{i=1}^M \overline{\Omega}_i 
    \end{cases}
\end{align}
is a recovery sequence for $u$.

The fact that $\{u_n\}_n\subset H^k(\Omega)$ follows from the same argument used in \cite[Theorem 1.3, Substep 2B]{ChermisiDalMaso11}. Also, \eqref{convergence L2 on Omegai}
and the fact that $\widetilde{u}_n\to u$ in $L^2(\Omega)$ imply that $u_n\to u$ in $L^2(\Omega)$ as $n\to+\infty$.

We now study the asymptotic behaviour of the energy of $\{u_n\}$. First, we show that $F_{\e_n}(u_n,\Omega')$ is negligible. To this aim, for every $n\in\N$ we set 
\begin{equation}\label{def Rn}
  \mathcal{R}_n:= \Big\{x \in \Omega: \text{dist} (x, \partial P) \le \max \big\{\e_n, \frac{\e_n}{2\e^1_0},...,\frac{\e_n}{2\e^M_0}  \big\} \Big\}.\end{equation}   We observe that $\Omega'\cap\mathcal{R}_n\subset U_\sigma$ for $n$ large enough and that  $\Hd(\partial P\cap U_\sigma)\leq C\sigma$. Hence,  by the definition of $\Omega'$, \eqref{prop2bar}, and \eqref{prop2bar 2}, we have
\begin{equation}\label{R_n}
    F_{\e_n}(u_n,\Omega') = \int_{\Omega'\cap \mathcal{R}_n}\Bigl[\frac{1}{\e_n}W(\widetilde{u}_n)+ \sum_{\ell=1}^k |q_\ell|\e_n^{2\ell-1} |\nabla^{(\ell)} \widetilde{u}_n|_{\ell}^2\Bigr]dx\leq C\sigma.
        \end{equation}

We now study the energetic contribution of the remaining terms. We fix $i\in\{1,...,M\}$ and compute
\begin{align}\notag
    F_{\e_n}(u_n, \Omega_i)&=F_{\e_n}(u_n^i, \Omega_i \setminus U_{2\sigma})+ F_{\e_n}(u_n, \Omega_i \cap (U_{2\sigma}\setminus U_\sigma))+F_{\e_n}(\widetilde{u}_n, \Omega_i \cap U_{\sigma})\\\notag
    &\le F_{\e_n}(u_n^i, \Omega_i \setminus U_{2\sigma})
    + \int_{\Omega_i \cap (U_{2\sigma}\setminus U_\sigma)} \Bigl[\frac{1}{\e_n}W(u_n)+\sum_{\ell=1}^{k}|q_\ell|\e_n^{2\ell-1}|\nabla^{(\ell)}u_n|_\ell^2\Bigr]\,dx \\\notag
    &\quad\,\,\,+\int_{\Omega_i \cap U_\sigma} \Bigl[\frac{1}{\e_n}W(\widetilde{u}_n)+\sum_{\ell=1}^{k}|q_\ell|\e_n^{2\ell-1}|\nabla^{(\ell)}\widetilde{u}_n|_\ell^2\Bigr]\,dx\\\notag
    &=: \mathcal{K}_n^1 + \mathcal{K}_n^2 + \mathcal{K}_n^3.
\end{align}
We separately investigate the three terms $\mathcal{K}^1_n,$ $ \mathcal{K}^2_n$, and $\mathcal{K}^3_n$, starting from $\mathcal{K}^1_n$. Recalling that $q_k=1$, we observe that 
\begin{align}\notag
    \mathcal{K}_n^1 &= \int_{\Omega_i \setminus U_{2\sigma}} \Bigl[\frac{1}{\e_n}W(u_n^i) + \e_n^{2k-1}|\nabla^{(k)}u_n^i|_k^2\Bigr]\,dx 
    + \int_{\Omega_i \setminus U_{2\sigma}}  \Bigl[\sum_{\ell=1}^{k-1} q_\ell \e_n
    ^{2\ell-1}|\nabla^{(\ell)}{u}_n^i|_\ell^2\Bigr]\,dx \\\notag
    & \leq \int_{\Omega_i } \Bigl[\frac{1}{\e_n}W(u_n^i) + \e_n^{2k-1}|\nabla^{(k)}u_n^i|_k^2\Bigr]\,dx +\int_{\Omega_i }  \Bigl[\sum_{\ell=1}^{k-1} q_\ell \e_n^{2\ell-1}|\nabla^{(\ell)}u_n^i|_\ell^2\Bigr]\,dx\\\notag
    &\,\quad- \int_{\Omega_i \cap U_{2\sigma}}  \Bigl[\sum_{\ell=1}^{k-1} q_\ell \e_n^{2\ell-1}|\nabla^{(\ell)}{u}_n^i|_\ell^2\Bigr]\,dx\\\label{conti K1}
    & \le F_{\e_n}(u_n^i, \Omega_i) + \int_{\Omega_i \cap U_{2\sigma}}  \Bigl[\sum_{\ell=1}^{k-1} |q_\ell| \e_n^{2\ell-1}|\nabla^{(\ell)}{u}_n^i|_\ell^2\Bigr]\,dx.
\end{align}
 By the periodicity of $u_n^i$ with respect to the variables tangential to $H_i$, we may argue as in {\it Step 1}, using Fubini's Theorem and the Riemann-Lebesgue Lemma to deduce that 
\begin{align}\notag
    \lim_{n\to +\infty} \int_{\Omega_i \cap U_{2\sigma}}  \Bigl[\sum_{\ell=1}^{k-1} |q_\ell| \e_n^{2\ell-1}|\nabla^{(\ell)}{u}_n^i|_\ell^2\Bigr]\,dx\leq C\sigma\int_{Q_1^{\nu_i}}\Big[\sum_{\ell=1}^{k-1}|\nabla^{(\ell)}v^i|_\ell^2\Big]\, dx\leq C\sigma.
\end{align}
Hence, it follows from \eqref{conti K1} and \eqref{recoveri aux step 2} that 
\begin{equation}\label{stima K1}
    \limsup_{n\to+\infty}\mathcal{K}_n^1\leq (g(\nu_i)+\eta)\Hd(H_i\cap \Omega_i)+C\sigma.
\end{equation}

We now investigate the term $\mathcal{K}^2_n$. We begin observing that for every  $\ell\in\{1,...,k\}$ and multi-index $\alpha$ of order $|\alpha|=\ell$, we have 
\begin{align}\nonumber
\partial^\alpha u_n&=\sum_{\beta\leq \alpha}\binom{\alpha}{\beta}\Big(\partial^\beta(1-\phi_\sigma)\partial^{\alpha-\beta}u_n^i+\partial^\beta\phi_{\sigma}\partial^{\alpha-\beta}\widetilde{u}_n\Big)\\\nonumber
&=(1-\phi_{\sigma})\partial^\alpha u^i_n + \phi_{\sigma}\partial^\alpha\widetilde{u}_n+ \sum_{0<\beta\leq\alpha}\binom{\alpha}{\beta}\big(\partial^\beta\phi_{\sigma}\,\partial^{\alpha-\beta}(\widetilde{u}_n-u_n^i)\big),
\end{align} 
which by \eqref{properties cut off step2} implies
\begin{align}\notag
    |\nabla^{(\ell)}u_n|^2_\ell\leq C&\Bigl(|\nabla^{(\ell)}u^i_n|_\ell^2
    + |\nabla^{(\ell)}\widetilde{u}_n|_\ell^2 \\ \label{K_2derivate}
    +&\sigma^{-2\ell}|\widetilde{u}_n-
    u^i_n|^2+\sum_{j=1}^{\ell-1}\sigma^{-2(\ell-j)}|\nabla^{(j)}(\widetilde{u}_n-u^i_n)|^2_j\Bigr).
\end{align}
Since by (H$3$) we have that
\begin{equation}\label{K_2pozzo}
    W(u_n) \leq W(u^i_n)+C,
\end{equation}
 it follows from \eqref{def Rn}, \eqref{K_2derivate}, and \eqref{K_2pozzo} that
\begin{align*}
\mathcal{K}_n^2 & \leq  C\int_{\Omega_i \cap (U_{2\sigma}\setminus U_\sigma)\cap \mathcal{R}_n} \Big[\frac{1}{\e_n}W(u^i_n)+\frac{1}{\e_n}+\sum_{\ell=1}^{k}|q_\ell|\e_n^{2\ell-1}\Big(|\nabla^{(\ell)}u_n^i|_\ell^2+\sum_{j=1}^{\ell-1}\sigma^{-2(\ell-j)}|\nabla^{(j)}u^i_n|^2_j\Big)\Big]\,dx\\
&\quad+ C\int_{\Omega_i \cap (U_{2\sigma}\setminus U_\sigma)\cap \mathcal{R}_n} \!\!\Big[\sum_{\ell=1}^{k}|q_\ell|\e_n^{2\ell-1}\Big(|\nabla^{(\ell)}\widetilde{u}_n|_\ell^2+\sigma^{-2\ell}|u^i_n-\widetilde{u}_n|^2+\sum_{j=1}^{\ell-1}\sigma^{-2(\ell-j)}|\nabla^{(j)}\widetilde{u}_n|^2_j\Big)\Big]\,dx.
\end{align*}
 Arguing as in \eqref{primo Riemann} and \eqref{conclusion step 1}, we obtain that
\begin{align}\notag
  \lim_{n\to+\infty}  \int_{\Omega_i \cap (U_{2\sigma}\setminus U_\sigma)\cap \mathcal{R}_n} &\Big[\frac{1}{\e_n}W(u^i_n)+\frac{1}{\e_n}+\sum_{\ell=1}^{k}|q_\ell|\e_n^{2\ell-1}\Big(|\nabla^{(\ell)}u_n^i|_\ell^2+\sum_{j=1}^{\ell-1}\sigma^{-2(\ell-j)}|\nabla^{(j)}u^i_n|^2_j\Big)\Big]\,dx \\ 
    &\leq C\mathcal{H}^{d-1}(H_i \cap U_{2\sigma}) \le C\sigma.\label{K_2pozzobis}
\end{align}
 Then, we observe that using \eqref{bound recory with target}, with $u_n$ replaced by $u_n^i$, $v$ replaced by $v^i$, and $\Omega$ replaced by $\Omega_i$, together with  \eqref{prop2bar 2}, we have 
\begin{align}\notag
    \int_{\Omega_i \cap (U_{2\sigma}\setminus U_\sigma)\cap \mathcal{R}_n} &\Big[\sum_{\ell=1}^{k}|q_\ell|\e_n^{2\ell-1}\Big(|\nabla^{(\ell)}\widetilde{u}_n|_\ell^2+\sigma^{-2\ell}|\widetilde{u}_n-u^i_n|^2+\sum_{j=1}^{\ell-1}\sigma^{-2(\ell-j)}|\nabla^{(j)}\widetilde{u}_n|^2_j\Big)\Big]\,dx \\&\leq \notag
   \frac{C}{\e_n}\Big(\sum_{\ell=1}^k\Big(1+\sum_{j=1}^{\ell-1}\frac{\e_n^{2(\ell-j)}}
   {\sigma^{2(\ell-j)}}\Big)\Big)\Ld\big(\Omega_i \cap (U_{2\sigma}\setminus U_\sigma)\cap \mathcal{R}_n\big)\\\notag &\quad\,
+C\sum_{\ell=1}^k\Big(\frac{\e_n^{2\ell}}{\sigma^{2\ell}}\|v^i\|^2_{L^2(\Omega_i)}  
   + \frac{\e_n^{2\ell-1}}{\sigma^{2\ell}}\Ld\big(\Omega_i \cap (U_{2\sigma}\setminus U_\sigma)\cap \mathcal{R}_n\big)\Big) \\
   &\leq \label{conclusion step 2 2}C\Big(\sum_{\ell=1}^k\Big(1+\sum_{j=1}^{\ell-1}\frac{\e_n^{2(\ell-j)}}
   {\sigma^{2(\ell-j)}}\Big)\Big)\sigma+C\sum_{\ell=1}^k\Big(\frac{\e_n^{2\ell}}{\sigma^{2\ell}}+\frac{\e_n^{2\ell}}{\sigma^{2\ell-1}}\Big).
\end{align}
As \begin{equation*}
\sum_{\ell=1}^k\sum_{j=1}^{\ell-1}\frac{\e_n^{2(\ell-j)}}
   {\sigma^{2(\ell-j)}}+\sum_{\ell=1}^k\Big(\frac{\e_n^{2\ell}}{\sigma^{2\ell}}+\frac{\e_n^{2\ell}}{\sigma^{2\ell-1}}\Big)\longrightarrow 0\quad\text{ for $n\to+\infty$},
\end{equation*}
combining \eqref{K_2pozzobis} with \eqref{conclusion step 2 2}, we obtain that 
\begin{equation}\label{Stima K2}
    \limsup_{n\to+\infty }\mathcal{K}^2_n\leq C\sigma.
\end{equation} 

To conclude, we use  \eqref{prop2bar 2} to  estimate the remaining term $\mathcal{K}^3_n$ as 
\begin{align}\notag
    \mathcal{K}_n^3 &\le \int_{ U_\sigma\cap \{x:|(x-x_i)\cdot \nu_i|\le \e_n\}} \Bigl[\frac{1}{\e_n}W(\widetilde{u}_n)+\sum_{\ell=1}^{k}|q_\ell|\e_n^{2\ell-1}|\nabla^{(\ell)}\widetilde{u}_n|_\ell^2\Bigr]\,dx\\\label{stima K3}
    & \le \frac{C}{\e_n}\mathcal{L}^d \left(\left\{x \in U_\sigma : |(x-x_i)\cdot \nu_i| \le \e_n\right \}\right)
    \le C \mathcal{H}^{d-1} (U_\sigma \cap H_i)
    \le C\sigma.
\end{align}
Finally, gathering \eqref{R_n}, \eqref{stima K1},\eqref{Stima K2},  \eqref{stima K3}, and summing over $i\in\{1,...,M\}$ we obtain that 
\begin{align*}
    \limsup_{n\to+\infty}F_{\e_n}(u_n, \Omega) & \leq \sum_{i=1}^M\limsup_{n\to+\infty} F_{\e_n}(u_n,\Omega_i)+\limsup_{n\to+\infty}F_{\e_n}(u_n,\Omega')\\
    &\ =\sum_{i=1}^m(g(\nu_i)+\eta)\Hd(\Omega_i\cap H_i)+C\sigma\leq F_0(u)+C(\eta+\sigma).
\end{align*}
Since $\eta>0$ and $\sigma>0$ were arbitrary, a diagonal arguments allows to conclude.

{\it Step 3.}
Finally, we deal with the general case  $u=2\chi_A -1$, for $A\subseteq \Omega$ a set of relatively finite perimeter in $\Omega$. By Theorem \ref{smoothapproximation}, we can find a sequence of polyhedral sets $\{A_n\}_n$ such that $\chi_{A_n} \to \chi_A$ in $L^1(\Omega)$. By the previous step, setting $u_n:= 2\chi_{A_n}-1$,
\begin{equation}\notag
 \Gamma\text{-} \limsup_{\e \to 0} F_\e(u_n) \le  F_0(u_n) \text{ for every } n \in \N,
\end{equation}
where we recall that the $\Gamma$-limsup is the functional defined by  \eqref{def Gammalimsup}. By the lower semicontinuity of the  $\Gamma$-limsup we  then infer
\begin{align}\notag
    \Gamma \text{-}\limsup_{\e \to 0} F_\e(u) \le \liminf _{n \to \infty} \ \Gamma \text{-}\limsup_{\e \to 0} F_\e(u_n) \le \liminf _{n \to \infty} F_0(u_n).
\end{align}
Since $g$ is upper semicontinous on $\Sd$ (see Proposition \ref{densitywelldef} above), by \cite[Remark 9.11, Theorem 9.15]{DalMaso} there exists a sequence $\{g_m\}_m$ of Lipschitz continuous functions such that $g \le g_{m+1}\le g_m \le \max_{\Sd} g $ and
$$
g(\nu)= \inf _{m \in \N} g_m(\nu) \ \text{ for every } \nu \in \Sd.$$
Using again Theorem \ref{smoothapproximation} and  (\ref{contlimit}), we obtain for every $m\in\N$ that 
\begin{align*}
    \lim_{n\to + \infty} \int_{\partial^*\! A_n \cap \Omega} g_m(\nu_{A_n}(x)) d\mathcal{H}^{d-1}(x)=\int_{\partial^* \!A \cap \Omega} g_m(\nu_{A}(x)) d\mathcal{H}^{d-1}(x)
\end{align*}
so that by monotonicity
\begin{align*}
    \liminf_{n \to \infty} F_0(u_n)
    & =\liminf_{n \to \infty} \int_{\partial^*\! A_n \cap \Omega} g(\nu_{A_n}(x)) d\mathcal{H}^{d-1}(x) \\
    & \le 
     \lim_{n \to \infty} \int_{\partial ^*\! A_n \cap \Omega} g_m(\nu_{A_n}(x)) d\mathcal{H}^{d-1}(x) \\
     & =\int_{\partial^*\!A \cap \Omega} g_m(\nu_{A}(x)) d\mathcal{H}^{d-1}(x).
\end{align*}
Finally, by Monotone Convergence we obtain that
\begin{align*}
     \liminf_{n \to \infty} F_0(u_n)  & \le \lim_{m \to \infty} \int_{\partial^* \!A \cap \Omega} g_m(\nu_{A}(x)) d\mathcal{H}^{d-1}(x)  \\
     & \le \int_{\partial^* \!A \cap \Omega} g(\nu_{A}(x)) d\mathcal{H}^{d-1}(x) = F_0(u),
\end{align*}
 which concludes the proof.
\end{proof}

We conclude this section proving that the surface energy density $g$ is strictly positive. This observation makes the $\Gamma$-limit non trivial.

\begin{proposition}\label{positivedensity} Let $g$ be the function defined in \eqref{density} with $k>1$ an integer and let $\mathcal{N}:=\{\ell\in\{1,...,k-1\}: q_\ell \leq 0\}$.  Assume that {\rm (H1)-(H3)} are satisfied, and, if  $\mathcal{N}\neq \emptyset$, further assume that $q_\ell >-\overline{q}_\ell$ for all $\ell\in\mathcal{N}$, where $\{\overline{q}_\ell : \ell\in\mathcal{N}\}$ denote the same positive constants appearing in Corollary \ref{interpolationwithnorms}. Then $\inf_{\Sd} g>0$.
 \end{proposition}
      \begin{proof}
Similarly to Section $5$, we introduce $\widetilde{Q}^\nu$ an open set with $C^2$ boundary contained in $Q^\nu$ such that
\begin{equation*}
    \widetilde{Q}^{\nu}\cap \{y\in\Rd:|y\cdot\nu|\leq \tfrac{1}{4}\}=Q^\nu\cap \{y\in\Rd:|y\cdot\nu|\leq \tfrac{1}{4}\}
         \end{equation*}
and, letting $\{\nu_1,...,\nu_{d-1},\nu\}$ denote an orthonormal basis describing $Q^\nu$,  with the additional property  that
\begin{equation}\label{Rnu}
    \widetilde{Q}^{\nu}\cap \bigcap_{i=1}^{d-1} \,\{y\in\Rd:|y\cdot\nu_i|\leq \tfrac{1}{4}\}=Q^\nu\cap \bigcap_{i=1}^{d-1} \,\{y\in\Rd:|y\cdot\nu_i|\leq \tfrac{1}{4}\}=:R^\nu.
         \end{equation}
Recalling the notation in Remark \ref{rmk: liminf}, we note that the function $\overline{g}:\Sd\to\R$ given by
 \begin{equation*}
      \overline{g}(\nu):=\inf\Big\{\int_{{\widetilde{Q}^\nu}} \Bigl[\frac{1}{\e}W(v)+\sum_{\ell=1}^{k}q_\ell\e^{2\ell-1} |\nabla^{(\ell)}v|_\ell^2\Bigr]\,dx  :  v \in \overline{\A}^\nu_\e,\e\in(0,1)\Big\}
 \end{equation*}
satisfies $\overline{g}(\nu)\geq g(\nu)$ for every $\nu\in\Sd$. We combine this observation with Remark \ref{rmk: liminf} and Proposition \ref{gammalimsup} to infer that
\begin{align*}
    \Gamma\text{-}\liminf_{\e\to0} F_\e(u, \Omega) & \geq
    \begin{cases}
\displaystyle\int_{\partial^*\!A\cap \Omega} \overline{g}(\nu_A)\, d\Hd & \text{ if } u\in BV(\Omega;\{-1,1\}),\, \text{ with } u=2\chi_A-1, \\
  +\infty & \text{ if } u\in L^2(\Omega)\setminus BV(\Omega;\{-1,1\}), 
  \end{cases} \\
  & \geq \begin{cases}
\displaystyle\int_{\partial^*\!A\cap \Omega} g(\nu_A)\, d\Hd & \text{ if } u\in BV(\Omega;\{-1,1\}),\, \text{ with } u=2\chi_A-1, \\
  +\infty & \text{ if } u\in L^2(\Omega)\setminus BV(\Omega;\{-1,1\})  
  \end{cases} \\
  & \geq \Gamma\text{-}\limsup_{\e\to0} F_\e(u, \Omega),
    \end{align*}
which, in turn, implies that $\overline{g}(\nu)= g(\nu)$ for every $\nu\in\Sd$. For this reason, we equivalently prove that $\inf_{\Sd} \overline{g}>0$.
    
Fix $\nu\in \Sd$, and let $\{\e_n\}_n$, $v_n\in \overline{\A}^\nu_{\e_n},$ $ n\in\N$ such that
\begin{equation*}
    \overline{g}(\nu)=\lim_{n\to+\infty} \int_{{\widetilde{Q}^\nu}} \Bigl[\frac{1}{\e_n}W(v_n)+\sum_{\ell=1}^{k}q_\ell\e_n^{2\ell-1} |\nabla^{(\ell)}v_n|_\ell^2\Bigr]\,dx.
\end{equation*}
Arguing as in Remark \ref{re.inftolim}, it is not restrictive to suppose $\e_n\to0$ and moreover, we may further assume that for all $n\in\N$ it holds $\e_n<\e_0(\{q_\ell:\ell\in\ \mathcal{N}\}, \widetilde{Q}^\nu)$, the constant appearing in Corollary \ref{interpolationwithnorms}.

Clearly, the set $R^\nu$ defined in \eqref{Rnu} is contained in $\widetilde{Q}^\nu$; then, we apply Corollary \ref{interpolationwithnorms} with $\Omega=\widetilde{Q}^\nu$ and, by the equivalence of the norms, we obtain that 
\begin{align*}
     \int_{\widetilde{Q}^\nu} \Bigl[\frac{1}{\e_n}W(v_n)+\sum_{\ell=1}^kq_\ell\e_n^{2\ell-1}|\nabla^{(\ell)}v_n|_\ell^2\Bigr]\,dx & \geq \delta \int_{\widetilde{Q}^\nu}\Bigl[\frac{1}{\e_n}W(v_n)+\sum_{\ell=1}^k\e_n^{2\ell-1}|\nabla^{(\ell)}v_n|_\ell^2\Bigr]\,dx \\
     & \geq \delta \int_{R^\nu}\Bigl[\frac{1}{\e_n}W(v_n)+\sum_{\ell=1}^k\e_n^{2\ell-1}|\nabla^{(\ell)}v_n|_\ell^2\Bigr]\,dx \\
     & \geq \delta' \int_{R^\nu}\Bigl[\frac{1}{\e_n}W(v_n)+\e_n^{2k-1}\|\nabla^{(k)}v_n\|_k^2\Bigr]\,dx
\end{align*}
for every $n\in\N$ and some $\delta'>0$ which is independent of $\widetilde{Q}^\nu$.
Applying Fubini's Theorem, \eqref{operatorial}, and a change of variable, we obtain
\begin{align*}
 \int_{R^\nu}\Bigl[\frac{1}{\e_n}W(v_n)+\e_n^{2k-1}\|\nabla^{(k)}v_n\|_k^2\Bigr]\,dx   & \geq \int_{R^\nu \cap \nu^\perp}\int_{-\frac{1}{2}}^{\frac{1}{2}}\Bigl[\frac{1}{\e_n} W(v_n^{\nu,y})+ \e_n^{2k-1}|(v_n^{\nu,y})^{(k)}|^2\Bigr]dt\,d\Hd(y)  \\
  & = \int_{R^\nu \cap \nu^\perp}\int_{-\frac{1}{2\e_n}}^{\frac{1}{2\e_n}} \bigl[W\bigl(w^{\nu,y}_{n})+ |(w_{n}^{\nu,y})^{(k)}|^2\bigr]\,dt\,\Hd(y) \\
  & \geq \frac{1}{2^{d-1}} \inf\Big\{ \int_{-\frac{1}{2\e_n}}^{\frac{1}{2\e_n}} \!\!\bigl[W(w){+}|w^{(k)}|^2\bigr]dt, w \in H^k\bigl(\bigl(-\frac{1}{2\e_n}, \frac{1}{2\e_n}\bigr)\bigr) ,  \\
        &  \qquad \qquad \quad  w(t)=\text{sgn}(t) \text{ if } |t|>M \text{ for some } M\in\Bigl(0,\frac{1}{2\e_n}\Bigr)\Bigl\},
\end{align*}
 where we set $w_{n}^{\nu,y}(t):= v_{n}^{\nu,y}(\e_n t)$ and in the last inequality we used that $v_n\in \overline{\A}_{\e_n}^\nu$ for all $n\in\N$. As (H$1$) and (H$2$) are satisfied, we are in position to apply Propositions $4.1$ and $4.2$ in \cite{BDS}, deducing that
\begin{equation*}
\overline{g}(\nu)\geq \liminf_{n\to+\infty}\delta' \int_{R^\nu}\Bigl[\frac{1}{\e_n}W(v_n)+\e_n^{2k-1}\|\nabla^{(k)}v_n\|_k^2\Bigr]\,dx  \geq \frac{\delta'}{2^{d-1}} m_k > 0, 
\end{equation*}
    $m_k$ being as in \eqref{minBDS}, which concludes the proof. 
    \end{proof}

\section{Analysis of the cell problem in dimension \texorpdfstring{1}{}}

In this final section we analyze the surface energy density $g$ defined in \eqref{density} when $d=1$. In this instance, such density clearly is a constant that we prove being determined by an optimal-profile problem on the real line. It is not restrictive to assume that $|\cdot|_\ell$ coincides with the absolute value for every $\ell\in\{1,...,k\}$ as any norm on $\R$ is a multiple of the euclidean one.

By Remark \ref{re.inftolim} and by a change of variable, we find that
\begin{align*}
    g= \inf_{T>\frac{1}{2}} \inf\Big\{\int_{-T}^{T} \Bigl[W(v)+\sum_{\ell=1}^{k}q_\ell |v^{(\ell)}|^2\Bigr]\,dt : &\ v \in  H^k((-T, T)), \\
    & v(t)=\text{sgn}(t) \text{ if } |t|>M \text{ for some } M\in(0,T)\Big\} \\
    & \hspace{-6.5cm} =\lim_{T\to+\infty} \inf\Big\{\int_{-T}^{T} \Bigl[W(v)+\sum_{\ell=1}^{k}q_\ell |v^{(\ell)}|^2\Bigr]\,dt : \ v \in  H^k((-T, T)), \\
    & v(t)=\text{sgn}(t) \text{ if } |t|>M \text{ for some } M\in(0,T)\Big\},
\end{align*}
hence, for $T>1/2$, we set 
\begin{align*}
     m(T):=\inf\Big\{\int_{-T}^{T} \Bigl[W(v)+\sum_{\ell=1}^{k}q_\ell |v^{(\ell)}|^2\Bigr]\,dx, & \ v \in  H^k((-T,T)), \\
       &v(t)=\text{sgn}(t) \text{ if } |t|>M \text{ for some } M\in(0,T)\Big\},
\end{align*}
in such a way that
\begin{equation*}
   g= \lim_{T\to+\infty} m(T),
\end{equation*}
and we put
\begin{equation}\label{optimalprofile}
    m:=\inf \Bigl\{ \int_{-\infty}^{+\infty}\Bigl[W(u)+\sum_{\ell=1}^kq_\ell|u^{(\ell)}|^2\Bigr]\,dt: u\in H^k_{\text{loc}}(\R), \lim_{t\to\pm\infty} u(t)=\pm1 \Bigr\}.
\end{equation}
The main result of this section is that, under the by now usual, suitable assumptions on the coefficients $q_1,...,q_{k-1}$, the equality $g=m$ holds. In order to prove this, we state the one-dimensional version of Corollary \ref{interpolationwithnorms} at scale $\e=1$. 

\begin{corollary}\label{interpolationwithnorms1d}  Let $I\subset \R$ be an open interval of length $1$, let $k>1$
be an integer, and let $\mathcal{N}:=\{\ell\in\{1,...,k-1\}: q_\ell \leq 0\}$.  Assume that {\rm (H2)} is satisfied, and, if $\mathcal{N}\neq \emptyset$, further assume that $q_\ell >-\overline{q}'_\ell$ for all $\ell\in\mathcal{N}$, where, letting $\{q'_\ell : \ell\in\mathcal{N}\}$ denote the same positive constants appearing in Corollary \ref{nonlinint1dell} and given $\{\alpha_\ell : \ell\in \mathcal{N}\}$ such that $\alpha_\ell>0$ for all $\ell\in\mathcal{N}$ and $\sum_{\ell\in\mathcal{N}}\alpha_\ell=1$, it holds that $\overline{q}'_\ell=\alpha_\ell q'_\ell$ for all $\ell\in\mathcal{N}$. Then, there exists a  positive constant $\delta=\delta(\{q_\ell: \ell\in\mathcal{N}\})$ such that  
\begin{equation*}
\int_I\Bigl[W(u)+\sum_{\ell=1}^kq_\ell|u^{(\ell)}|^2\Bigr]\,dt \geq \delta \int_I\Bigl[W(u)+\sum_{\ell=1}^k|u^{(\ell)}|^2\Bigr]\,dt
\end{equation*}
for every $u\in H^k(I)$.  
 \end{corollary}

\begin{proof}
    The proof is obtained by following the same line of the one of Corollary \ref{interpolationwithnorms}, with Corollary \ref{nonlinint1dell} used for every $\ell\in\{1,...,k-1\}$ in place of Theorem \ref{nonlinintd}.
\end{proof}

\begin{proposition} Let $k>1$ be an integer, assume that $q_\ell>-\overline{q}'_\ell$ for all $\ell\in\mathcal{N}$, with $\mathcal{N}$ and $\{\overline{q}'_\ell : \ell\in\mathcal{N}\}$ as in Corollary \ref{interpolationwithnorms1d}, and assume that {\rm (H1)} and {\rm (H2)} are satisfied. Then $m=\lim_{T\to+\infty}m(T)$.
\end{proposition}
\begin{proof}
    Clearly, $m(T)\geq m$ for every $T$; hence, $\lim_{T\to+\infty} m(T) \geq m$.

    To prove the converse, let $u\in H^k_{\text{loc}}(\R)$ be such that $\lim_{t\to\pm\infty} u(t)=\pm1$ and $\int_{-\infty}^{+\infty}[W(u)+\sum_{\ell=1}^kq_\ell|u^{(\ell)}|^2]\,dt<+\infty$, and let $T_o>0$ be such that $||u(t)|-1|<1/2$ if $|t|>T_o$.

    Consider a function $\phi\in C^{\infty}(\R;[0,1])$ that is supported in $(-\infty,1/2)$ and such that $\phi(t)=1$ for every $t<0$. For every $n\in\N$ we set
    \begin{equation*}
        v_n(t):=\phi(t-n)u(t)+1-\phi(t-n), 
    \end{equation*}
and we observe that $v_n$ coincides with $u$ on $(-\infty, n)$ and it is constantly equal to $1$ on $(n+1/2,+\infty)$. As a consequence 
\begin{align*}
\int_{-\infty}^{+\infty}\Bigl[W(v_n)+\sum_{\ell=1}^kq_\ell|v_n^{(\ell)}|^2\Bigr]\,dt & =  \int_{-\infty}^{n-\frac{1}{2}}\Bigl[W(u)+\sum_{\ell=1}^kq_\ell|u^{(\ell)}|^2\Bigr]\,dt \\
 & \quad +\int_{n-\frac{1}{2}}^{n+\frac{1}{2}}\Bigl[W(v_n)+\sum_{\ell=1}^kq_\ell|v_n^{(\ell)}|^2\Bigr]\,dt.
\end{align*}
We estimate the energy of the function $v_n$ on the interval $(n-1/2,n+1/2)$; that is, we estimate
\begin{equation}\label{energyvn}
    \int_{n-\frac{1}{2}}^{n+\frac{1}{2}}\Bigl[W(v_n)+\sum_{\ell=1}^kq_\ell|v_n^{(\ell)}|^2\Bigr]\,dt.
\end{equation}
First, we observe that $\min\{u(t),1\}\leq v_n(t)\leq \max\{u(t),1\}$ for all $t\in \R$ and $n\in\N$; therefore, since $u(t)\to 1$ as $t\to+\infty$, for every $\e>0$ there exists $T_\e>0$ such that $|v_n(t)-1|< \e$ for all $t>T_\e$ and $n\in\N$. Hence, the continuity of $W$ and (H$1$) imply that 
\begin{equation}\label{potentialestimate}
    \lim_{n\to+\infty}\int_{n-\frac{1}{2}}^{n+\frac{1}{2}}W(v_n)\,dt =0.
\end{equation}
As for the derivatives of $v_n$, we observe that for all $\ell\in\{1,...,k-1\}$, there exists a positive constant $C$ depending on $k,$ $\ell$, and the function $\phi$ such that
\begin{equation*}
    |v_n^{(\ell)}(t)|^2 \leq C \Bigl\{|u(t)-1|^2 +\sum_{j=1}^\ell |u^{(j)}(t)|^2\Bigr\} \qquad \text{ for all } t\in\R, 
\end{equation*}
therefore,
\begin{align*}
    \Bigl|\int_{n-\frac{1}{2}}^{n+\frac{1}{2}}\Bigl[\sum_{\ell=1}^kq_\ell|v_n^{(\ell)}|^2\Bigr]\,dt\Bigr| & \leq q_0 \sum_{\ell=1}^k \int_{n-\frac{1}{2}}^{n+\frac{1}{2}}\Bigl[ C \Bigl\{|u-1|^2 +\sum_{j=1}^\ell |u^{(\ell)}|^2\Bigr\}  \Bigr]\,dt\\
    & \leq q_0Ck \int_{n-\frac{1}{2}}^{n+\frac{1}{2}} \Bigl[W(u) + \sum_{\ell=1}^k|u^{(\ell)}|^2\Bigr]\,dt,
\end{align*}
where $q_0:=\max\{|q_1|, ...,|q_{k-1}|,1\}$, and the last inequality follows by (H$2$) assuming that $n-1>T_o$.

\noindent By Corollary \ref{interpolationwithnorms1d},  there exists a 
positive $\delta$ such that
\begin{equation*}
\int_{n-\frac{1}{2}}^{n+\frac{1}{2}}\Bigl[W(u)+\sum_{\ell=1}^kq_\ell|u^{(\ell)}|^2\Bigr]\,dt \geq \delta \int_{n-\frac{1}{2}}^{n+\frac{1}{2}}\Bigl[W(u)+\sum_{\ell=1}^k|u^{(\ell)}|^2\Bigr]\,dt
\end{equation*}
for every $n$; therefore  
\begin{equation}\label{derivativesestimate}
    \Bigl|\int_{n-\frac{1}{2}}^{n+\frac{1}{2}}\Bigl[\sum_{\ell=1}^kq_\ell|v_n^{(\ell)}|^2\Bigr]\,dt\Bigr| \leq \frac{q_0Ck}{\delta}\int_{n-\frac{1}{2}}^{n+\frac{1}{2}}\Bigl[W(u)+\sum_{\ell=1}^kq_\ell|u^{(\ell)}|^2\Bigr]\,dt,
\end{equation}
which tends to $0$ as $n\to+\infty$.

Gathering \eqref{potentialestimate} and \eqref{derivativesestimate}, we obtain that \eqref{energyvn} vanishes as $n\to+\infty$. With a similar reasoning, we obtain functions $\{w_n\}_n\subset H^k_{\text{loc}}(\R)$ which coincide with $u$ on $(-n,n)$, which are equal to the sign function outside $(-n-\frac{1}{2},n+\frac{1}{2})$, and such that
\begin{equation*}
\int_{-\infty}^{+\infty}\Bigl[W(w_n)+\sum_{\ell=1}^kq_\ell|w_n^{(\ell)}|^2\Bigr]\,dt  =  \int_{-n-\frac{1}{2}}^{n+\frac{1}{2}}\Bigl[W(u)+\sum_{\ell=1}^kq_\ell|u^{(\ell)}|^2\Bigr]\,dt +o_n(1).
\end{equation*}
Since each function $w_n$ is admissible for the minimum problem $m(n+1)$, letting $n\to+\infty$ the conclusion follows by the arbitrariness of $u$.
\end{proof}

Finally, we prove the existence of an optimal-profile. We omit some minor details in the proof as the arguments involved are similar to those that appear in the proofs of Proposition $4.1, 4.2$ in \cite{BDS} and Lemma $2.5$ in \cite{FM}.

\begin{proposition} Let $k>1$ be an integer, assume that $q_\ell>-\overline{q}'_\ell$ for all $\ell\in\mathcal{N}$, with $\mathcal{N}$ and $\{\overline{q}'_\ell : \ell\in\mathcal{N}\}$ as in Corollary \ref{interpolationwithnorms1d}, and assume that {\rm (H1)} and {\rm (H2)} are satisfied. Then, the infimum \eqref{optimalprofile} is a minimum.
\end{proposition}
\begin{proof}
Consider $\{u_n\}_n$ a minimizing sequence and, by translation invariance, assume that $u_n(0)=0$ for every $n\in \N$. Applying Corollary \ref{interpolationwithnorms1d} on each interval $(j,j+1),\, j\in \mathbb{Z}$ and summing over $j$, we have
\begin{equation}\label{finalremark-1}
    \int_{-\infty}^{+\infty}\Bigl[W(u_n)+\sum_{\ell=1}^kq_\ell|u_n^{(\ell)}|^2\Bigr]\,dt \geq \delta \int_{-\infty}^{+\infty}\Bigl[W(u_n)+\sum_{\ell=1}^k|u_n^{(\ell)}|^2\Bigr]\,dt
\end{equation}
for all $n\in\N$, and then
\begin{equation}\label{finalremark0}
    \sup_n \int_{-\infty}^{+\infty}\Bigl[W(u_n)+\sum_{\ell=1}^k|u_n^{(\ell)}|^2\Bigr]\,dt < +\infty.
\end{equation}
We apply the Fundamental Theorem of Calculus and H\"older's inequality to obtain
\begin{equation*}
|u_n(t)|\leq \int_0^t |u_n'(s)|\,ds \leq \sqrt{t}\|u'_n\|_{L^2(\R)}     \qquad \text{ for all } t\in \R, 
\end{equation*}
which, combined with \eqref{finalremark0}, implies that for every $T>0$ there exists a positive constant $C(T)$ such that 
$$\sup_n \|u_n\|_{H^k((-T,T))}\leq C(T).$$
Therefore, there exists $u\in H^k_{\text{loc}}(\R)$ such that, up to subsequences, $\{u_n^{(\ell)}\}_n$ converges locally uniformly on compact subsets of $\R$ to $u^{(\ell)}$ for all $\ell\in\{0,...,k-1\}$, with the convention that $v^{(0)}=v$, and $u_n^{(k)}\rightharpoonup u^{(k)}$ weakly in $L^2(\R)$. Note that, in particular, $u(0)=0$.

\noindent As for the lower semicontinuity of the energy, we first note that, as a consequence of Corollary \ref{interpolationwithnorms1d}, we have
\begin{equation}\label{positivity}
    \int_{-\infty}^{-T} \Bigl[W(u_n)+\sum_{\ell=1}^kq_\ell|u_n^{(\ell)}|^2\Bigr]\,dt + \int_{T}^{+\infty} \Bigl[W(u_n)+\sum_{\ell=1}^kq_\ell|u_n^{(\ell)}|^2\Bigr]\,dt  \geq 0
\end{equation}
for all $n\in \N$ and $T\in \N$. Then, we observe that Rellich's Theorem implies that $u_n^{(\ell)}\to u^{(\ell)}$ in $L^2((-T,T))$ as $n\to+\infty$ for every $\ell\in\{1,...,k-1\}$ and $T>0$, which in turn yields
\begin{equation}\label{strongconvergence}
    \lim_{n\to+\infty} \int_{-T}^{T}\Bigl[\sum_{\ell=1}^{k-1}q_\ell|u_n^{(\ell)}|^2\Bigr]\,dt = \int_{-T}^{T}\Bigl[\sum_{\ell=1}^{k-1}q_\ell|u^{(\ell)}|^2\Bigr]\,dt
\end{equation}
for every $T>0$; moreover, by Fatou's lemma and the lower semicontinuity of the $L^2$-norm, we have
\begin{equation}\label{weakconvergence}
    \liminf_{n\to+\infty}  \int_{-\infty}^{+\infty}\Bigl[W(u_n)+|u_n^{(k)}|^2\Bigr]\,dt \geq \int_{-\infty}^{+\infty}\Bigl[W(u)+|u^{(k)}|^2\Bigr]\,dt.
\end{equation}
Combining \eqref{positivity}, \eqref{strongconvergence}, and \eqref{weakconvergence}, and recalling that $q_k=1$, we obtain
\begin{align} \notag
    m=\liminf_{n\to+\infty} \int_{-\infty}^{+\infty} \Bigl[W(u_n)+\sum_{\ell=1}^kq_\ell|u_n^{(\ell)}|^2\Bigr]\,dt  & \geq \liminf_{n\to+\infty} \int_{-T}^{T} \Bigl[W(u_n)+\sum_{\ell=1}^kq_\ell|u_n^{(\ell)}|^2\Bigr]\,dt  \\ \label{finalremark-2}
    & \geq \int_{-T}^{T} \Bigl[W(u)+\sum_{\ell=1}^kq_\ell|u^{(\ell)}|^2\Bigr]\,dt
\end{align}
for all $T\in \N$.

\noindent Recall that $\{u_n\}_n$ is a minimizing sequence, hence, by \eqref{finalremark-1} and Fatou's Lemma, we infer
\begin{equation*}
     m = \liminf_{n\to+\infty}\int_{-\infty}^{+\infty}\Bigl[W(u_n)+\sum_{\ell=1}^kq_\ell|u_n^{(\ell)}|^2\Bigr]\,dt \geq \delta \int_{-\infty}^{+\infty}\Bigl[W(u)+\sum_{\ell=1}^k|u^{(\ell)}|^2\Bigr]\,dt;
\end{equation*}
therefore, we can pass to the limit as $T\to+\infty$ in \eqref{finalremark-2} by Dominated Convergence Theorem to obtain
\begin{equation*}
    m \geq \int_{-\infty}^{+\infty} \Bigl[W(u)+\sum_{\ell=1}^kq_\ell|u^{(\ell)}|^2\Bigr]\,dt.
\end{equation*}
If $\lim_{t\to\pm \infty} u(t)=\pm1$, the proof is completed; while, if $\lim_{t\to\pm \infty} u(t)=\mp1$, a minimizer is $u(-t)$. For this reason, in order to conclude, we prove that both $\lim_{t\to-\infty} u(t)$ and $\lim_{t\to+\infty} u(t)$ exist, and that $\{\lim_{t\to-\infty} u(t), \lim_{t\to+\infty} u(t) \} =\{-1,1\}$.     

As a first step we prove the existence of the limit at $+\infty$, the other case being analogous.  To see this, it is useful to observe that $u'\in H^1(\R)$, which implies 
\begin{equation} \label{finalremark1}
    \lim_{t\to\pm\infty} u'(t)=0.
\end{equation}
If $\lim_{t\to+\infty} u(t)$ exists, it equals $-1$ or $1$, otherwise $\int_{-\infty}^{+\infty} W(u)\,dt=+\infty$ which is a contradiction. We are left with considering the case there exist real numbers $a,b$, with $a<b$, and two increasing sequences $\{x_n\}_n, \{y_n\}_n$ such that 
\begin{equation*}
    \lim_{n\to+\infty}x_n=+\infty, \quad \lim_{n\to+\infty}y_n=+\infty, \quad \lim_{n\to+\infty} u(x_n) = a, \quad  \lim_{n\to+\infty} u(y_n) = b. 
\end{equation*}
Without loss of generality, we can suppose 
\begin{equation*}
   \quad x_n<y_n, \quad  u(x_n)\leq a+\frac{b-a}{8}, \quad u(y_n)\geq b-\frac{b-a}{8} \qquad \text{ for all } n\in \N;
\end{equation*}
moreover, by the continuity of $u$, we can also assume that both $-1$ and $1$ are not contained in the interval $(a,b)$.

\noindent If there exists $\{n_j\}_j$ such that $|x_{n_j}-y_{n_j}|\leq 2$ for every $j\in \N$, then, by Lagrange's Theorem, there exists $z_{n_j}\in (x_{n_j},y_{n_j})$ such that
\begin{equation*}
    |u'(z_{n_j})|\geq \frac{|u(y_{n_j})-u(x_{n_j})|}{2}\geq \frac{3}{8}(b-a), 
\end{equation*}
hence, $\limsup_{t\to+\infty} |u'(t)|\geq 3(b-a)/8$, which contradicts \eqref{finalremark1}. Otherwise, we may assume that $|x_n-y_n|\geq 2$ for every $n\in \N$. We set
\begin{equation*}
    \widetilde{x}_n:=\max\{t \in (x_n, y_n): u(t)\leq a+\frac{b-a}{4} \}
\end{equation*}
and 
\begin{equation*}
    \widetilde{y}_n:=\min\{t \in (x_n, y_n): u(t)\geq b-\frac{b-a}{4} \}.
\end{equation*}
If $|\widetilde{x}_n-\widetilde{y}_n|\leq1$ for infinitely many $n\in \N$, the conclusion follows by Lagrange's Theorem as before, otherwise, it holds that $|\widetilde{x}_n-\widetilde{y}_n|\geq1$ for all $n\geq N_o$ with some $N_o\in\N$ large enough. Recalling that the interval $(a,b)$ does not contain the points $-1,1$, the above observation and (H$1$) imply
\begin{equation*}
   \int_{-\infty}^{+\infty}W(u)\,dt\geq \sum _{n\geq N_o} \int_{\widetilde{x}_n}^{\widetilde{y}_n}W(u)\,dt \geq \sum _{n\geq N_o} \min\bigl\{ W(t) : t\in \bigl(a+\frac{b-a}{4} ,b-\frac{b-a}{4}\bigr)\bigr\}=+\infty, 
\end{equation*}
which is a contradiction, and concludes the proof of the existence of the limits at infinity.
We finally prove that they are different.

\noindent By contradiction suppose $\lim_{t\to\pm \infty} u(t)=-1$. By Lemma $3.2$ in \cite{BDS} and by the convergence $u_n\to u$ in $H^k_{\text{loc}}(\R)$, we find points $\{t_j\}_j$ and a subsequence of indices $\{n_j\}_j$ such that 
\begin{equation}\label{finalremark2}
 t_j\to+\infty, \quad    u_{n_j}(t_j)\to-1, \qquad u_{n_j}^{(\ell)}(t_j)\to 0 \quad \text{ as } j\to+\infty, \,\,\text{ for all } \ell\in\{1,...,k-1\}.
\end{equation}
Then we have
\begin{align*}
\int_{-\infty}^{+\infty}  \!\Bigl[W(u_{n_j})\!+\!\sum_{\ell=1}^kq_\ell|u_{n_j}^{(\ell)}|^2\Bigr]\,dt & = \int_{-\infty}^{t_j} \!\Bigl[W(u_{n_j})+\sum_{\ell=1}^kq_\ell|u_{n_j}^{(\ell)}|^2\Bigr]\,dt + \int_{t_j}^{+\infty} \!\Bigl[W(u_{n_j})\!+\!\sum_{\ell=1}^kq_\ell|u_{n_j}^{(\ell)}|^2\Bigr]dt \\
& \geq \int_{-\infty}^{t_j} \Bigl[W(u_{n_j})+\sum_{\ell=1}^kq_\ell|u_{n_j}^{(\ell)}|^2\Bigr]\,dt + m - \widetilde{m}(u_{n_j}(t_j),...,u^{(k-1)}_{n_j}(t_j)), 
\end{align*}
where we set
\begin{align*}
\widetilde{m}(u_0,...,u_{k-1}):= \inf  \Bigl\{\int_0^1 \Bigl[W(v) + \sum_{\ell=1}^kq_\ell |v^{(\ell)}|^2 \Bigr]\,dt :\, & v^{(\ell)}(0)=0 \text{ for all } \ell\in\{0,...,k-1\}, \\
& v^{(\ell)}(1)=u_\ell \text{ for all } \ell\in\{0,...,k-1\} \Bigr\},
\end{align*}
with the convention that $v^{(0)}=v$. Since $\widetilde{m}(u_0,...,u_{k-1})\to0$ as $(u_0,...,u_{k-1})\to0$, by \eqref{finalremark2} we infer that 
\begin{align*}
    m & = \liminf_{j\to+\infty} \int_{-\infty}^{+\infty}  \Bigl[W(u_{n_j})+\sum_{\ell=1}^kq_\ell|u_{n_j}^{(\ell)}|^2\Bigr]\,dt \\
    & \geq \liminf_{j\to+\infty} \int_{-\infty}^{t_j} \Bigl[W(u_{n_j})+\sum_{\ell=1}^kq_\ell|u_{n_j}^{(\ell)}|^2\Bigr]\,dt + m - \widetilde{m}(u_{n_j}(t_j),...,u^{(k-1)}_{n_j}(t_j)) \\
    & \geq \liminf_{j\to+\infty} \int_{-\infty}^{t_j} \Bigl[W(u_{n_j})+\sum_{\ell=1}^kq_\ell|u_{n_j}^{(\ell)}|^2\Bigr]\,dt + m,
\end{align*}
which implies 
\begin{equation*}
    \liminf_{j\to+\infty} \int_{-\infty}^{t_j} \Bigl[W(u_{n_j})+\sum_{\ell=1}^kq_\ell|u_{n_j}^{(\ell)}|^2\Bigr]\,dt \leq 0.
\end{equation*}
But, by Corollary \eqref{interpolationwithnorms1d} and Fatou's Lemma, it holds that
\begin{align*}
   \liminf_{j\to+\infty} \int_{-\infty}^{t_j} \Bigl[W(u_{n_j})+\sum_{\ell=1}^kq_\ell|u_{n_j}^{(\ell)}|^2\Bigr]\,dt & \geq
   \delta\liminf_{j\to+\infty} \int_{-\infty}^{t_j} \Bigl[W(u_{n_j})+\sum_{\ell=1}^k|u_{n_j}^{(\ell)}|^2\Bigr]\,dt \\
   & \geq \delta\int_{-\infty}^{+\infty} \Bigl[W(u)+\sum_{\ell=1}^k|u^{(\ell)}|^2\Bigr]\,dt \geq   0;
\end{align*}
therefore, we deduce that 
that $u$ is identically equal to $-1$. This contradicts $u(0)=0$, which concludes the proof.
\end{proof}

\begin{remark} {\rm Restoring the general dimension $d$, we may deal with the case of special norms on tensors that are `compatible with slicing' upon supposing the non-negativity of the coefficients $q_1,...,q_{k-1}$. More precisely, we suppose that for every $\ell\in\{1,...,k\}$ the norm $|\cdot|_\ell$ satisfies
\begin{equation}\label{compatibleslicing}
|T|_\ell \geq |T(\xi,...,\xi)| \qquad \text{ for every } \xi\in\Sd,
\end{equation}
and we assume that $q_\ell\geq 0$ for every $\ell\in\{1,...,k-1\}$. Also in this instance, the density energy $g$ defined in \eqref{density} is the constant $m$ given by \eqref{optimalprofile}.

The argument is similar to the one in the proof of Proposition \ref{positivedensity}. For fixed $\nu\in \Sd$ and $\eta>0$, let $v$ and $\e$ be such that
\begin{equation*}
        g(\nu)+\eta \geq \int_{{Q_1^\nu}} \Bigl[\frac{1}{\e}W(v)+\sum_{\ell=1}^{k}q_\ell\e^{2\ell-1} |\nabla^{(\ell)}v|_\ell^2\Bigr]\,dx.
    \end{equation*}
Applying Fubini's Theorem, \eqref{compatibleslicing}, and the slicing properties of Sobolev functions, we obtain
\begin{align*}
         g(\nu) +\eta & \geq \int_{{Q_1^\nu}} \Bigl[\frac{1}{\e}W(v)+\sum_{\ell=1}^{k}q_\ell \e^{2\ell-1}|\nabla^{(\ell)}v|_\ell^2\Bigr]\,dx \\
         & \geq \int_{Q_1^\nu \cap \nu^\perp}\int_{-\frac{1}{2}}^{\frac{1}{2}} \Bigl[\frac{1}{\e}W(v^{\nu,y}(t))+ \sum_{\ell=1}^{k}q_\ell\e^{2\ell-1}|(v^{\nu,y})^{(\ell)}(t)|^2\Bigr]\,dt\,d\Hd(y) \\
        & \geq \inf\Big\{\int_{-\frac{1}{2\e}}^{\frac{1}{2\e}} \Bigl[W(v)+\sum_{\ell=1}^{k}q_\ell |v^{(\ell)}|^2\Bigr]\,dx,  \ v \in  H^k\bigl(\bigl(-\frac{1}{2\e},\frac{1}{2\e}\bigr)\bigr), \\
        & \qquad \qquad \qquad \qquad \qquad \qquad \quad  v(t)=\text{sgn}(t) \text{ if } |t|>M \text{ for some } M\in\bigl(0,\frac{1}{2\e}\bigr)\Big\} \\
        & =m\Bigl(\frac{1}{2\e}\Bigr),
        \end{align*}
        where the second inequality also follows by the assumption $q_\ell\geq 0$ for every $\ell\in\{1,...,k-1\}$. By the arbitrariness of $\eta$ and $\e$, the previous chain of inequalities implies $g\geq m$. The converse inequality is obtained using as a test function for $g(\nu)$ the function $v(x):=u(\frac{x\cdot\nu}{\e})$, with $u$ an (almost) optimal one-dimensional profile for $m(1/(2\e))$.}
\end{remark} 

\section*{Acknowledgements}

The authors wish to thank Prof. \!\!Andrea Braides for having proposed the problem during his course {\em \lq\lq Singular perturbations in fractional Sobolev spaces\rq\rq} held at SISSA in the Fall session of 2024. The authors are members of Gruppo Nazionale per l'Analisi Matematica, la Probabilità e le loro Applicazioni (GNAMPA) of Istituto Nazionale di Alta Matematica (INdAM).

\bigskip

{\noindent \textbf{Competing Interests statement:} The authors declare that they have not financial or non-financial interests related to the work submitted for publication.}

{\noindent \textbf{Data Availability statement:} Data sharing not applicable to this article since no datasets were generated or analysed during the current study.}

\end{document}